\documentclass[final]{siamart190516}
\usepackage{amsfonts,amsmath,amssymb,graphicx,color}
\usepackage{psfrag}
\usepackage{pdflscape}
\usepackage{multirow}
\usepackage{blindtext} 
\usepackage{enumitem}
\graphicspath{ {./figures/} }
\usepackage{hyperref}
\usepackage{soul}
\hypersetup{colorlinks=true}
\usepackage{algpseudocode}
\usepackage{epstopdf}
\usepackage{framed}
\usepackage{footnote}
\usepackage{pifont}
\usepackage{color}
\usepackage{todonotes}
\usepackage{amssymb}
\usepackage{pifont}

\usepackage{cancel}
\usepackage{comment}
\usepackage[normalem]{ulem}

\usepackage{yuchenmacro}

\numberwithin{figure}{section}
\numberwithin{table}{section}

\def\E{\mathbb{E}}

\newcommand{\bequ}{\begin{equation}}     \newcommand{\eequ}{\end{equation}}
\newcommand{\benn}{\begin{equation*}}    \newcommand{\eenn}{\end{equation*}}
\newcommand{\bbma}{\begin{bmatrix}}      \newcommand{\ebma}{\end{bmatrix}}

\newcommand{\R}{\mathbb{R}}

\newcommand{\jn}[1]{{\color{blue}#1}}

\newcommand{\rb}[1]{{\color{purple}#1}}
\newcommand{\vb}[1]{{\color{cyan}#1}}

\newtheorem{thm}{Theorem}[section]

\newtheorem{lem}[thm]{Lemma}
\newtheorem{cor}[thm]{Corollary}
\newtheorem{assum}[thm]{Assumptions}
\numberwithin{equation}{section}

\newcommand{\beq}{\begin{equation}}
\newcommand{\eeq}{\end{equation}}

\newcommand{\bproof}{\begin{description} \item[{\it Proof}.] ~ }
	\newcommand{\eproof}{\hspace*{\fill}$\Box$\medskip \end{description}}

\newcounter{algo}[section]

		\newcounter{prog}[section]
		
\newcommand{\cmnt}[1]{\quad \text{(#1)}}
\title{ Constrained and Composite Optimization via Adaptive Sampling Methods}
\author{       
        Yuchen Xie\thanks{Department of Industrial Engineering and Management Sciences, Northwestern University, 
       Evanston, IL, USA.  This author was supported by the Office of Naval Research grant N00014-14-1-0313 P00003, and by National Science Foundation grant DMS-1620022.}
       \and  
       Raghu Bollapragada\thanks{Department of Mechanical Engineering, University of Texas, Austin, USA.}
       \and   
       Richard Byrd \thanks{Department of Computer Science, University of Colorado,
        Boulder, CO, USA.  This author was supported by National Science Foundation grant DMS-1620070.}
        \and
       Jorge Nocedal \thanks{Department of Industrial Engineering and Management Sciences, Northwestern University, 
       Evanston, IL, USA.  This author was supported by the Office of Naval Research grant N00014-14-1-0313 P00003, and National Science Foundation grant DMS-1620022.} 
      }

\date{\today}
\begin{document}

 \maketitle

 \begin{abstract}
 The motivation for this paper stems from the desire to develop an adaptive sampling method for solving constrained optimization problems in which the objective function is stochastic and the constraints are deterministic. The method proposed in this paper 
 is a proximal gradient method that can also be applied to the composite optimization problem  min $f(x) + h(x)$, where $f$ is stochastic and $h$ is convex (but not necessarily differentiable). Adaptive sampling methods employ a  mechanism for gradually improving the quality of the gradient approximation so as to keep computational cost to a minimum. The mechanism commonly  employed in unconstrained optimization is no longer reliable in the constrained or composite optimization settings because it is based on pointwise decisions that cannot correctly predict the quality of the proximal gradient step. The method proposed in this paper measures the result of a complete step to determine if the gradient approximation is accurate enough; otherwise a more accurate gradient is generated and a new step is computed. Convergence results are established both for strongly convex and general convex $f$. Numerical experiments are presented to illustrate the practical behavior of the method.

 \end{abstract}

\section{Introduction}
In this paper, we study the solution of constrained and composite optimization problems in which the objective function is stochastic and the constraints or regularizers are deterministic. We propose  methods that automatically adjust the quality of the gradient estimate so as to keep computational cost at a minimum while ensuring a fast rate of convergence. Methods of this kind have been studied in the context of unconstrained optimization but their extension to the constrained and composite optimization settings  is not simple because the projections or proximal operators used in the methods introduce discontinuities. This renders existing rules for the control of the gradient unreliable. Whereas in the unconstrained setting pointwise decisions suffice to estimate the quality of a gradient approximation, in the presence of constraints or nonsmooth regularizers one must analyze the result of a complete step. 

Let us begin by considering the  optimization problem
\begin{equation}    \label{conprob}
	\min_{x \in \mathbb{R}^n}  f(x) \quad\mbox{s.t.} \ \ x \in \Omega,
\end{equation}
where $f:\mathbb{R}^n \rightarrow \mathbb{R}$ is a stochastic objective function and $\Omega$ is a deterministic convex set. Automatic rules for controlling the quality of the gradient  when $\Omega= \mathbb{R}^n$ have been studied from a theoretical perspective and  have been successfully applied  to  expected risk minimization problems arising in machine learning. Since in that context the  gradient approximation is controlled by the sample size, these methods have been called ``adaptive sampling'' methods. 
A fundamental mechanism  for controlling the quality of the gradient in the unconstrained setting  is the \textit{norm test} \cite{byrd2012sample}, which lies behind most algorithms and theory of adaptive sampling methods. 

To describe this test, let $\Omega = \mathbb{R}^n$, and consider the iteration 
 \begin{equation}  \label{iteration}
 x_{k+1}= x_k -\alpha_k g_k,
 \end{equation}
where $\alpha_k >0$ is a steplength and $g_k$ is an approximation to the  gradient $\nabla f(x_k)$. To determine if $g_k$ is sufficiently accurate to ensure that iteration \eqref{iteration} is convergent, one can test the inequality  \cite{byrd2012sample}: 
\begin{equation}\label{normt}
    \mathbb{E}[  \| g_k - \nabla f(x_k)\|_2^2] \leq  \xi \| \nabla f(x_k)\|_2^2,  \quad 
    \xi >0,
\end{equation}
where the expectation is taken with respect to the choice of $g_k$ at iteration $k$.  If \eqref{normt} is  satisfied, $g_k$ is deemed accurate enough; otherwise a new and more accurate gradient approximation is computed.  We refer to this procedure as the norm test to distinguish it from tests based on angles \cite{bollapragada2018adaptive}.

The norm test is, however, not adequate in the constrained setting. To see this, suppose that we apply the gradient projection method, $ x_{k+1}= P_\Omega[x_k -\alpha_k g_k] $,  to solve problem \eqref{conprob} when $\Omega \neq \mathbb{R}^n$.   A  condition such as \eqref{normt} on the quality of the gradient approximation at one point cannot always predict the quality of the full step because the latter is based on a projection of the gradient, which may be much smaller.
This is illustrated in Figure \ref{fig:failure norm test}, where we consider the minimization of a strongly convex quadratic function subject to a linear constraint:
\begin{equation*}
  \begin{aligned}
    \min_{x \in \mathbb{R}^n} & ~\tfrac{1}{2} x^T Q x + b^T x + c \qquad
    \text{s.t.}  ~ a^T x \leq 0.
  \end{aligned}
\end{equation*}
In Figure \ref{fig:failure norm test}, $\widehat{x}^*$ denotes the unconstrained minimizer and $x^{ *}$  the solution of the constrained problem.  We let the iterate $x_k$ lie on the boundary of the constraint,  very close to the solution $x^*$, and  observe that $\|\nabla f(x_k)\|$ is large, and stays large as $x_k$ approaches $x^*$.
Thus,  \eqref{normt} does not force the error in $g_k$ to zero
as $x_k \rightarrow x^*.$ 

The instance of $g_k$ shown in Figure \ref{fig:failure norm test}  satisfies $\| g_k - \nabla f(x_k)\| < \| \nabla f(x_k) \|$, but results in a poor step. Clearly satisfaction of \eqref{normt} allows for many such steps. Note, however, that $\| g_k - \nabla f(x_k)\| $ is greater than the norm of the projected gradient $P_\Omega[g_k]$, which is a more appropriate measure.   
\begin{figure}[htp]
  \centering
  \includegraphics[width=0.5\linewidth]{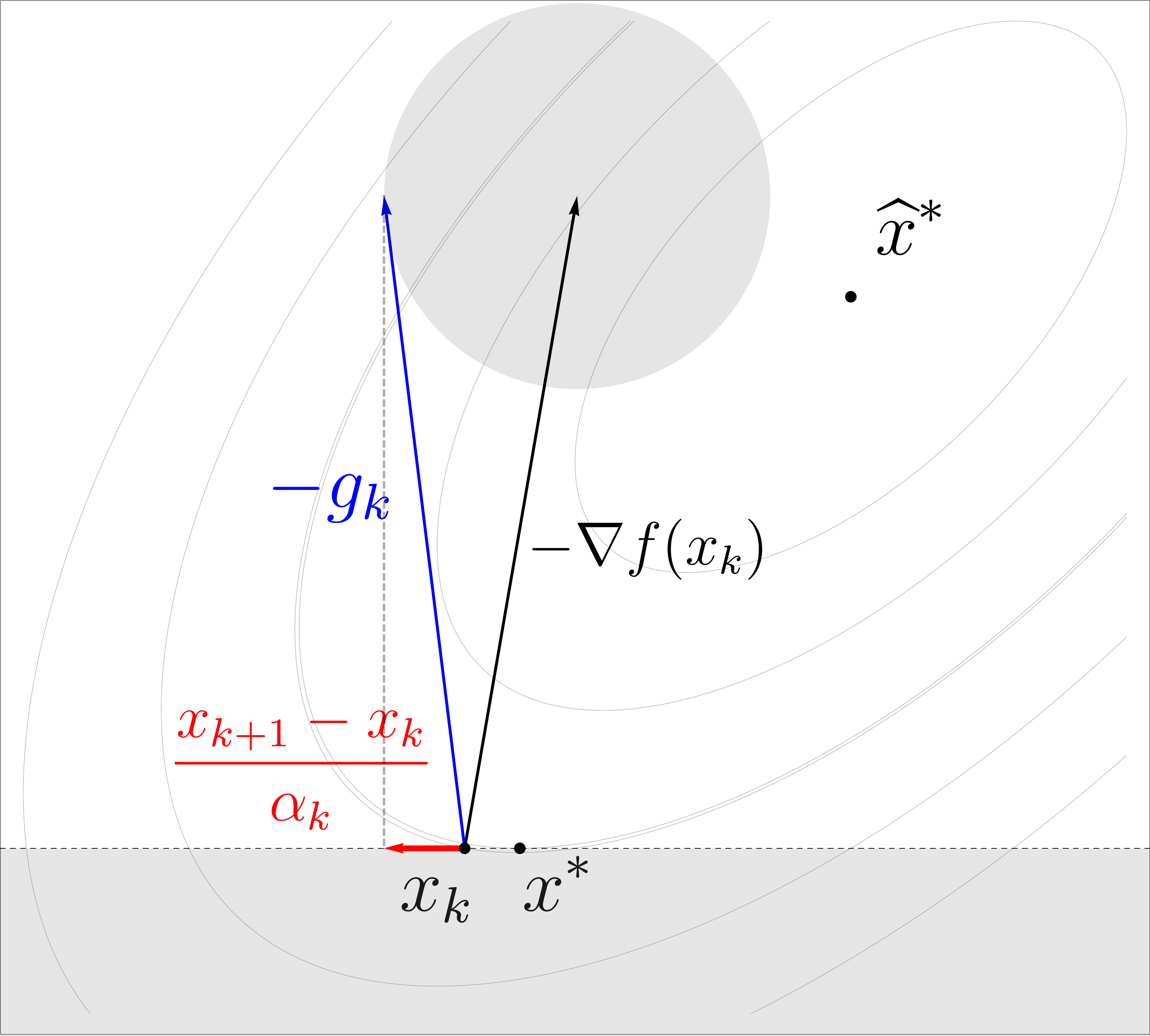}
  \caption{Failure of norm test for constrained problems. }
  \label{fig:failure norm test}
\end{figure}
Thus, since we  are  concerned about the error in the total step, and in this example the step is given by $x_{k+1} - x_k= -\alpha_k P_\Omega[g_k]$,  it makes sense to compare $\| g_k - \nabla f(x_k)\|$ to 
$\|P_\Omega [g_k]\| = \|x_{k+1} -x_k\| /\alpha_k$.

We generalize this idea and propose the following procedure for measuring the quality of the gradient approximation. We first compute a projected step $\overline{x}_{ k+1}= P_\Omega[x_k -\alpha_k g_k]$ based on the current gradient estimate $g_k$, and regard $g_k$ to be acceptable if the following inequality holds:
\begin{align}   \label{newtest}
	 \mathbb{E}[  \| g_k - \nabla f(x_k)\|_2^2] \equiv &\, \varr{k}{g_k} \leq \xi {\nrm{2}{\frac{\EE{k}{\overline x_{k+1}} - x_k}{\alpha_k}}^2} , \quad{\xi >0}.
\end{align}
Otherwise, we compute a more accurate gradient estimate $g_k$, and recompute the step to obtain the new iterate $x_{k+1}$.

In this strategy one must therefore look ahead, suggesting that a convenient framework for the design and analysis of adaptive sampling methods for constrained optimization is the \emph{proximal gradient method}. In addition to its versatility, the proximal gradient method allows us to expand the range of our investigation to include the composite optimization problem
\begin{equation}    \label{proxprob}
	\min_{x \in \mathbb{R}^n} ~ \phi(x) = f(x) + h(x),
\end{equation}
where $f$ is a stochastic function and $h$ is a convex (but not necessarily smooth or finite-valued) function. The constrained optimization problem \eqref{conprob} can be written in the form \eqref{proxprob} by defining  $h$ to be the convex indicator function for the set $\Omega$. 

The  goal of this paper is to design an adaptive mechanism for gradually improving the gradient accuracy that can be regarded as an extension of the norm test \eqref{normt} to  problems \eqref{conprob} and \eqref{proxprob}. We argue in Section~\ref{derivation} that the condition \eqref{newtest}, with $\phi$ replacing $f$, can be used to build such a mechanism within a proximal gradient method. Although  condition \eqref{newtest} appears to be impractical since it involves $\EE{k}{\overline x_{k+1}}$, we show how to approximate it in practice. The proposed algorithm reacts to information observed during the course of the iteration, as opposed to methods that dictate the increase in the gradient size \emph{a priori}. Specifically, it  has been established in \cite{jalilzadeh2018optimal} that for a stochastic proximal gradient method in which the sample size grows  like $a^k$, with $a>1$,  convergence can be assured in the convex case. However, the behavior of the algorithm depends very strongly on the value of $a$, and there are no clear guidelines on how to choose it for a given problem.

\smallskip\noindent{\em N.B.} As this paper was being readied for publication, we became aware that \cite{beiser2020adaptive}, which deals with a similar subject, had just been posted. The two papers differ, however, in various ways in their treatment of the topic.

\subsection{Literature Review}
A deterministic version of the norm test was used by Carter \cite{carter} in the design of a trust region method for unconstrained optimization that employs inexact gradients. Friedlander and Schmidt \cite{friedlander2012hybrid} propose {increasing} the sample size geometrically for the solution of the finite-sum problem, establish a linear convergence result, and report numerical tests with a quasi-Newton method. Byrd et al. \cite{byrd2012sample} studied the expected risk minimization problem and provide a complexity result for the geometric growth condition. That paper also introduces the stochastic version of the norm test \eqref{normt}, and reports results with a Newton-like method. Bollapragada et al. 
\cite{bollapragada2017adaptive} introduced a variant of the norm test, called the {the inner product {test}, which is designed to improve the practical efficiency of the method at the price of weakening the theoretical convergence guarantees.  (An adaption of this test to problems \eqref{conprob} and \eqref{proxprob} is presented in Section \ref{sec:ip}.)
Adaptive sampling methods have also been studied by Cartis and Scheinberg \cite{cartis2015global}, who establish a global rate of convergence of unconstrained optimization methods that (implicitly) satisfy the norm condition. Pasupathy et al. \cite{2014pasglyetal} study sampling rates in stochastic recursions. Roosta et al. \cite{roosta2016sub1,roosta2016sub} analyze sub-sampled Newton methods with adaptive sampling, and De et al. \cite{de2017automated} study automatic inference with adaptive sampling.

There is a large literature on proximal gradient methods for solving composite optimization problems; see e.g. \cite{bertsekas2015convex} and the references therein. Some of these studies consider inexact gradients \cite{NIPS2011_8f7d807e},  but  these studies do not propose an automatic procedure for improving the quality of the gradient.  \footnote{An exception is \cite{beiser2020adaptive}, which as  mentioned above, was released very shortly before this paper was posted.}

\section{Outline of the Algorithm}
Since  the constrained optimization problem \eqref{conprob} is  a special case of the composite problem \eqref{proxprob}, we focus on the latter and state the problem under consideration as 
\begin{equation}
	\min_{x \in \mathbb{R}^n} ~ \phi(x) = f(x) + h(x), \quad\mbox{where} \  f(x) = \EE{\theta \sim \Theta}{{F(x, \theta)}} .  \label{probex}
\end{equation}
 Here, $F(\cdot, \theta): \mathbb{R}^n \rightarrow \mathbb{R}$ is a smooth function,  $\theta$ is a random variable with support $\Theta$, and $h: \mathbb{R}^n \rightarrow \mathbb{R} \cup \{\infty\}$ is a convex (and generally nonsmooth) function. A popular method for solving composite optimization problems is the proximal gradient method (see e.g. \cite{bertsekas2003convex}), which in the context of problem \eqref{probex} is given as
\begin{align}  
	x_{k+1} & \gets \argmin{x \in \mathbb{R}^n}~ f(x_k) + g_k^T (x - x_k) + \frac{1}{2\alpha_k}  \|x-x_k\|^2 + h(x),  \ \  \mbox{with} \ \ 0 < \alpha_k \leq \tfrac{1}{L},
	\label{proxmethod}
\end{align}
where $g_k$ is an unbiased estimator of $\nabla f(x_k)$ and $L$ is a Lipschitz constant defined below. Here and henceforth, $\| \cdot \|$ denotes the Euclidean norm. As is well known, we can also write this iteration as
\begin{equation}
	{x}_{k+1} = \text{prox}_{\alpha_k h}\bpa{x_k - \alpha_k {g}_k} ,
\end{equation} 
where
\begin{align}
	 \text{prox}_{\alpha_k h}(z_k)  = \underset{x \in \mb{R}^d}{\text{argmin}} ~ h(x) + \frac{1}{2\alpha_k}  \nrm{}{x-z_k}^2 .
\end{align}

The proposed adaptive sampling proximal gradient algorithm proceeds in two stages. At a given iterate $x_k$, it first computes a gradient approximation (using the current sample size)  as well as a proximal gradient step. Based on 
information gathered from this step, it computes a second proximal gradient step that determines the new iterate $x_{k+1}$. An outline of this method is given in Algorithm~1.

\newpage
\begin{framed}
\noindent \textbf{Algorithm 1: Outline of {Adaptive Sampling} Algorithm for Solving Problem \eqref{probex}}:

\smallskip
\noindent \textbf{Input}: $x_0$,  sample size $S \in \mathbb{N}^+$, and sequence $\{\alpha_k >0\}$.

\smallskip
\noindent  \textbf{For k=1, \dots}:

 1. Draw $S$  i.i.d. samples $\bst{\theta_0, \theta_1, \cdots, \theta_{S-1}}$ from $\Theta$,   compute 
\begin{equation}  \label{barg}
	\overline{g}_k = \frac{1}{S} \sum_{i=0}^{S-1} ~ \nabla_x F(x, \theta_i),
\end{equation}
~~~~~~~~and the proximal gradient step
\begin{equation}
	\label{x_estimate}
	\overline{x}_{k+1} = \text{prox}_{\alpha_k h}\bpa{x_k - \alpha_k \overline{g}_k} .
\end{equation}
~~~~~2. Determine the new sample size $S_k \geq S$ (see the next section).

 3. \textbf{If} $S_k > S$

~~~~ re-sample $S_k$ i.i.d. samples $\bst{\wh \theta_0, \wh \theta_1, \cdots, \wh \theta_{S_k}-1}$ from $\Theta$, and compute:
\begin{equation}
 g_k=   \frac{1}{S_k} \sum_{i=0}^{{S_k}-1}~ \nabla_x F(x, \wh \theta_i)
 \end{equation}
\begin{equation}
	x_{k+1} = \text{prox}_{\alpha_k h}\bpa{x_k - \alpha_k g_k}
\end{equation}

~~~~\textbf{Else}
      \[x_{k+1}= \overline{x}_{k+1}.\]

4. Set $S \leftarrow S_k$

\noindent  \textbf{End For}
\end{framed}

\medskip
As discussed in Section~\ref{sec:alg}, when $S_k >S$, one can reuse the samples from Step~1, and in Step~3 only gather $(S_k-S)$ additional i.i.d. samples  $\bst{\wh \theta_{S}, \wh \theta_{S+1}, \cdots, \wh \theta_{S_k-1}}$  from $\Theta$. 

The unspecified parts of this algorithm are the steplength sequence $\{\alpha_k\}$ and  the determination of a sample size $S_k$ in Step~2. The analysis in the next section provides the elements for making those decisions.  One requirement of the strategy used in Step~2-3 is that, when $h$ is not present, Algorithm~1 should reduce to the iteration \eqref{iteration}-\eqref{normt}, i.e., to an adaptive sampling gradient method using the norm test to control the sample size.

\section{Derivation of the  Algorithm}    \label{derivation}
To motivate our approach for determining the sample size $S_k$, we begin by deriving a fundamental condition (see  \eqref{eq_test} below) that ensures that the steps are good enough to ensure convergence in expectation.  The rest of the derivation of the algorithm consist of devising a procedure for approximating condition \eqref{eq_test} in practice.

\subsection{A Fundamental Inequality}

We recall that a function $f: \mathbb{R}^n \to \mathbb{R}$ is $\mathbb{\mu}$-strongly convex (with $\mu > 0$) iff
\begin{equation} \label{strong}
  f({\gamma} x + (1-{\gamma})y) \leq {\gamma} f(x) + (1-{\gamma}) f(y) - \frac{\mu}{2}{\gamma} (1-{\gamma}) \nrm{}{y-x}^2 \cm \forall x, y \in \mathbb{R}^n \cm \forall \gamma \in [0, 1].
  \end{equation}
We also have that if $f$ is a strongly convex and differentiable function, then
\begin{equation}  \label{strong2}
 f(x) \geq f(y ) + \nabla f(y)^T(x-y) + \frac{\mu}{2} \| x -y \|^2 \cm \forall x \in \mathbb{R}^n.
 \end{equation}
If a function is continuously differentiable, $\mu$-strongly convex, and has a Lipschitz continuous gradient with Lipschitz constant $L$, we  say that $f$ is $[\mu,L]$--smooth.
We make the following assumptions about problem \eqref{probex}.

\begin{assum}\label{assume1} $f: \mathbb{R}^n \to \mathbb{R}$ is a $[\mu, L]$-smooth function and {$h: \mathbb{R}^n \to {\mathbb{R}} \cup \{\infty\}$}
is a closed convex and proper function.  
\end{assum} 

These assumptions imply that the objective function $\phi$ defined in \eqref{probex} is strongly convex,  and we denote its minimizer by $x^\ast$ {and the minimum objective value by $\phi^*$}.
We consider  the stochastic proximal gradient method \eqref{proxmethod}
where $g_k$ is an unbiased estimator of $\nabla f(x_k)$ adapted to the filtration $\mathbb{T}$  generated as
$	\mathbb{T}_k = \sigma \bpa{x_0, g_0, g_1, \cdots, g_{k-1}}.
$
In other words, we assume that 
\begin{equation} \label{unbiased} \E\bpa{g_k|\mathbb{T}_k} = \nabla f(x_k).
\end{equation}
 For simplicity,  we denote conditional expectation as $\EE{k}{\cdot} = \E\bpa{\cdot|\mathbb{T}_k}$ and conditional variance as $\varr{k}{\cdot} = \E \bbr{\nrm{}{\cdot}^2|\mathbb{T}_k} - \nrm{}{\E \bpa{\cdot|\mathbb{T}_k}}^2$. 
In what follows, we let $f_k, \nabla f_k$ denote $f(x_k), \nabla f(x_k)$, and similarly for other functions. We begin by establishing a technical lemma that provides the first stepping stone in our analysis.

\begin{lem} \label{tech} Suppose that Assumptions~\ref{assume1} hold and that $\{x_k\}$ is generated by iteration \eqref{proxmethod}, where $g_k$ satisfies \eqref{unbiased}. Then,
\begin{align*}
	\EE{k}{\phi_{k+1} - \phi^*} & \leq (1-\mu \alpha_k)(\phi_k - \phi^*) + \EE{k}{(\nabla f_k - g_k)^T (x_{k+1} - x_k)} \\
	& - \bpa{\frac{1}{2\alpha_k} - \frac{L}{2}}\EE{k}{\nrm{}{x_{k+1} - x_k}^2}.
\end{align*}
\end{lem}
\begin{proof}
By Assumptions~\ref{assume1},  we have that for any fixed $x_k \in \mathbb{R}^n$, 
\begin{align*}
	 \phi_{k+1}
	\leq ~& f_k + \nabla f_k^T(x_{k+1} - x_k) + \frac{L}{2} \nrm{}{x_{k+1} - x_k}^2 + h_{k+1} \\
	= ~& f_k + g_k^T(x_{k+1} - x_k) + \frac{1}{2\alpha_k} \nrm{}{x_{k+1} - x_k}^2 + h_{k+1}+ (\nabla f_k - g_k)^T (x_{k+1} - x_k)  \\
	& - \bpa{\frac{1}{2\alpha_k} - \frac{L}{2}}\nrm{}{x_{k+1} - x_k}^2 \\
	\leq ~& f_k + g_k^T(x-x_k) + \frac{1}{2\alpha_k} \nrm{}{x-x_k}^2 +  h(x) +  (\nabla f_k - g_k)^T (x_{k+1} - x_k) \\
	& - \bpa{\frac{1}{2\alpha_k} - \frac{L}{2}}\nrm{}{x_{k+1} - x_k}^2 \cmnt{for any $x \in \mathbb{R}^n$ by definition \eqref{proxmethod} of $x_{k+1}$} \\
	= ~& f_k + \nabla f_k^T(x-x_k) + \frac{1}{2\alpha_k} \nrm{}{x-x_k}^2 +  h(x) +  (\nabla f_k - g_k)^T (x_{k+1} - x_k) \\
	& + (g_k - \nabla f_k)^T(x-x_k) - \bpa{\frac{1}{2\alpha_k} - \frac{L}{2}}\nrm{}{x_{k+1} - x_k}^2 \\
	\leq ~& \phi(x) + \bpa{\frac{1}{2\alpha_k} - \frac{\mu}{2}} \nrm{}{x-x_k}^2 + (\nabla f_k - g_k)^T (x_{k+1} - x_k)   \\
	& + (g_k - \nabla f_k)^T(x-x_k) - \bpa{\frac{1}{2\alpha_k} - \frac{L}{2}}\nrm{}{x_{k+1} - x_k}^2  \cmnt{by \eqref{strong2}}.
\end{align*}
\normalsize
This inequality holds for any $x \in \mathbb{R}^n$. Let us substitute
\begin{equation}   \label{tilde}
x \leftarrow \wt x_k = \beta x^* + (1-\beta) x_k,  \quad\mbox{with}\ \ \beta = \mu \alpha_k,
\end{equation} in the relation above. Recalling the definition \eqref{strong} of strong convexity, we obtain
\begin{align*}
\phi_{k+1} 	\leq ~& \phi(\wt x_k) + \bpa{\frac{1}{2\alpha_k} - \frac{\mu}{2}} \nrm{}{\wt x_k-x_k}^2 + (\nabla f_k - g_k)^T (x_{k+1} - x_k) \\
	& + (g_k - \nabla f_k)^T(\wt x_k-x_k) - \bpa{\frac{1}{2\alpha_k} - \frac{L}{2}}\nrm{}{x_{k+1} - x_k}^2  \\
	\leq ~&  \beta \phi^* + (1-\beta) \phi_k \underbrace{-\frac{\mu}{2} \beta(1-\beta) \|x^*-x_k\|^2 + \bpa{\frac{1}{2\alpha_k} - \frac{\mu}{2}} \nrm{}{\wt x_k-x_k}^2}_\text{term 1} \\
	 &  + (\nabla f_k - g_k)^T (x_{k+1} - x_k) 
	 + (g_k - \nabla f_k)^T(\wt x_k-x_k) - \bpa{\frac{1}{2\alpha_k} - \frac{L}{2}}\nrm{}{x_{k+1} - x_k}^2  .
\end{align*}
Term 1 can be written as
\begin{align}
& -\frac{\mu}{2} \beta(1-\beta) \|x^*-x_k\|^2 + \bpa{\frac{1}{2\alpha_k} - \frac{\mu}{2}}\beta^2 \nrm{}{x^*-x_k}^2 \nonumber \\
 = & -\frac{\mu}{2} \beta(1-\beta) \|x^*-x_k\|^2 + \bpa{\frac{1 -\mu \alpha_k}{2\alpha_k}} \beta^2 \nrm{}{x^*-x_k}^2 \nonumber \\
 = & \ \beta(1-\beta) \|x^*-x_k\|^2  \left(\frac{\beta}{ 2\alpha_k}- \frac{\mu}{2} \right) \nonumber \\
 = & \ 0 \label{really} 
\end{align}
since $\beta = \mu \alpha_k$.
Therefore,
\begin{align*}
	\phi_{k+1} \leq ~& \beta \phi^* + (1-\beta) \phi_k + (\nabla f_k - g_k)^T (x_{k+1} - x_k) + (g_k - \nabla f_k)^T(\wt x_k -x_k) \\
	& - \bpa{\frac{1}{2\alpha_k} - \frac{L}{2}}\nrm{}{x_{k+1} - x_k}^2 .
\end{align*}
%
Taking conditional expectation, noting that $\wt x_k \in \mathbb{T}_k$, and recalling \eqref{unbiased}, we have 
\begin{align*}
	\EE{k}{\phi_{k+1}} & \leq \beta \phi^* + (1-\beta) \phi_k + \EE{k}{(\nabla f_k - g_k)^T (x_{k+1} - x_k)} \\
	& - \bpa{\frac{1}{2\alpha_k} - \frac{L}{2}}\EE{k}{\nrm{}{x_{k+1} - x_k}^2},
\end{align*}
and by the definition of $\beta$ we conclude that
\begin{align*}
	\EE{k}{\phi_{k+1} - \phi^*} & \leq (1-\mu \alpha_k)(\phi_k - \phi^*) + \EE{k}{(\nabla f_k - g_k)^T (x_{k+1} - x_k)} \\
	& - \bpa{\frac{1}{2\alpha_k} - \frac{L}{2}}\EE{k}{\nrm{}{x_{k+1} - x_k}^2}.
\end{align*}
\end{proof}

From this result, we can readily establish conditions under which the proximal gradient iteration, with a fixed steplength $\alpha_k = \alpha$, achieves Q-linear convergence of $\phi_k$, in expectation.

\begin{thm} \label{linearc}  Suppose that Assumptions~\ref{assume1} hold, that $\{x_k\}$ is generated by \eqref{proxmethod} with  $\alpha_k = (1-\eta)/L$ for $\eta \in (0,1)$, and that $g_k$ satisfies \eqref{unbiased}. If we have that for all $k$,
\begin{equation}
	\label{eq_test}
	\alpha_k \EE{k}{(\nabla f_k - g_k)^T (x_{k+1} - x_k)} \leq \frac{\eta}{2} \EE{k}{\nrm{}{x_{k+1} - x_k}^2}
\end{equation}
then 
\begin{equation} \label{linear}
	\EE{k}{\phi_{k+1} - \phi^*} \leq \bbr{1 - (1-\eta) \frac{\mu}{L}} (\phi_k - \phi^*).
\end{equation}
\end{thm}

 We note that when $h = 0$ and the iteration becomes \eqref{iteration},   condition \eqref{eq_test}  reduces  to the norm test \eqref{normt} with $\xi$ given by $\eta/(1-\eta)$.
 
The assumption on $\alpha_k$ in Theorem~\ref{linearc} is fairly standard. The key is inequality \eqref{eq_test}, which is the most general condition we have identified for ensuring linear convergence.
However, it does not seem to be possible to enforce this condition in practice, even approximately,
 for the composite optimization problem (which includes convex constrained optimization). Therefore, we seek an implementable version of \eqref{eq_test}, even if it is more restrictive. {Before doing so, we show that condition \eqref{eq_test} can also be used to establish convergence in the case when $f$ is convex, but not strongly convex.}

\subsection{Convergence for General Convex $f$}
We now show that when $f$ is convex, the sequence of function values $\{\phi(x_k)\}$ converges to the optimal value $\phi^*$ {of problem \eqref{probex}} at a sublinear rate, in expectation.  {To establish this result, we make the following assumptions.}

\begin{assum} \label{assume2} $f: \mathbb{R}^n \to \mathbb{R}$ is convex, differentiable and has an $L-$ Lipschitz continuous gradient, and  $h: \mathbb{R}^n \to \mathbb{R} \cup \{\infty\}$ is a closed, convex and proper function. 
\end{assum} }

We begin by proving a technical lemma.

\begin{lem}
	\label{lem:technical}
	Suppose that Assumptions \eqref{assume2} hold and that $\{x_k\}$ is generated by iteration \eqref{proxmethod}, where $g_k$ satisfies \eqref{unbiased} and $\alpha_k = (1-\eta)/L$ for $\eta \in (0,1)$. If in addition condition \eqref{eq_test} is satisfied, we have that for any given $z \in  \mathbb{T}_k$,
	\begin{equation}
	\E_k[\phi_{k+1}] \leq \phi(z) + \frac{1}{\alpha_k}\E_k\left[(x_k-x_{k+1})^T(x_k - z)\right] - \frac{1}{2\alpha_k}\E_k[\|x_k - x_{k+1}\|^2] .
	\end{equation}
\end{lem}
\begin{proof}
By Assumptions \eqref{assume2}, we have that
	\begin{align*}
	\phi_{k+1} &\leq f_k + \nabla f_k^T (x_{k+1} - x_k) + \frac{L}{2}\|x_{k+1} - x_k\|^2 + h_{k+1}  \\
	&\leq f(z) - \nabla f_k^T(z - x_k) + \nabla f_k^T (x_{k+1} - x_k) + \frac{L}{2}\|x_{k+1} - x_k\|^2 + h_{k+1} \cmnt{by convexity of $f$} \\
	&\leq f(z) - \nabla f_k^T(z - x_k) + \nabla f_k^T (x_{k+1} - x_k) + \frac{L}{2}\|x_{k+1} - x_k\|^2 + h(z) \\
	&\quad - \left(\frac{x_k - x_{k+1}}{\alpha_k} - g_k\right)^T(z - x_{k+1}) \cmnt{by convexity of $h$ and definition of $x_{k+1}$} \\
	&= \phi(z) - \nabla f_k^T(z - x_k) + \nabla f_k^T (x_{k+1} - x_k) + \frac{L}{2}\|x_{k+1} - x_k\|^2 \\
	& \quad - \left(\frac{x_k - x_{k+1}}{\alpha_k} - g_k\right)^T(z - x_k + x_k - x_{k+1}) \\
	&= \phi(z) + (g_k - \nabla f_k)^T(z - x_k) + \frac{1}{\alpha_k}(x_k-x_{k+1})^T(x_k - z) \\
	&\quad + \left(\frac{L}{2} - \frac{1}{\alpha_k}\right)\|x_k - x_{k+1}\|^2 + (\nabla f_k - g_k)^T(x_{k+1} - x_k), \cmnt{rearraging terms} 
	\end{align*}
	where the third inequality follows from the fact that $0 \in g_k + \partial h_{k+1}  + \frac{x_k - x_{k+1}}{\alpha_k}$. 
	Taking conditional expectation and using \eqref{eq_test}, we have
	\begin{align*}
	\E_k[\phi_{k+1}] &\leq \phi(z) + \E_k[(g_k - \nabla f_k)^T(z - x_k)] + \frac{1}{\alpha_k}\E_k[(x_k-x_{k+1})^T(x_k - z)] \\
	&\quad + \left(\frac{L}{2} - \frac{1}{\alpha_k}\right)\E_k[\|x_k - x_{k+1}\|^2] + \E_k[(\nabla f_k - g_k)^T(x_{k+1} - x_k)] \\
	&= \phi(z) + \frac{1}{\alpha_k}\E_k[(x_k-x_{k+1})^T(x_k - z)] + \left(\frac{L}{2} - \frac{1}{\alpha_k}\right)\E_k[\|x_k - x_{k+1}\|^2] \\
	&\quad + \E_k[(\nabla f_k - g_k)^T(x_{k+1} - x_k)] \cmnt{since $z \in  \mathbb{T}_k$ } \\
	&\leq \phi(z) + \frac{1}{\alpha_k}\E_k[(x_k-x_{k+1})^T(x_k - z)] - \left(\frac{1}{\alpha_k} - \frac{L}{2} - \frac{\eta}{2\alpha_k}\right)\E_k[\|x_k - x_{k+1}\|^2] \cmnt{by \eqref{eq_test}} \\
	&=  \phi(z) + \frac{1}{\alpha_k}\E_k[(x_k-x_{k+1})^T(x_k - z)] - \frac{1}{2\alpha_k}\E_k[\|x_k - x_{k+1}\|^2 ,
	\end{align*}
	where the last equality is due to $\alpha_k = \frac{1 - \eta}{L}$. 
\end{proof}	
\begin{thm}
	Suppose that Assumptions \eqref{assume2} hold and that $\{x_k\}$ is generated by iteration \eqref{proxmethod}, where $g_k$ satisfies \eqref{unbiased} and $\alpha_k = \alpha=(1-\eta)/L$ for $\eta \in (0,1)$. If in addition \eqref{eq_test} is satisfied, we have
	\begin{equation}\label{sublinear}
	\E[\phi_k - \phi^*] \leq \frac{L\|x_0 - x^*\|^2}{2(1-\eta) k},
	\end{equation}
	where $x^*$ is any optimal solution of  problem \eqref{probex}. 
\end{thm}

\begin{proof}
	From Lemma \ref{lem:technical}, for any $z \in  \mathbb{T}_k$, we have
	\begin{equation*}
	\E_k[\phi_{k+1}] \leq \phi(z) + \frac{1}{\alpha}\E_k\left[(x_k-x_{k+1})^T(x_k - z)\right] - \frac{1}{2\alpha}\E_k[\|x_k - x_{k+1}\|^2] .
	\end{equation*}
	Now substituting $z = x^* \in  \mathbb{T}_k$ and taking full expectations, we have 
	\begin{align*}
	\E[\phi_{k+1} - \phi^*] &\leq  \frac{1}{\alpha}\E\left[(x_k-x_{k+1})^T(x_k - x^*)\right] - \frac{1}{2\alpha}\E[\|x_k - x_{k+1}\|^2] \\
	&=\frac{1}{2\alpha}\E\left[2(x_k-x_{k+1})^T(x_k - x^*) - \|x_k - x_{k+1}\|^2\right] \\
	&=\frac{1}{2\alpha}\E\left[\|x_k - x^*\|^2 - \|x_{k+1} - x^*\|^2\right] .
	\end{align*}
	Summing the above inequality for $k=0$ to $k-1$, we get
	\begin{align*}
	\frac{1}{k} \sum_{t=0}^{k-1} \E[\phi_{t+1} - \phi^*] &\leq \frac{1}{2\alpha k}\E\left[\|x_0 - x^*\|^2 - \|x_{k} - x^*\|^2\right] \\
	&\leq \frac{\|x_0 - x^*\|^2}{2\alpha k} .
	\end{align*}
	By substituting $z =x_k \in \mathbb{T}_k$ in Lemma \ref{lem:technical}, we have that the sequence of expected function values is a deceasing sequence; specifically,
	\begin{align*}
	\E_k[\phi_{k+1}] &\leq \phi_k - \frac{1}{2\alpha}\E_k[\|x_k - x_{k+1}\|^2]. 
	\end{align*}
	Therefore, we have, 
	\begin{align*}
	\E[\phi_{k} - \phi^*] &\leq \frac{1}{k} \sum_{t=0}^{k-1} \E[\phi_{t+1} - \phi^*] \leq \frac{\|x_0 - x^*\|^2 }{2\alpha k}.
	\end{align*}
\end{proof}	

\subsection{A Practical Condition}
To obtain a condition that is more amenable to computation than \eqref{eq_test}, we look for an upper bound for the left hand side of this inequality and a lower bound for the right hand side. Imposing an inequality between these two bounds will imply \eqref{eq_test}.

\begin{thm}
\label{thm_pessimistic_test} 
{Suppose that Assumptions~\ref{assume1} hold, that $\{x_k\}$ is generated by \eqref{proxmethod} with  $\alpha_k = (1-\eta)/L$ for $\eta \in (0,1)$, and that $g_k$ satisfies \eqref{unbiased}. 
If $g_k$ additionally satisfies }
\begin{equation}
	\label{ptest}
	\varr{k}{g_k} \leq \frac{\eta}{2} {\nrm{}{\frac{\EE{k}{x_{k+1}} - x_k}{\alpha_k}}^2},
\end{equation}
then \eqref{eq_test} holds and hence
\begin{equation*}
	\EE{k}{\phi_{k+1} - \phi^*} \leq \bbr{1 - (1-\eta) \frac{\mu}{L}} (\phi_k - \phi^*).
\end{equation*}
If instead of Assumptions~\ref{assume1}, the weaker Assumptions~\ref{assume2} hold, then
\begin{equation*}
\E[\phi_k - \phi^*] \leq \frac{L\|x_0 - x^*\|^2}{2(1-\eta) k},
\end{equation*}
where $x^*$ is any optimal solution of  problem \eqref{probex}. 
\end{thm}
\begin{proof}
 We first note that in \eqref{eq_test}, 
\begin{align*}
	\EE{k}{\nrm{}{x_{k+1} - x_k}^2}  = ~&\nrm{}{\EE{k}{x_{k+1}} - x_k}^2 + \varr{k}{x_{k+1} - x_k} \\
	\geq ~& \nrm{}{\EE{k}{x_{k+1}} - x_k}^2 , 
\end{align*}
which is a quantity that we can approximate with a sample estimation; as we argue in {Section~\ref{estimation}.}

On the other hand, let
\begin{equation*}
	\wh x_{k+1} = \text{prox}_{\alpha_k h}\bpa{x_k - \alpha_k \nabla f_k} \in \mathbb{T}_k .
\end{equation*} 
Then, {since the prox operator is a contraction mapping,}
\begin{align*}
	& \EE{k}{(\nabla f_k - g_k)^T (x_{k+1} - x_k)} \\
	& ~~~= ~\EE{k}{(\nabla f_k - g_k)^T (x_{k+1} - \wh x_{k+1})} +
	\EE{k}{(\nabla f_k - g_k)^T (\wh x_{k+1} - x_k)}  \\
	& ~~~= ~\EE{k}{(\nabla f_k - g_k)^T (x_{k+1} - \wh x_{k+1})}  \cmnt{since $\wh x_{k+1} \in \mathbb{T}_k$}  \\
	& ~~~\leq  ~\EE{k}{\nrm{}{\nabla f_k - g_k} \nrm{}{x_{k+1} - \wh x_{k+1}}} \\
	 & ~~~=  ~\EE{k}{\nrm{}{\nabla f_k - g_k} \nrm{}{\text{prox}_{\alpha_k h}\bpa{x_k - \alpha_k g_k} - \text{prox}_{\alpha_k h}\bpa{x_k - \alpha_k \nabla f_k}}} \\
	& ~~~\leq  ~\alpha_k \EE{k}{\nrm{}{\nabla f_k - g_k}^2} \\
	& ~~~= \alpha_k \varr{k}{g_k}.
\end{align*}
Thus, we have obtained both
\begin{equation*}
	\frac{\eta}{2} {\nrm{}{\EE{k}{x_{k+1}} - x_k}^2} \leq \frac{\eta}{2} \EE{k}{\nrm{}{x_{k+1} - x_k}^2}
\end{equation*}
and
\begin{equation*}
	\alpha_k \EE{k}{(\nabla f_k - g_k)^T (x_{k+1} - x_k)} \leq \alpha_k^2 \varr{k}{g_k}.
\end{equation*}
Therefore, if we require that
\begin{align*}
	& \varr{k}{g_k} \leq \frac{\eta}{2} {\nrm{}{\frac{\EE{k}{x_{k+1}} - x_k}{\alpha_k}}^2} ,
\end{align*}
it follows that condition \eqref{eq_test} is satisfied.
\end{proof}

The significance of Theorem \ref{thm_pessimistic_test} is that it establishes the convergence of the algorithm under condition \eqref{ptest} which, although being more restrictive than condition \eqref{eq_test}, can be approximated empirically, as shown in Section~\ref{estimation}. Again, it is reassuring that when $h = 0$, condition \eqref{ptest} reduces to the  norm test \eqref{normt}. 

\subsection{Choice of the Sample Size $S_k$}   \label{estimation}
We now discuss how to ensure that condition  \eqref{ptest} is satisfied. 
{At iterate $x_k$, suppose we select $S_k$  i.i.d. samples $\bst{\theta_0, \theta_1, \cdots, \theta_{S_k-1}}$ from $\Theta$,  and set  
\begin{equation} \label{gsum}
	g_k = \frac{1}{S_k} \sum_{i=0}^{S_k-1} ~ \nabla_x F(x_k, \theta_i).
\end{equation}
Clearly, \eqref{unbiased} holds and 
the variance of $g_k$ is given by  
\begin{align}  
	\varr{k}{g_k} 
	&   =  \frac{\E_k\left[\|\nabla_x F(x_k, \theta) - \nabla f(x_k)\|^2\right]}{S_k} .\nonumber 
\end{align}
Therefore, \eqref{ptest} holds if $S_k$ satisfies 
\begin{equation}
\frac{ \E_k \left[\|\nabla_x F(x_k, \theta)  - \nabla f(x_k)\|^2\right]}{S_k} \leq \frac{\eta}{2} {\nrm{}{\frac{\EE{k}{x_{k+1}} - x_k}{\alpha_k}}^2},
\end{equation}
or 
\begin{equation}  \label{populations}
S_k \geq 
\E_k\left[\|\nabla_x F(x_k, \theta) - \nabla f(x_k)\|^2\right] \Bigg / \frac{\eta}{2} {\nrm{}{\frac{\EE{k}{x_{k+1}} - x_k}{\alpha_k}}^2}.
\end{equation}
This is the theoretical condition suggested by our analysis.  Based on this condition we can state
\begin{cor}
\label{sample_corollary}
Suppose that Assumptions~\ref{assume1} hold, that $\{x_k\}$ is generated by \eqref{proxmethod} with  $\alpha_k = (1-\eta)/L$ for $\eta \in (0,1)$ and with $g_k$ given by \eqref{gsum}.  If the sample sizes $S_k$ are chosen to satisfy \eqref{populations} for all $k$,
then \eqref{eq_test} is satisfied and hence \eqref{linear} holds. 
If instead of Assumptions~\ref{assume1}, the weaker Assumptions~\ref{assume2} are satisfied, then \eqref{sublinear} holds.
\end{cor}

 In practice, we need to estimate both quantities on the right hand side of \eqref{populations}
As  was done e.g. in  \cite{bollapragada2018adaptive,byrd2012sample} the population variance term in the numerator can be approximated by a sample average:
 \begin{equation}  
	  \E_k  \left[\|\nabla_x F(x_k, \theta ) - \nabla f(x_k)\|^2\right] 
	 \approx \frac{1}{S-1} \sum_{i=0}^{S-1} \nrm{}{\nabla_x F(x_k, \theta_i) - \overline{g}_k}^2 , \label{popvar}
\end{equation}
where $\overline{g}_k$ is defined in \eqref{barg}.
We handle the denominator on the right hand side of \eqref{populations} by approximating 
$\|\mathbb{E}_k[x_{k+1}]-x_k \|$ with the norm of the trial step $\|\overline{x}_{k+1}-x_k\|$  defined in \eqref{x_estimate}, i.e.,
\begin{equation*}
	\overline{x}_{k+1} = \text{prox}_{\alpha_k h}\bpa{x_k - \alpha_k \overline{g}_k}.
\end{equation*}
This is somewhat analogous to the approximations made in the unconstrained case  in  \cite{bollapragada2018adaptive,byrd2012sample}, the main difference being  the presence here of the prox operator, which is a contraction.
Given this, we can replace \eqref{populations} by
\begin{equation} 
	\label{sfullnew}
	S_k \geq \frac{1}{S-1} \sum_{i=0}^{S-1} \nrm{}{\nabla_x F(x, \theta_i) - \overline{g}_k}^2 \Bigg / \frac{\eta}{2} \nrm{}{\frac{\overline{x}_{k+1}-x_k}{\alpha_k}}^2.
\end{equation}
Our algorithm will use condition \eqref{sfullnew} to control the sample size.

\subsection{The Practical Adaptive Sampling Algorithm}
\label{sec:alg}
We now summarize the algorithm proposed in this paper, which employs the aforementioned approximations.
\bigskip
\begin{framed}
\noindent \textbf{Algorithm 2: Complete  Algorithm for Solving Problem \eqref{probex}}:

\smallskip
\noindent \textbf{Input}: $x_0$, initial sample size $S \in \mathbb{N}^+$, and sequence $\{\alpha_k >0\}$.

\smallskip
\noindent  \textbf{For k=1, \dots}:

 1. Draw $S$  i.i.d. samples $\bst{\theta_0, \theta_1, \cdots, \theta_{S-1}}$ from $\Theta$,   compute 
\begin{equation*}
	\overline{g}_k = \frac{1}{S} \sum_{i=0}^{S-1} ~ \nabla_x F(x, \theta_i),
\end{equation*}
~~~~~~~~~and trial proximal gradient step
\begin{equation}
	\label{x_full}
	\overline{x}_{k+1} = \text{prox}_{\alpha_k h}\bpa{x_k - \alpha_k \overline{g}_k} .
\end{equation}
~~~~~2. Compute
\begin{equation*}
	\label{sfull}
	a = \frac{1}{S-1} \sum_{i=0}^{S-1} \nrm{}{\nabla_x F(x, \theta_i) - \overline{g}_k}^2 \Bigg / \frac{\eta}{2} \nrm{}{\frac{\overline{x}_{k+1}-x_k}{\alpha_k}}^2
\end{equation*}
~~~~~~~~~and set
\begin{equation} \label{sfinal}
    S_k= \max \{a, S\} .
\end{equation}

3. If  $S_k > S$, choose $(S_k-S)$ additional i.i.d. samples  $\bst{ \theta_{S}, \theta_{S+1}, \cdots,  \theta_{S_k-1}}$  from $\Theta$,  and compute:
\begin{equation}
 g_k=   \frac{1}{S_k} \sum_{i=0}^{{S_k}-1}~ \nabla_x F(x, \theta_i)
 \end{equation}
\begin{equation}
	x_{k+1} = \text{prox}_{\alpha_k h}\bpa{x_k - \alpha_k g_k}
\end{equation}
~~~~~~~~\textbf{Else}
\[x_{k+1}= \overline{x}_{k+1}.\]

4. Set $S \leftarrow S_k$

\noindent  \textbf{End For}
\end{framed}

\medskip


As noted above, for large $S$, the choice of ${S_k}$ given by \eqref{sfinal} approximately ensures the condition \eqref{ptest} is satisfied. Computing ${S_k}$  involves one evaluation of the proximal operator as well as the evaluation $S$ stochastic gradients.


\section{Using an Inner-Product Test in Place of the Norm Test}
\label{sec:ip}
 In the unconstrained setting, Bollapragada et al \cite{bollapragada2018adaptive} have observed that the norm test, although endowed with optimal theoretical convergence rates, is too demanding  in terms of sample size requirements. They derived a practical test called the \emph{inner-product test}, which ensures that the search directions are descent directions with high probability,  and performs well with smaller sample sizes. In this section, we extend these ideas to the constrained (or composite) optimization settings, and derive the equivalent inner-product test  for adaptively controlling the sample sizes.

The goal is to choose the sample sizes such that the algorithm step provides descent with high probability. 
We would like to choose a sample size so that $\bar{d_k} =(\bar{x}_{k+1}-x_k)/\alpha_k$ provides decrease on the objective approximation $ \nabla f (x_k)^T{d} + h(x_k+ {d}), $ and specifically so that  $ \nabla f (x_k)^T\bar{d_k} + h(x_k+ \bar{d_k}) -h(x_k) \leq \beta (\bar{g_k}^T \bar{d_k} + h(x_k+ \bar{d_k}) -h(x_k))$ for some $\beta \in (0,1)$. This means we want to satisfy
\begin{equation} \label{ip}
 (\nabla f (x_k) - \bar{g_k}  ) ^T\bar{d_k}  \leq -(1-\beta)\left( \bar{g_k}^T \bar{d_k}+h(x_k+ \bar{d_k})-h(x_k) \right) .
\end{equation}
To estimate the left hand side of \eqref{ip}, note that in general, given a vector $p\in \R^n$, since $\EE{k}{(\bar{g}- \nabla f(x_k))^T p}=0$, we can estimate the size of the quantity $(\bar{g_k} - \nabla f (x_k)) ^Tp$ by estimating the variance of $(\bar{g_k} - \nabla f (x_k)) ^Tp $.  Since the initial sample size is $S$ this is given by
\begin{equation}     \label{ipvar}
\varr{k}{\bar{g}_k^Tp} \approx \frac{1}{S_k} \frac{1}{S-1} \sum_{i=0}^{S-1} \left((\nabla_x F(x, \theta_i) - \overline{g}_k)^T p\right )^2 .
\end{equation}
We use this estimate in \eqref{ip} with $p=\bar{d_k} $ and get the condition
\begin{equation}    \label{ipsample}
S_k \geq  \frac{1}{S-1} \sum_{i=0}^{S-1} \left((\nabla_x F(x, \theta_i) - \overline{g}_k)^T \bar{d}_k\right )^2 \Bigg /
(1-\beta)^2(\bar{g_k}^T \bar{d_k}+h(x_k+ \bar{d_k})-h(x_k) )^2 .
\end{equation}
where $(1-\beta)^2$ is analogous to $\eta/2$ in \eqref{sfullnew}.
We are aware that, in using \eqref{ipvar} with $p=\bar{d_k} $  , we are treating $\bar{g_k}$ and $\bar{d_k}$ as independent while they are not, but the practical success of the inner product test in  \cite{bollapragada2018adaptive} indicates that this approach is  worthy of exploration.  

We also note that in a problem with convex constraints, where  $h(\cdot)$  involves a convex indicator function with possibly infinite values, the algorithm will only generate feasible points $x_k$ and $x_k+ \bar{d_k}$, so that in the above discussion $h$ only takes on finite values at those points.

The version of our algorithm using this approach consists of following Algorithm 2 with the right hand side of \eqref{ipsample} used in place of the formula for $a$ in Step 2.  

\section{Numerical Experiments}
We conducted numerical experiments to illustrate the performance of the proposed algorithms. We consider binary classification problems where the objective function is given by the logistic loss  with $\ell_1$-regularization:
\begin{equation} \label{logistic-loss}
\phi(x)=\frac{1}{N} \sum_{i=1}^{N}\log(1 + \exp(-y^ix^Tz^i)) + {\lambda}\|x\|_1.
\end{equation}  
Here $\{(z^i,y^i),~i=1,\ldots, N\}$ are the input output data pairs, and the regularization parameter  is chosen as $ \lambda = {1}/{N}$. This problem falls into the general category of minimizing composite optimization problems of the form \eqref{proxprob}, where $f(x) = \E\left[\log(1 + \exp(-y^ix^Tz^i))\right]$, with expectation taken over a discrete uniform probability distribution defined on the dataset, and $h(x) = \lambda \|x\|_1$.
We use the data sets listed in Table~\ref{tab:data_sets}. 
\begin{table}[htp]
	\centering
	\begin{tabular}{|c||c|c|c|}
		\hline
		Data Set  & Data Points $N$ & Variables $d$ & Reference \\ \hline
		\texttt{covertype} & 581012       & 54          &           \cite{blackard1999comparative}\\ 
		\texttt{gisette}   & 6000		   & 5000		 &			\cite{guyon2004result} \\
		\texttt{ijcnn}     & 35000         & 22         &          \cite{Lichman2013} \\ 
		\texttt{MNIST}    & 60000         & 784         &      \cite{lecun2010mnist}     \\    
		\texttt{mushrooms} & 8124          & 112         &          \cite{Lichman2013} \\
		\texttt{sido}     & 12678         & 4932         &          \cite{Lichman2013} \\
		\hline
	\end{tabular}
	\caption{Characteristics of the binary datasets used in the experiments.}
	\label{tab:data_sets}
\end{table} 


We implemented three different variants of proximal stochastic gradient methods where the batch sizes are either: (1) continuously increased at a geometric rate (labelled \texttt{GEOMETRIC}), i.e.,
\begin{equation}  \label{greene}
	S_k = \left \lceil S_0 \left(1 + \gamma \right)^k \right \rceil;
\end{equation}
where $\gamma>0$ is a parameter that will be varied in the experiments;
(2) adaptively chosen based on the norm test \eqref{sfullnew} (labelled \texttt{NORM}); or (3) adaptively chosen based on the inner-product test \eqref{ipsample} (labelled \texttt{IP}). The steplength parameter $\alpha_k$ in each method is chosen as the number in  the set $\{2^{-10}, 2^{-7}, \cdots, 2^{15}\}$ that leads to best performance. The initial sample size was set to $S_0 = 2$. The methods are terminated if $\|x_{k+1} - x_k\| / \alpha_k \leq 10^{-8}$ or if 100 epochs (passes through entire dataset) are performed. An approximation $\phi^*$ of the optimal function value was computed for each problem by running the deterministic proximal gradient method for 50,000 iterations.

Figures~\ref{fig:mushroom_func} and \ref{fig:mushroom_batchsizes} report the performance of the three methods for the dataset \texttt{mushroom}, for various values of the parameter $\gamma$ in \eqref{greene} and $\eta$ in
\eqref{sfullnew} and \eqref{ipsample} ($\eta$ in \eqref{ipsample} corresponds to $2(1-\beta)^2$). Figure~\ref{fig:mushroom_func}, the vertical axis measures the optimality gap, $\phi(x) - \phi^*$, and the horizontal axis measures the number of effective gradient evaluations, defined as ${\sum_{j=0}^k~S_j}/{N}$. In Figure~\ref{fig:mushroom_batchsizes}, the vertical axis measures the batchsize as a fraction of total number of data points $N$, and the horizontal axis measures the number of iterations. The results for other datasets in Table~\ref{tab:data_sets} can be found in Appendix \ref{addnumerical}.

We observe that the inner product test is the most efficient in terms of effective gradient evaluations, which is indicative of the total computational work and CPU time. For the geometric strategy, smaller values of $\gamma$ typically lead to better performance, as they prevent the batch size from growing too rapidly. The norm test gives a performance comparable to the best runs of the geometric strategy, but the latter has much higher variability. In other words, whereas the geometric strategy can be quite sensitive to the choice of $\gamma$, the norm and inner product tests are fairly insensitive to the choice of $\eta$.

\begin{figure}[htp]
    \centering
    \includegraphics[width=0.45\linewidth]{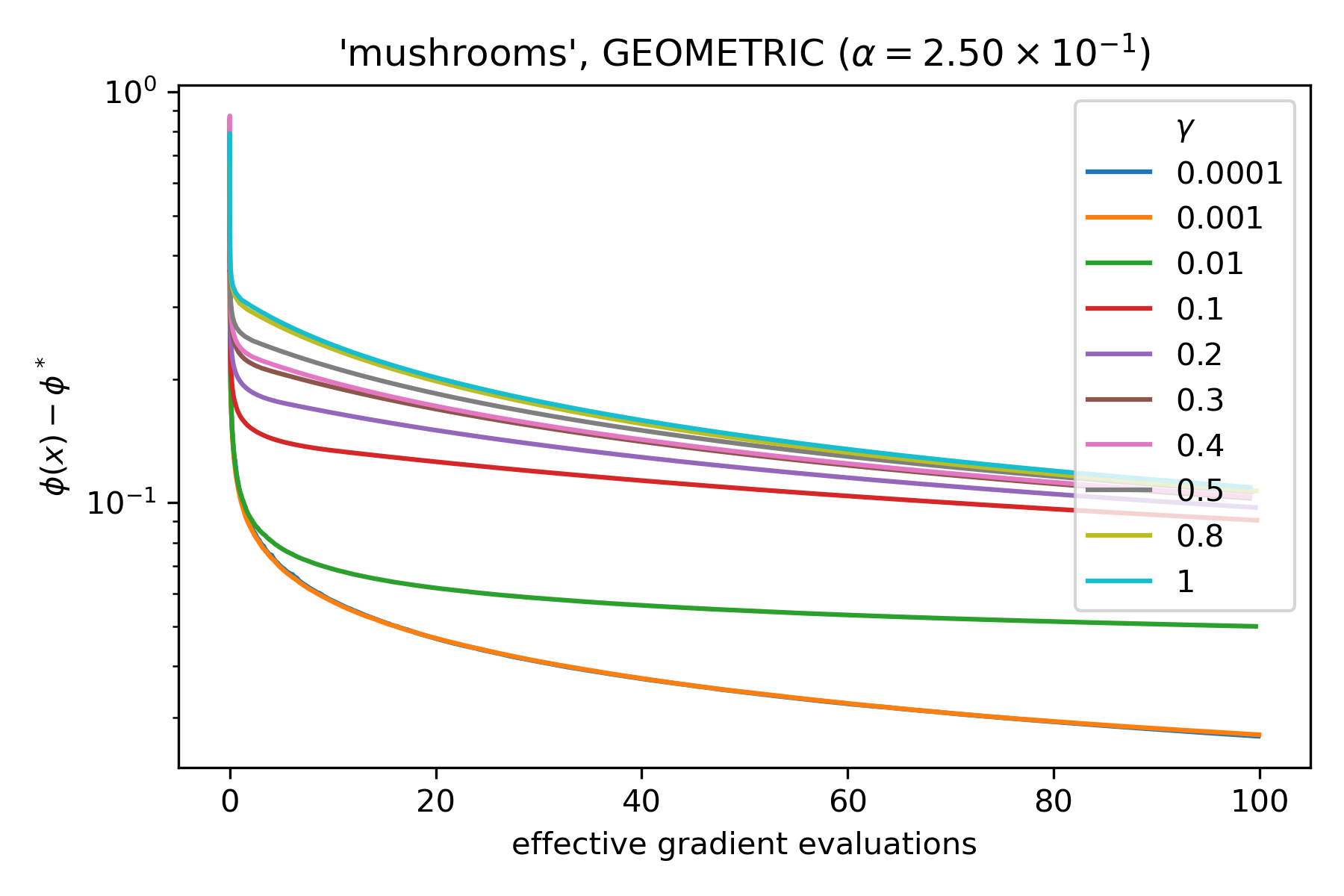}
	\includegraphics[width=0.45\linewidth]{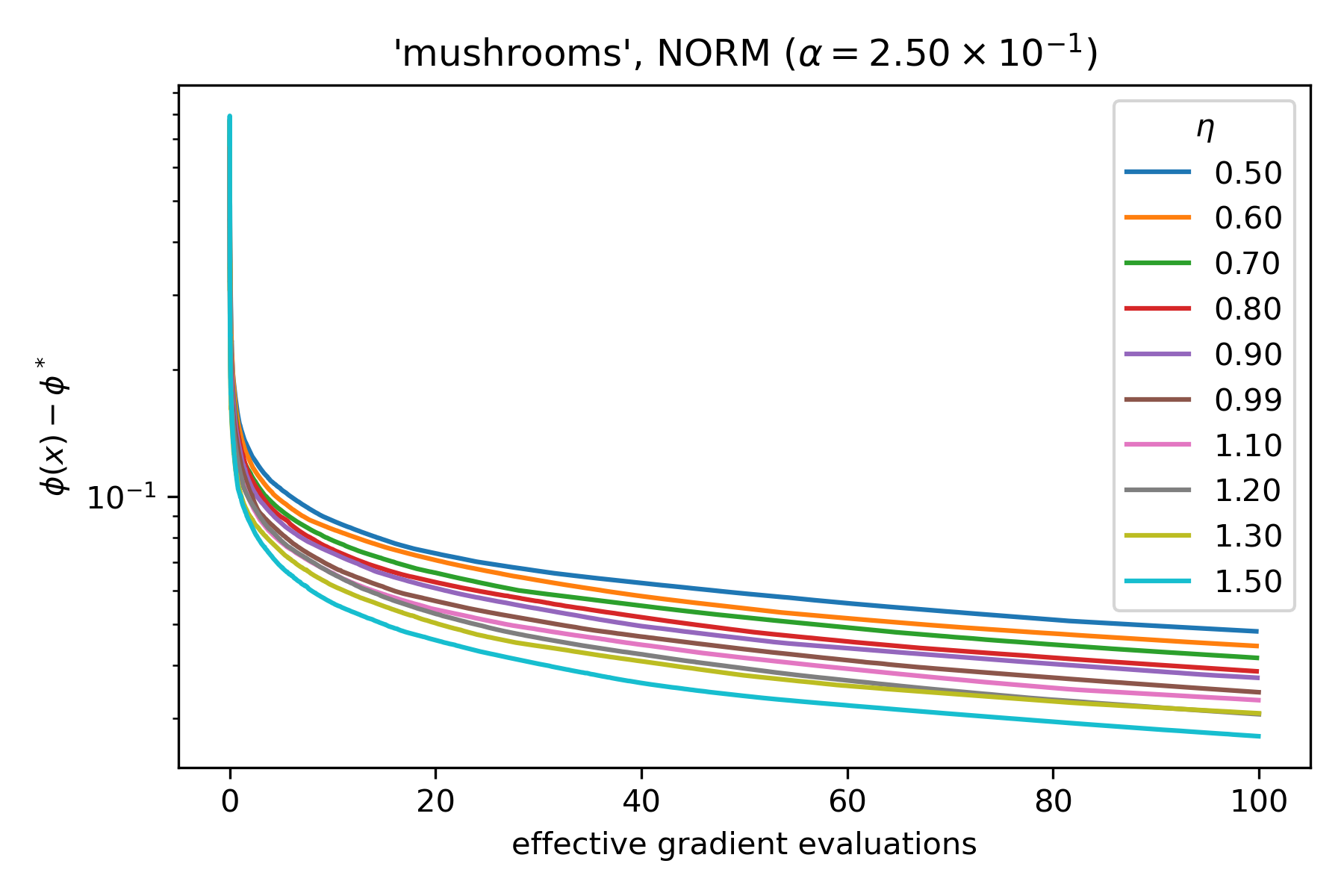} \\
	\includegraphics[width=0.45\linewidth]{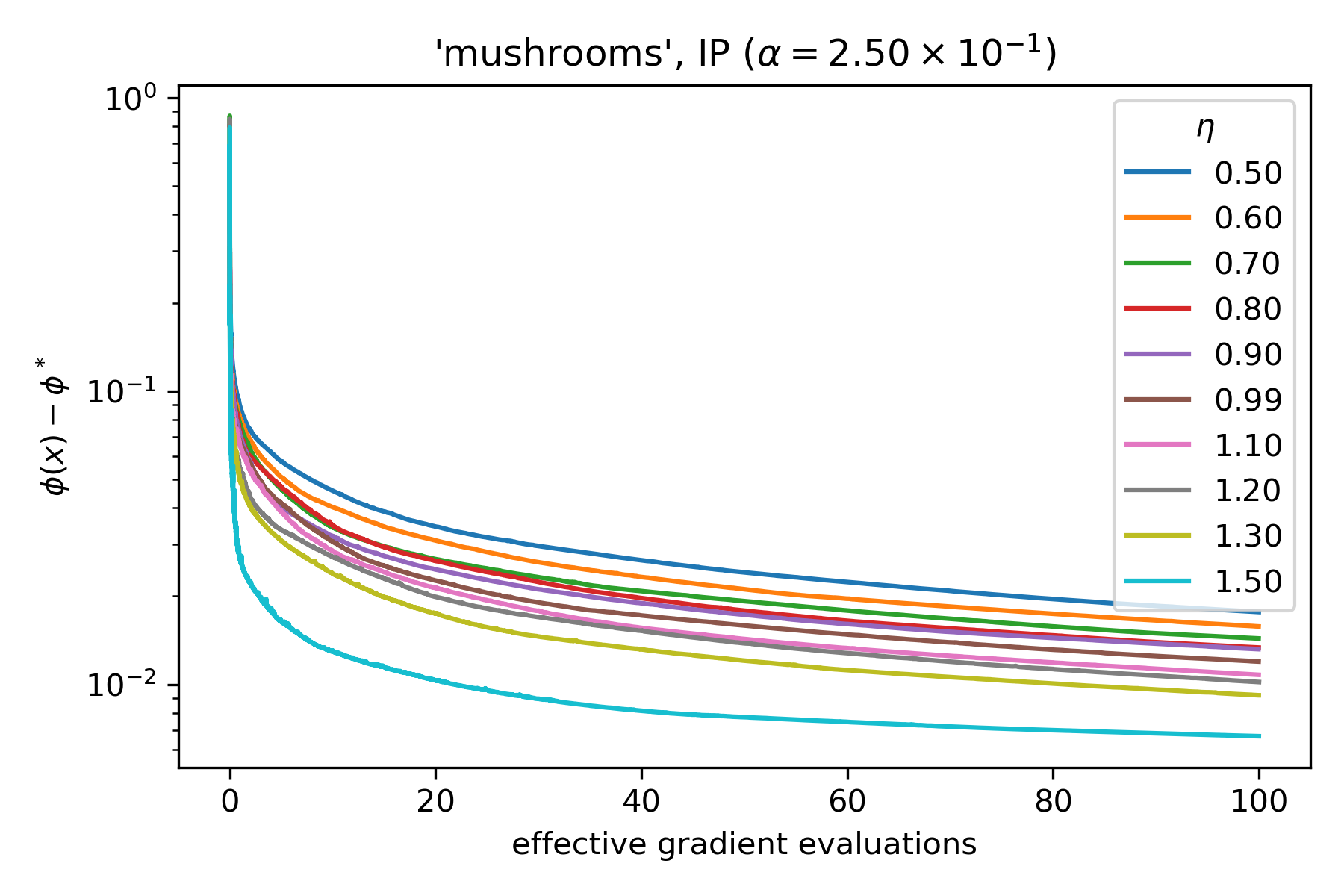}
    \includegraphics[width=0.45\linewidth]{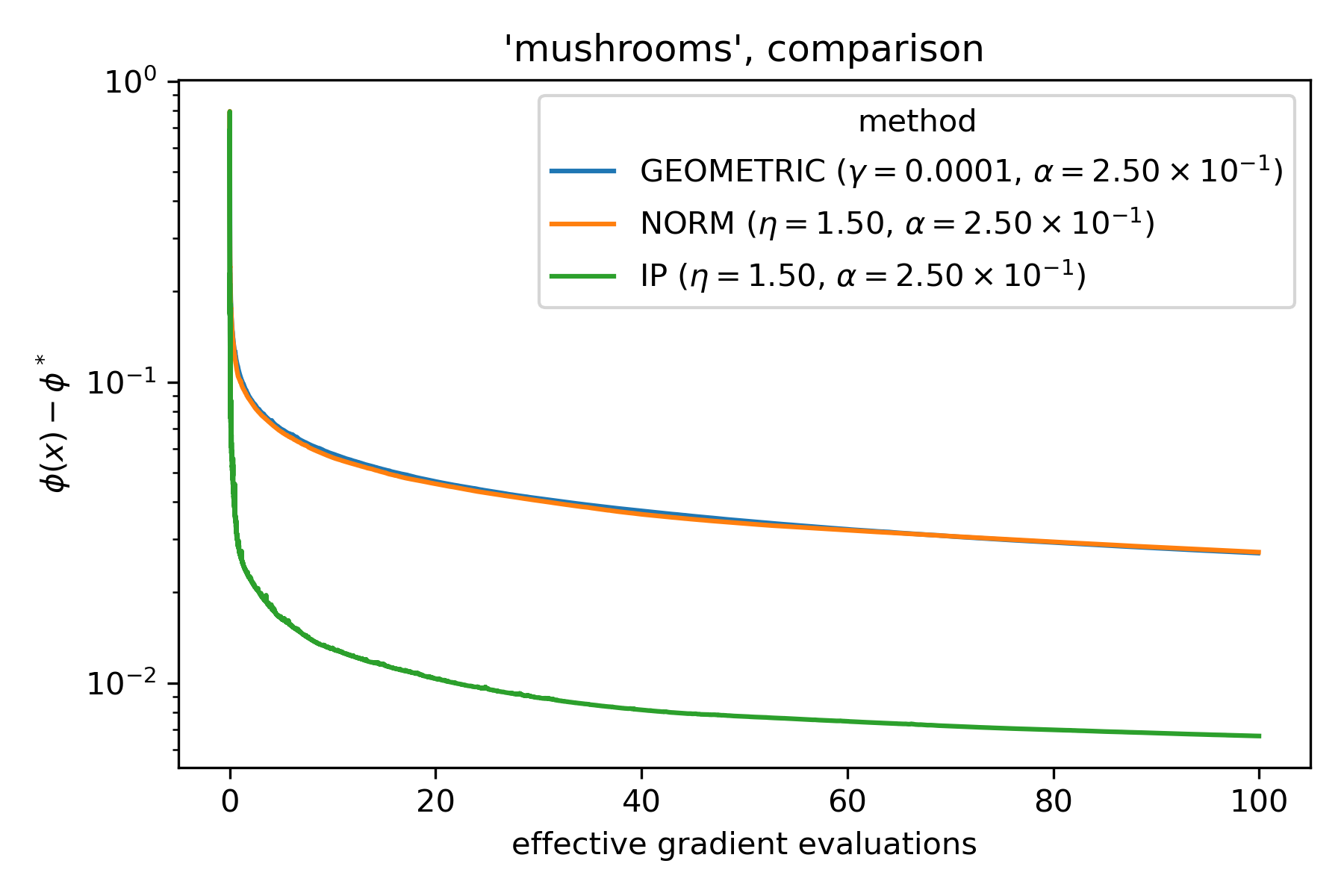} \\
    \caption{Optimality gap $\phi(x_k) - \phi^*$ against effective gradient evaluations on dataset \texttt{mushrooms}, with different strategies to control batch size: geometric increase (top left), norm test (top right), inner-product test (bottom left), and comparison between the best run for each method (bottom right). 
    }
    \label{fig:mushroom_func}
\end{figure}

\begin{figure}[htp]
    \centering
    \includegraphics[width=0.45\linewidth]{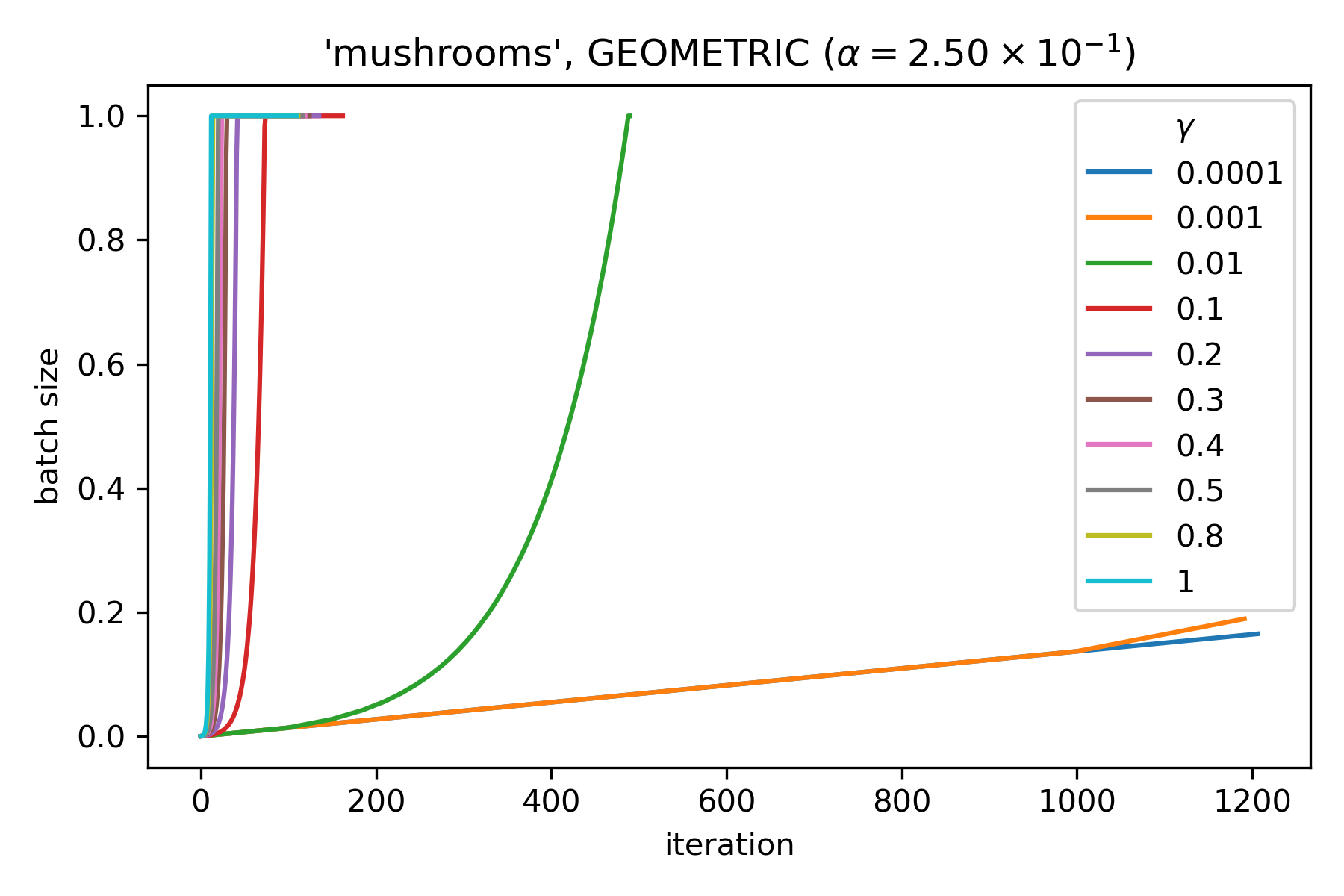}
	\includegraphics[width=0.45\linewidth]{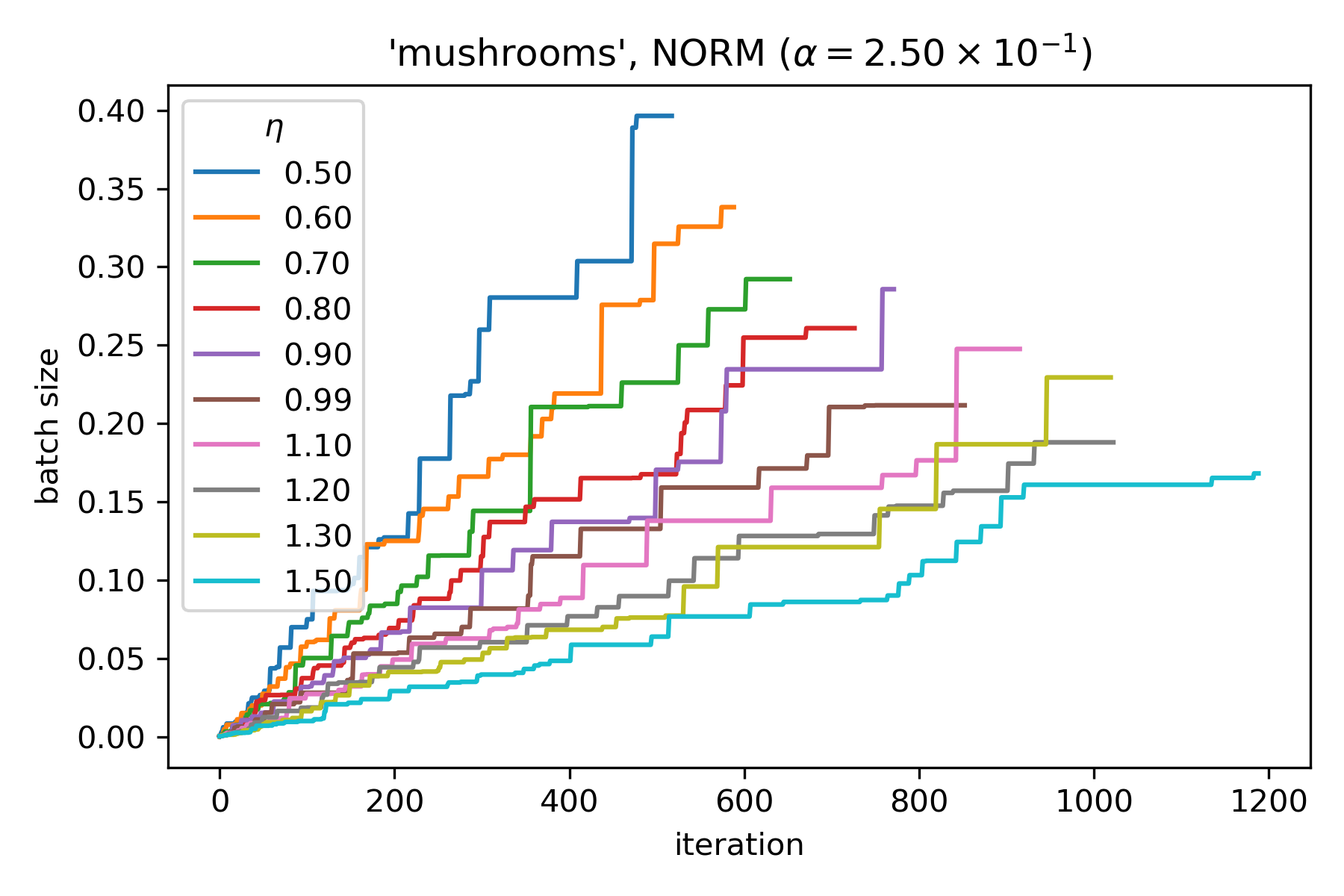} \\
	\includegraphics[width=0.45\linewidth]{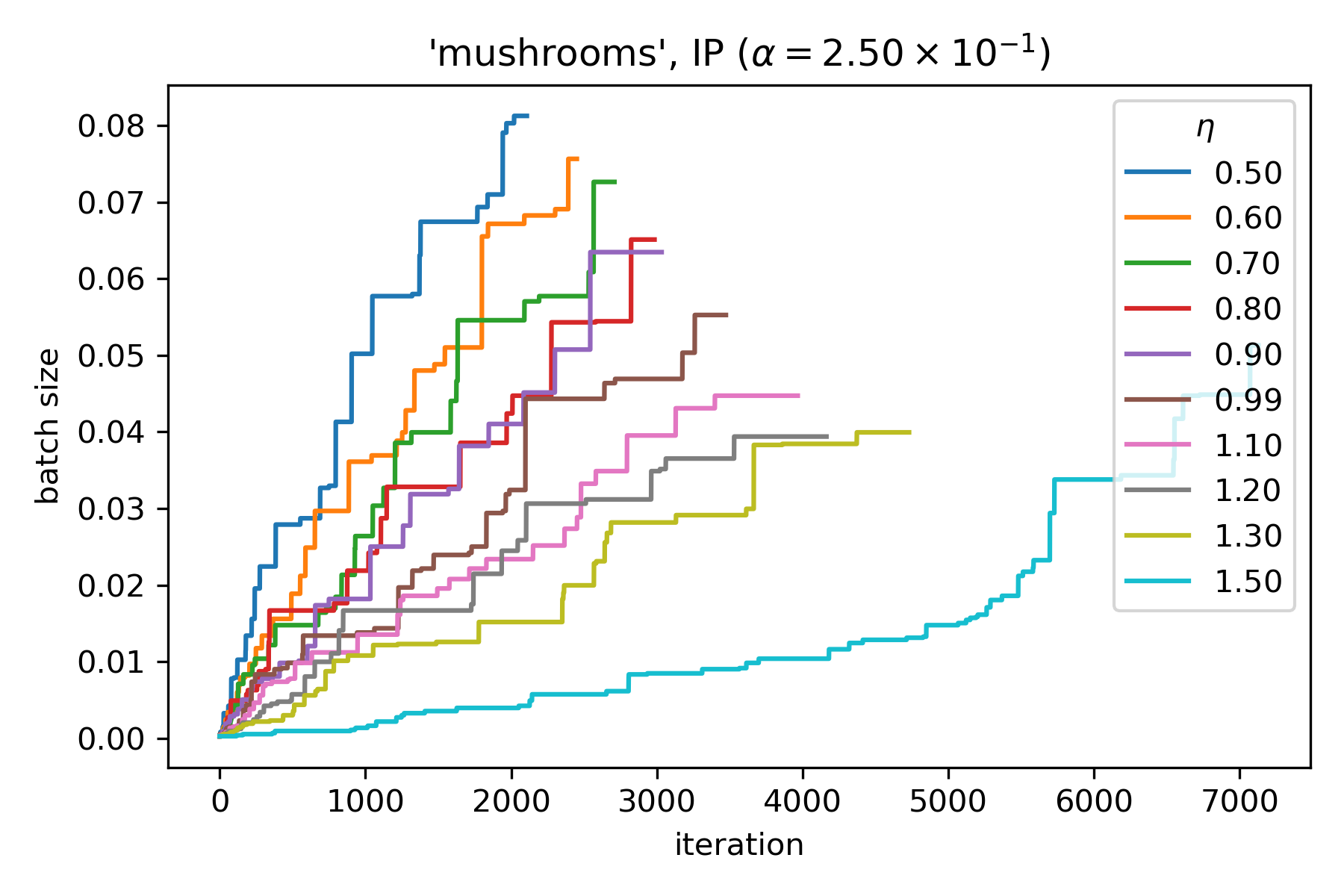}
    \includegraphics[width=0.45\linewidth]{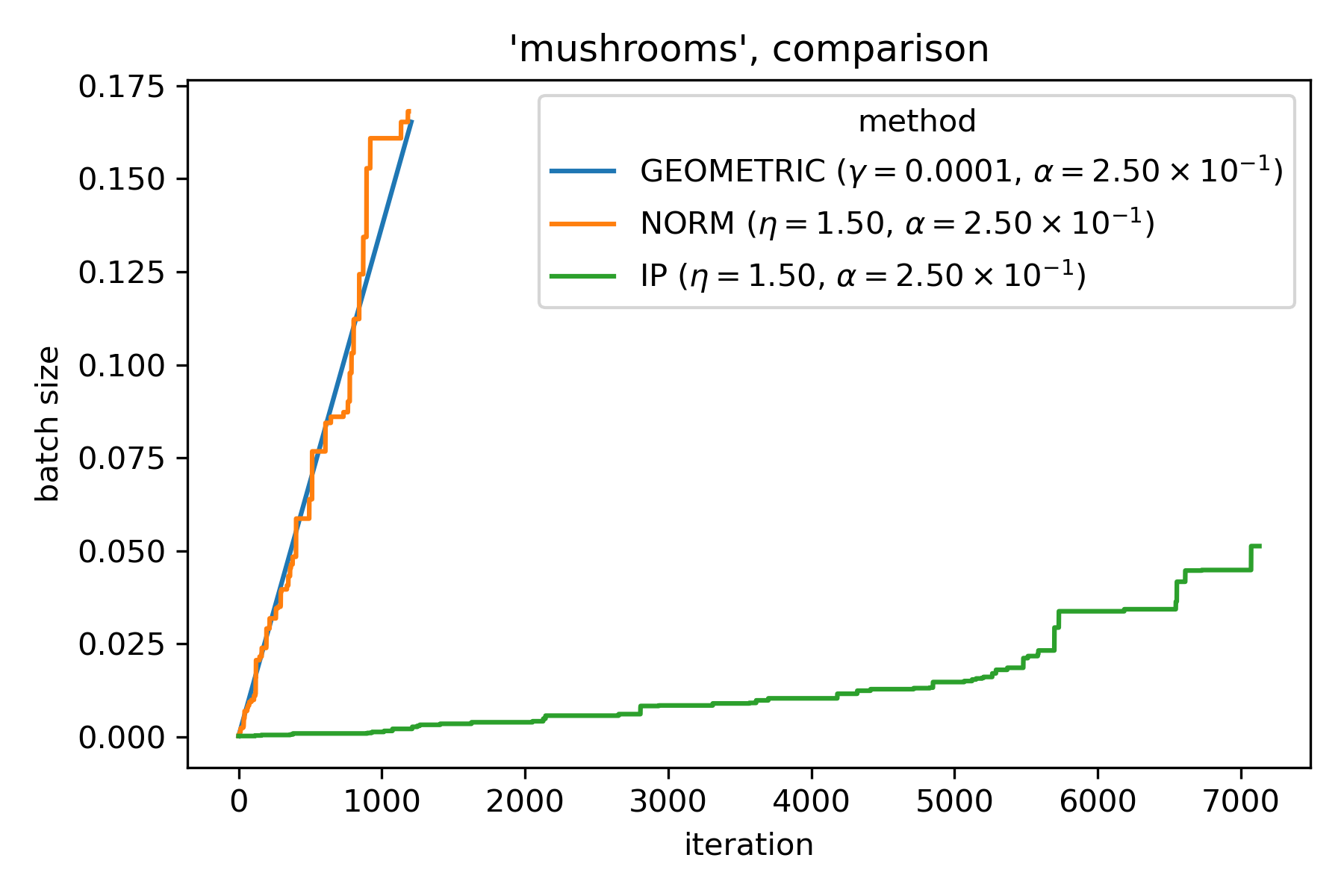} \\
    \caption{Batch size (as a fraction of total number of data points $N$) against iterations on dataset \texttt{mushrooms}, with different strategies to control batch size: geometric increase (top left), norm test (top right), inner-product test (bottom left), and comparison between the best run for each method (bottom right).
    }
    \label{fig:mushroom_batchsizes}
\end{figure}

\section{Final Remarks}
Algorithms that adaptively improve the quality of the approximate gradient during the optimization process are of interest from theoretical and practical perspectives, and have been well-studied in the context of unconstrained optimization.  
In this paper, we proposed an adaptive method for solving constrained and composite optimization problems. The cornerstone of the proposed algorithm and its analysis is condition \eqref{eq_test}, which we regard as a  natural generalization of the well-known norm test from unconstrained optimization. As this condition is difficult to implement precisely in practice, we approximate it by  condition \eqref{ptest}. We are able to prove convergence for the resulting methods under standard conditions. It remains to be seen whether there is a condition with similar properties as \eqref{eq_test}, that is  amenable to computation and less restrictive than \eqref{ptest}. In this paper, we also proposed a practical  inner-product condition \eqref{ip} that extends the ideas proposed in the unconstrained settings, and is more efficient in practice than the norm condition.}
\newpage
\bibliographystyle{plain}
\bibliography{cosample_Arxiv}
\appendix
\newpage
\section{Additional Numerical Experiments}\label{addnumerical}
Here we present the numerical experiments for remaining data sets.

\begin{figure}[htp]
    \centering
    \includegraphics[width=0.45\linewidth]{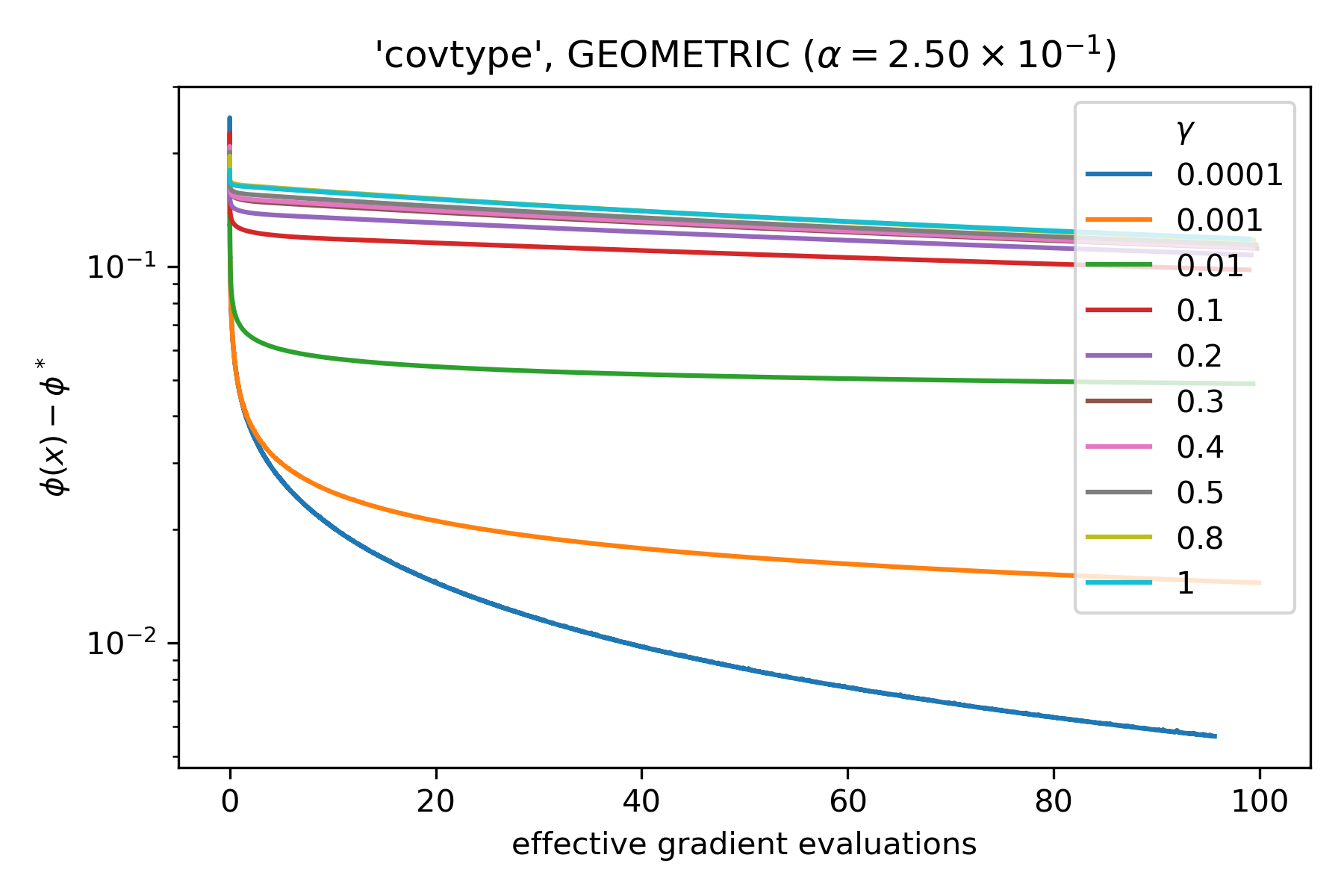}
	\includegraphics[width=0.45\linewidth]{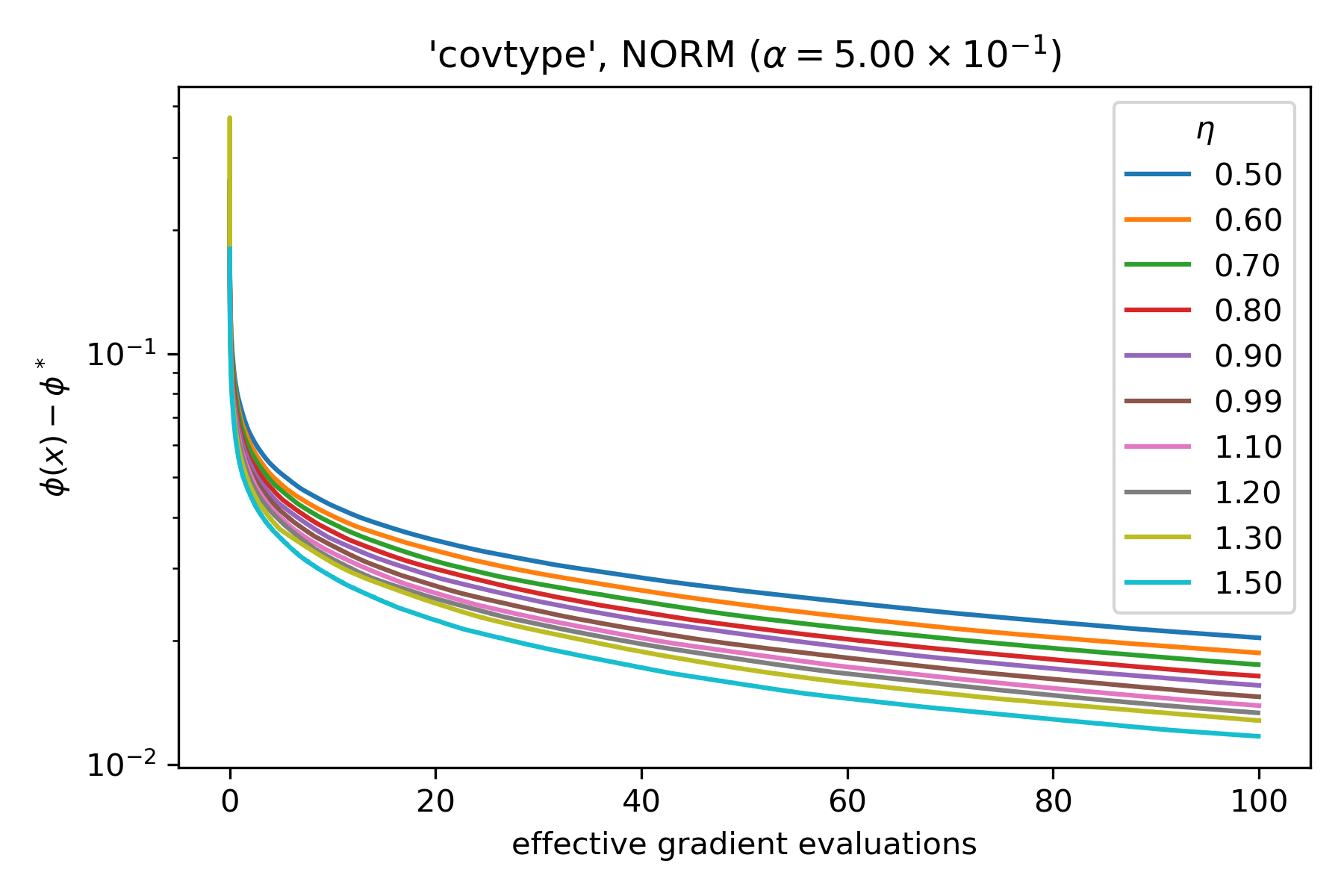} \\
	\includegraphics[width=0.45\linewidth]{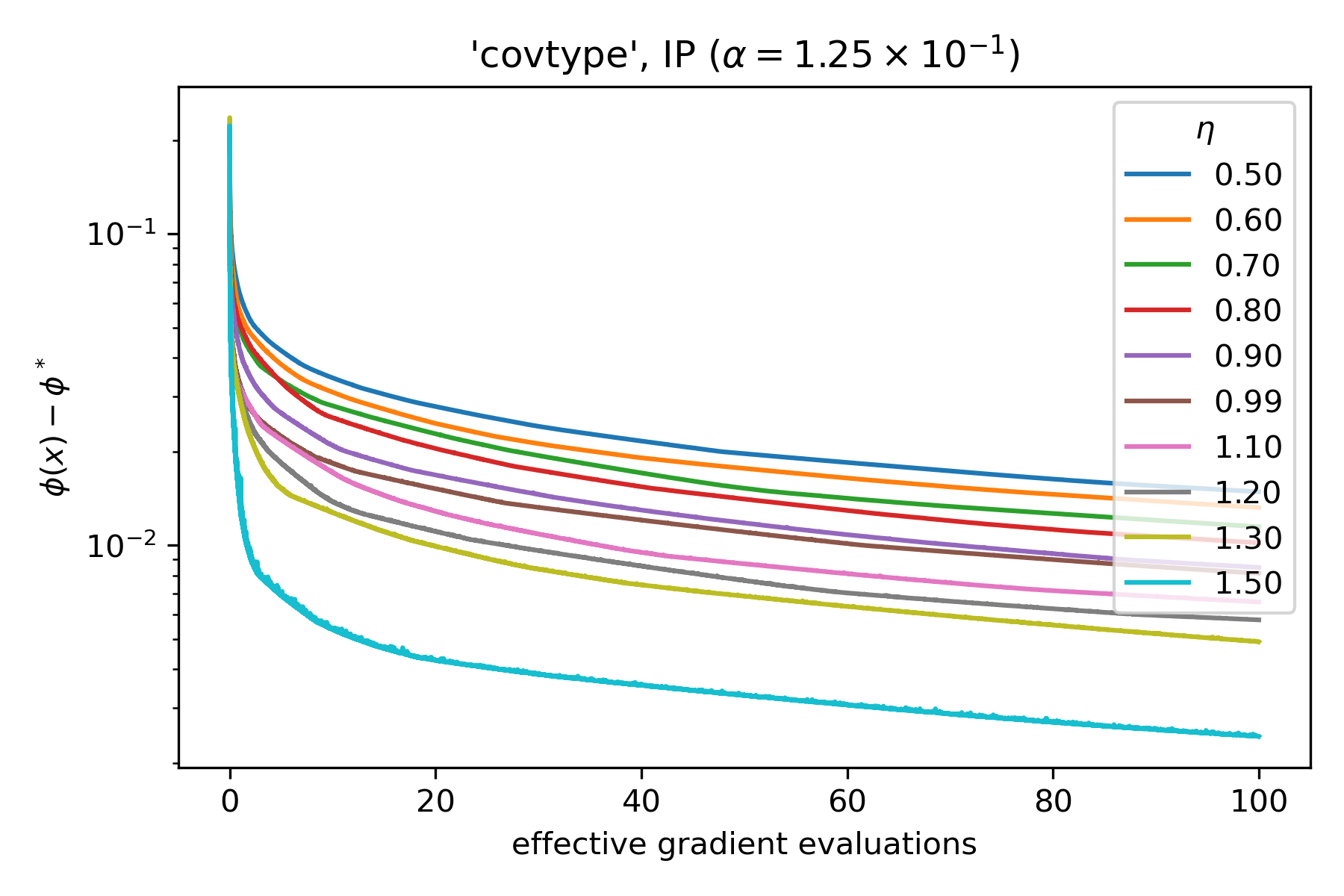}
    \includegraphics[width=0.45\linewidth]{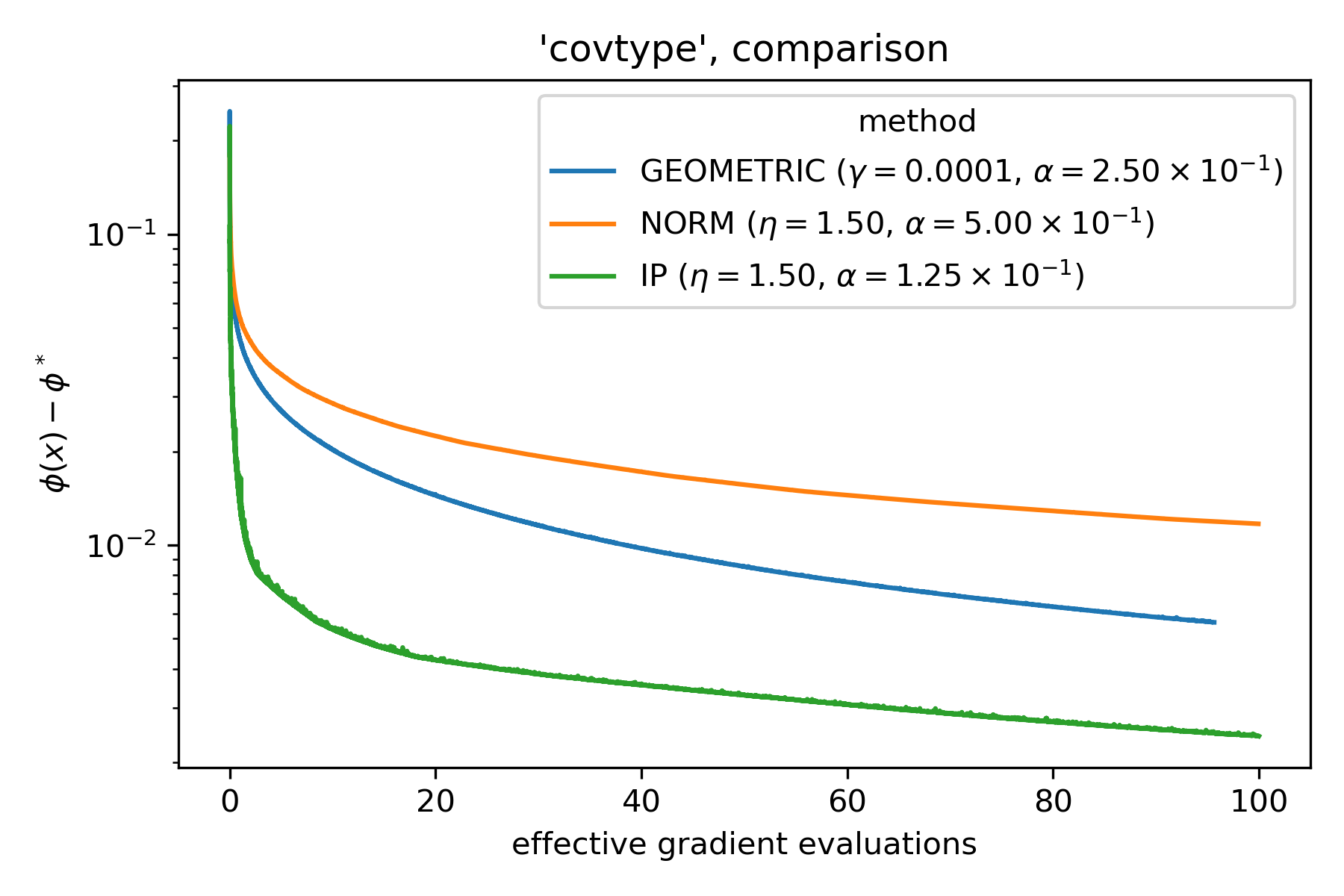} \\
    \caption{Optimality gap $\phi(x_k) - \phi^*$ against effective gradient evaluations on dataset \texttt{covtype}, with different strategies to control batch size: geometric increase (top left), norm test (top right), inner-product test (bottom left), and comparison between the best run for each method (bottom right).
    }
    \label{fig:covtype_func}
\end{figure}

\begin{figure}[htp]
    \centering
    \includegraphics[width=0.45\linewidth]{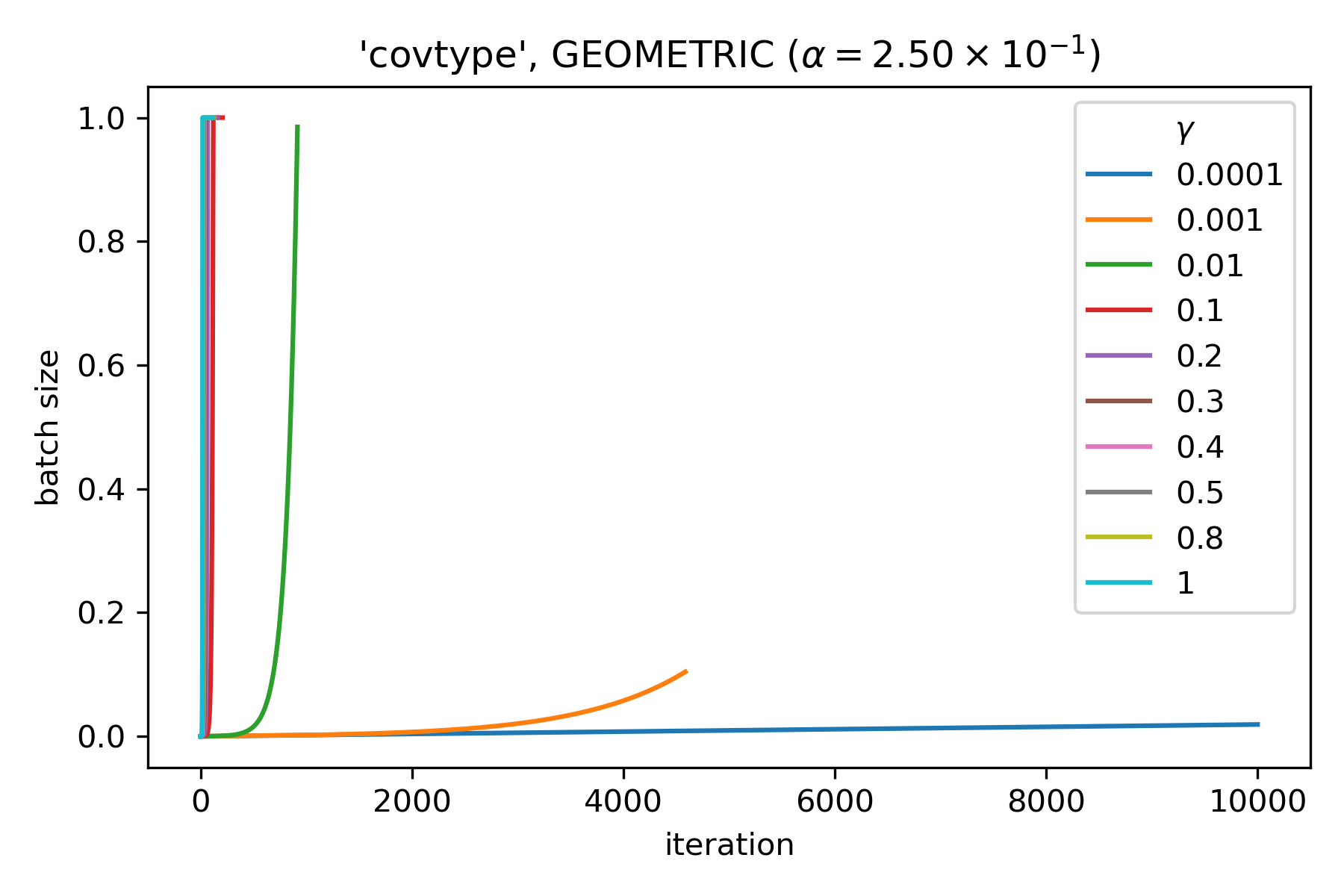}
	\includegraphics[width=0.45\linewidth]{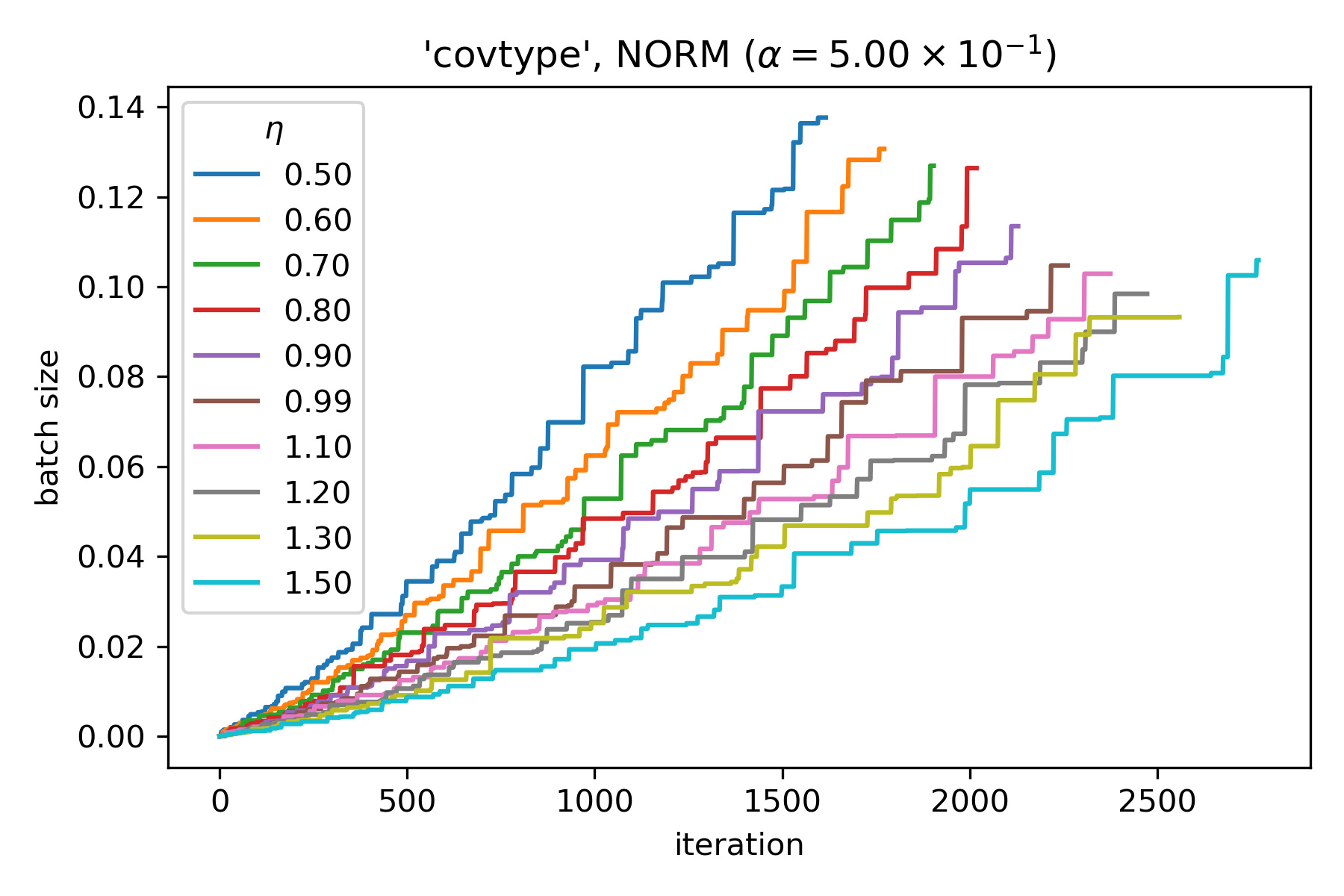} \\
	\includegraphics[width=0.45\linewidth]{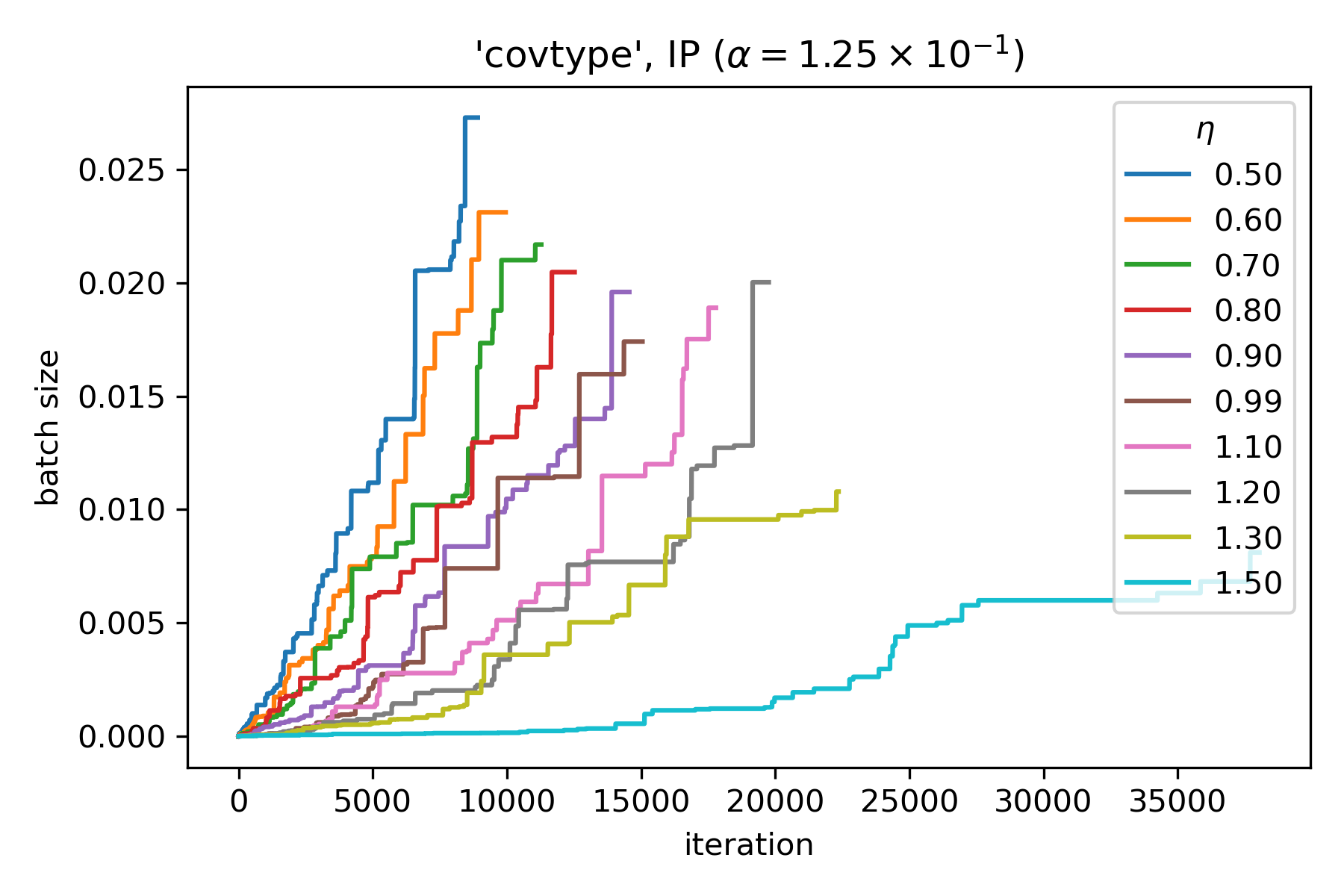}
    \includegraphics[width=0.45\linewidth]{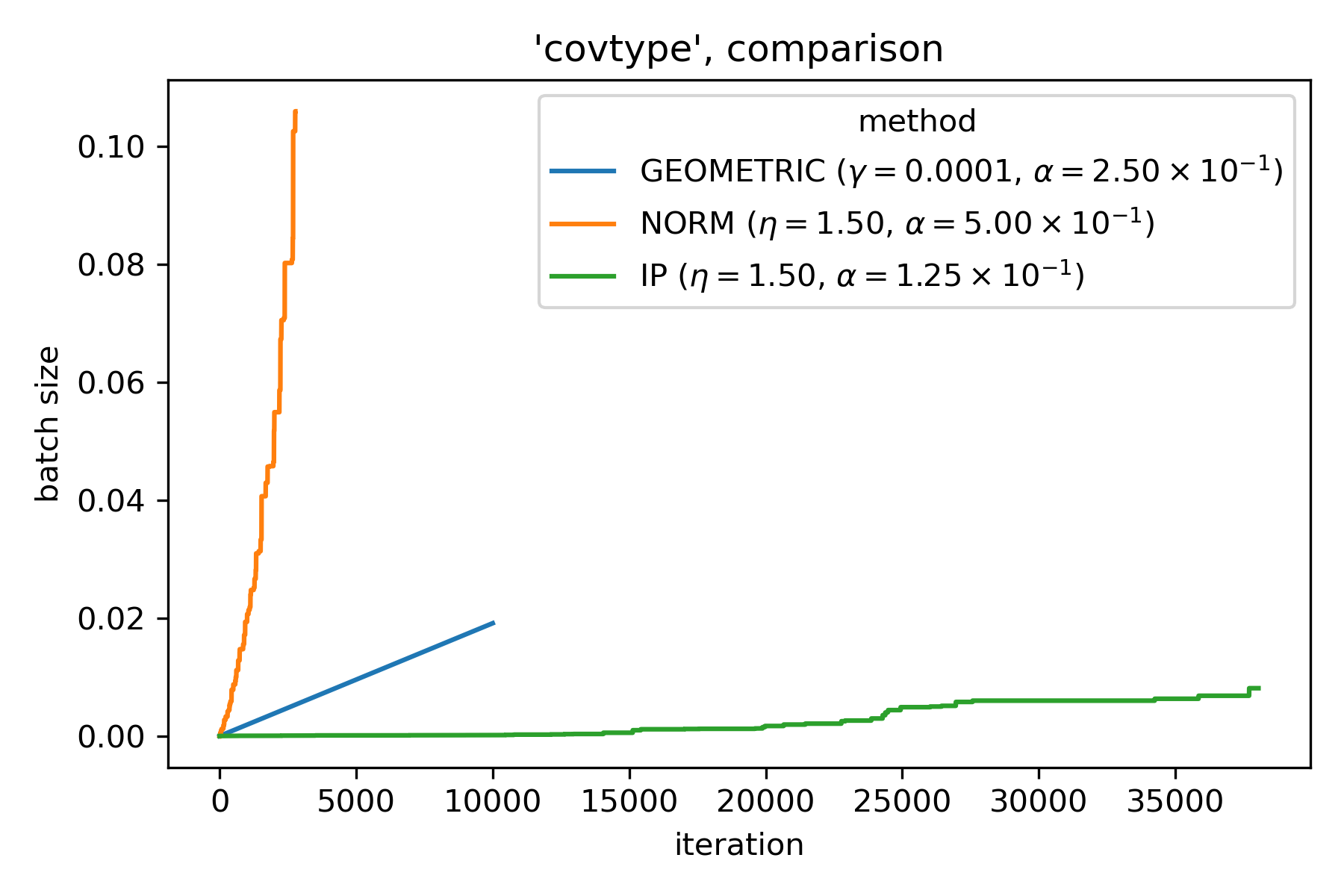} \\
    \caption{Batch size (as a fraction of total number of data points $N$) against iterations on dataset \texttt{covtype}, with different strategies to control batch size: geometric increase (top left), norm test (top right), inner-product test (bottom left), and comparison between the best run for each method (bottom right).
    }
    \label{fig:covtype_batchsizes}
\end{figure}

\begin{figure}[htp]
    \centering
    \includegraphics[width=0.45\linewidth]{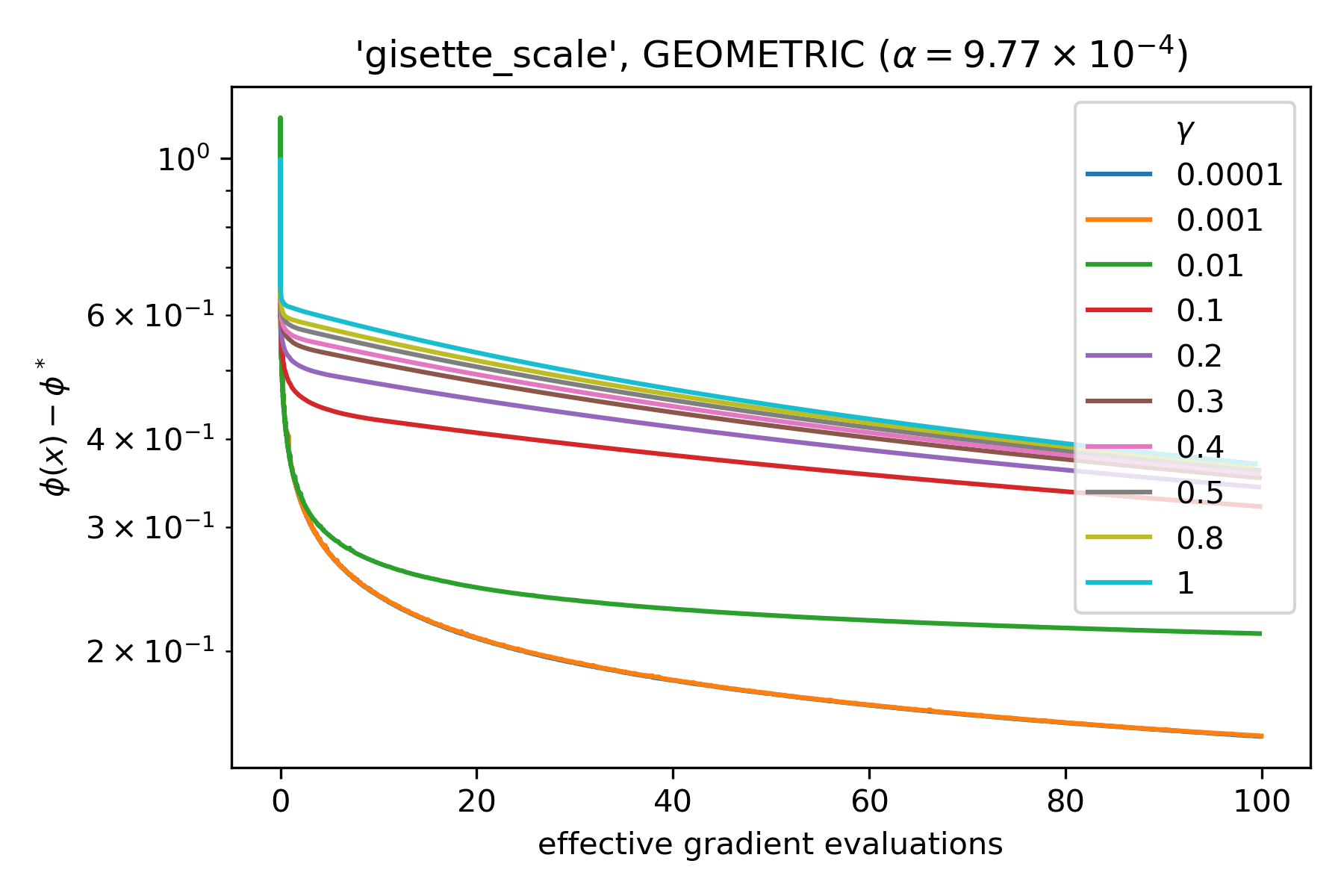}
	\includegraphics[width=0.45\linewidth]{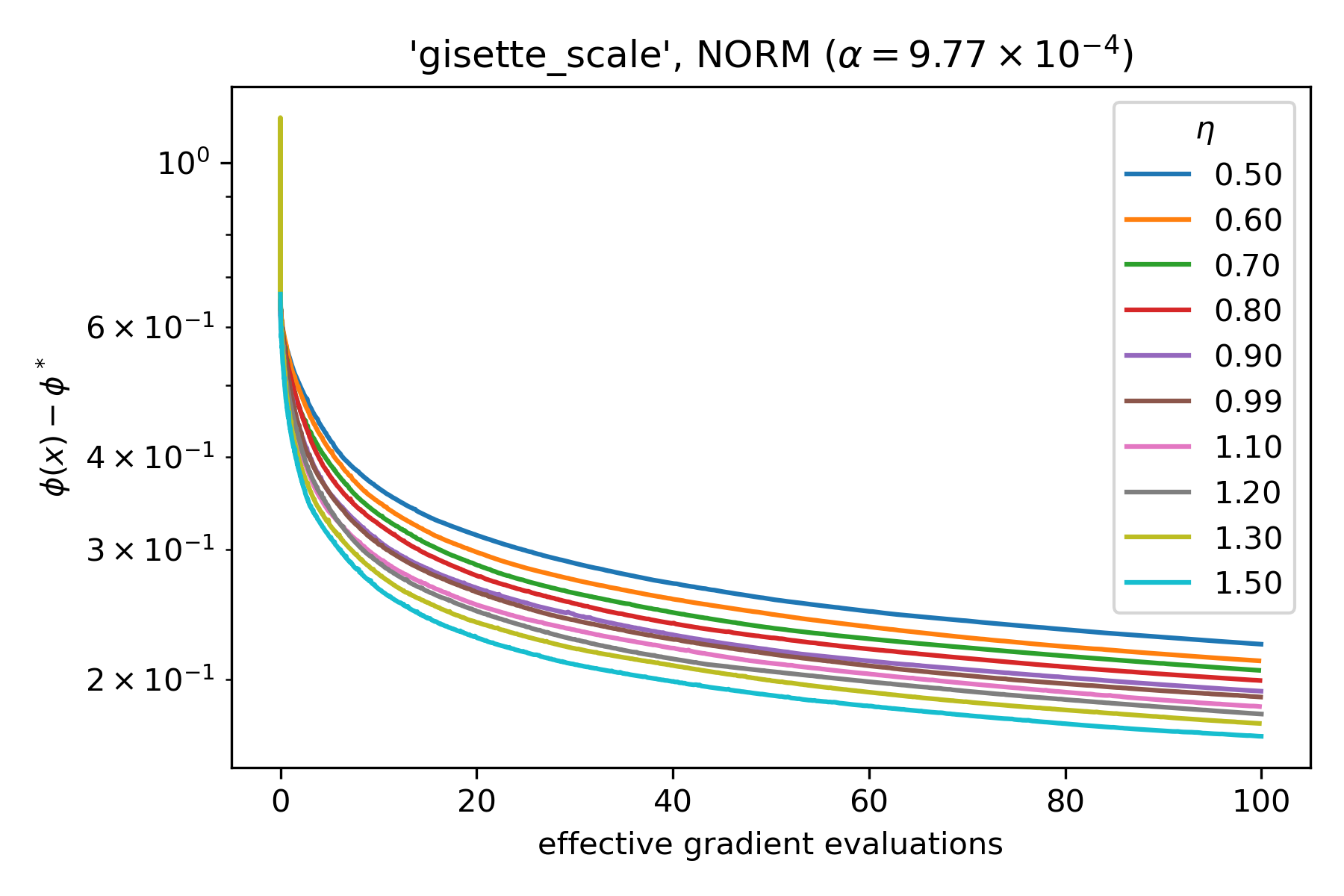} \\
	\includegraphics[width=0.45\linewidth]{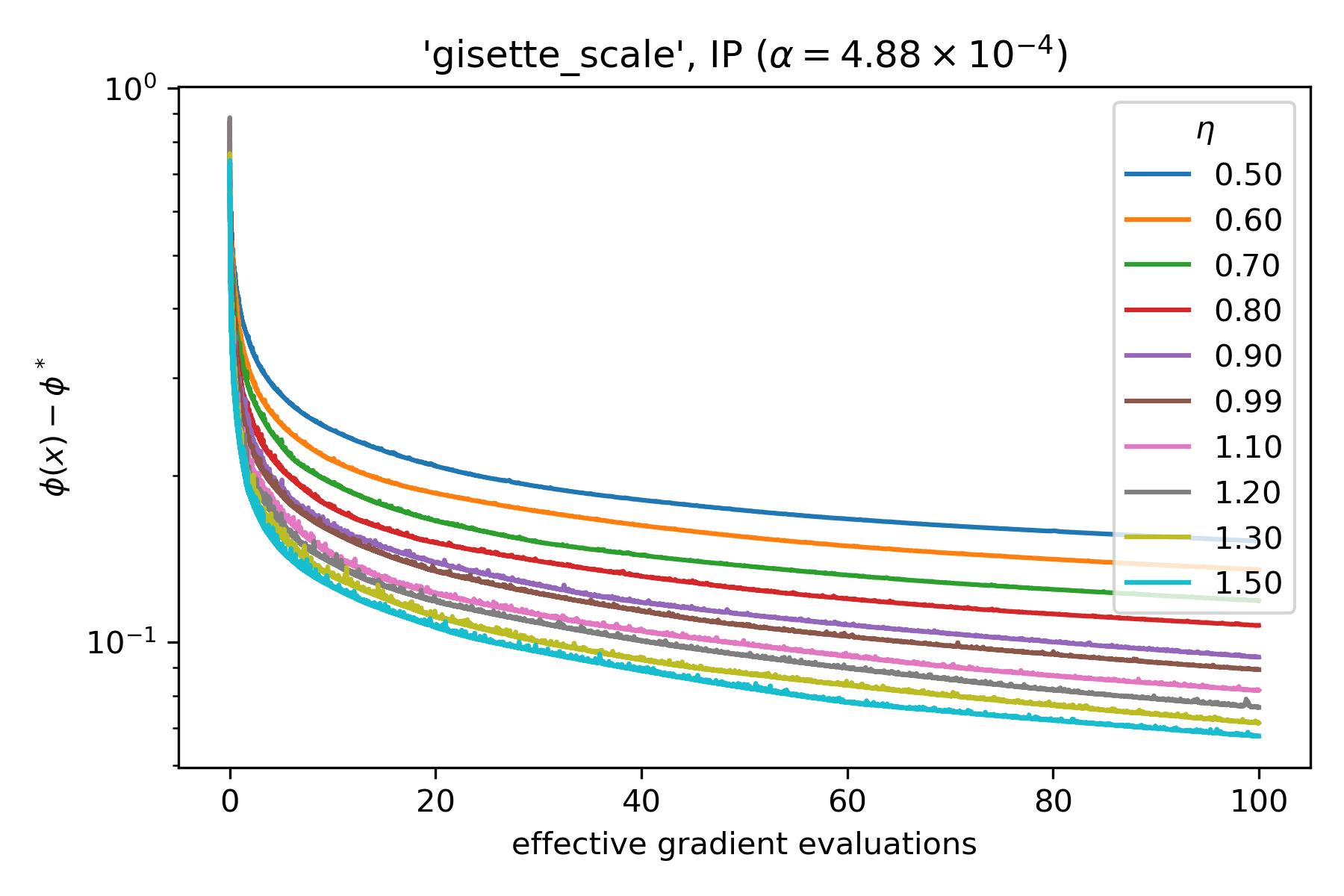}
    \includegraphics[width=0.45\linewidth]{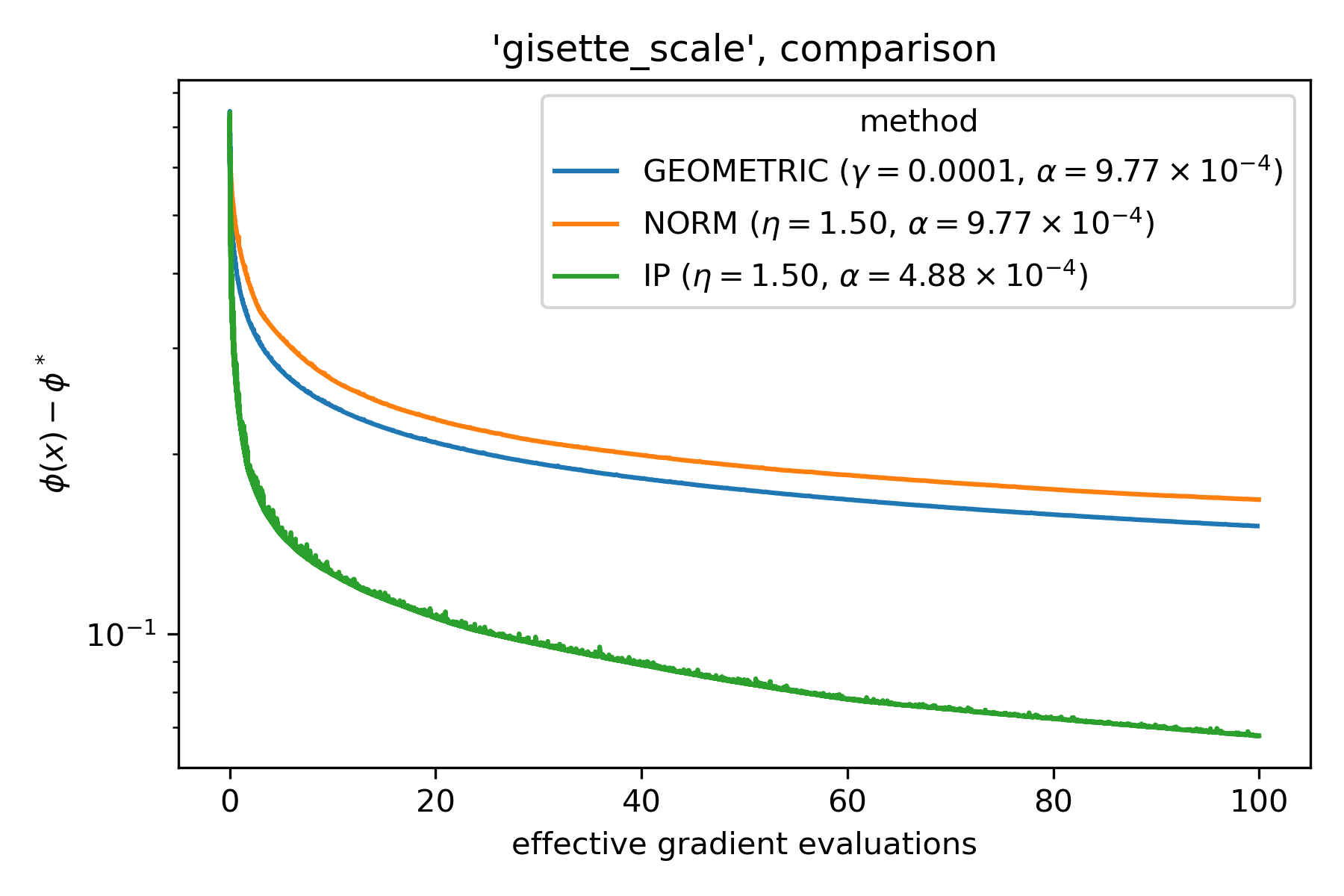} \\
    \caption{Optimality gap $\phi(x_k) - \phi^*$ against effective gradient evaluations on dataset \texttt{gisette\_scale}, with different strategies to control batch size: geometric increase (top left), norm test (top right), inner-product test (bottom left), and comparison between the best run for each method (bottom right).
    }
    \label{fig:gisette_scale_func}
\end{figure}

\begin{figure}[htp]
    \centering
    \includegraphics[width=0.45\linewidth]{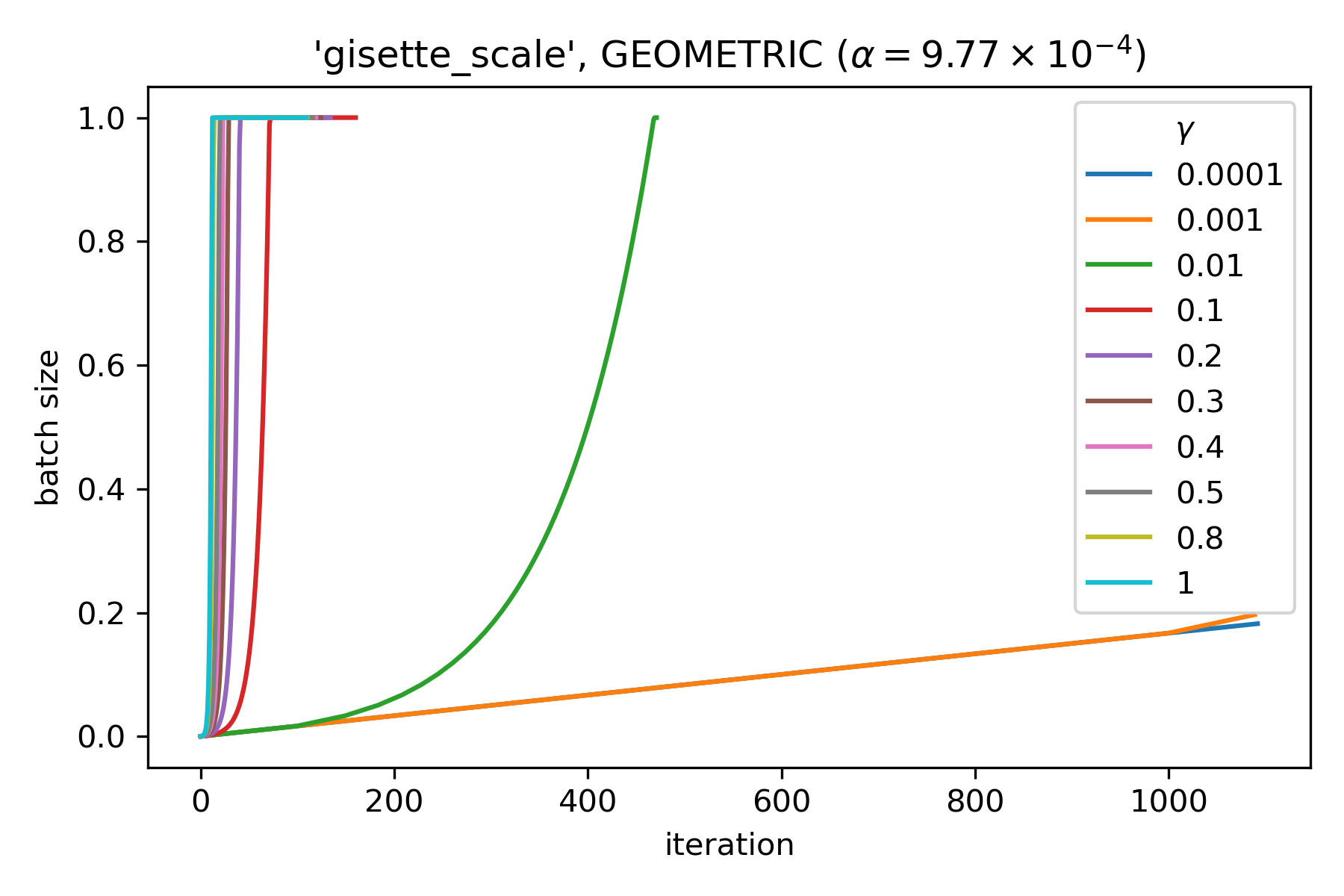}
	\includegraphics[width=0.45\linewidth]{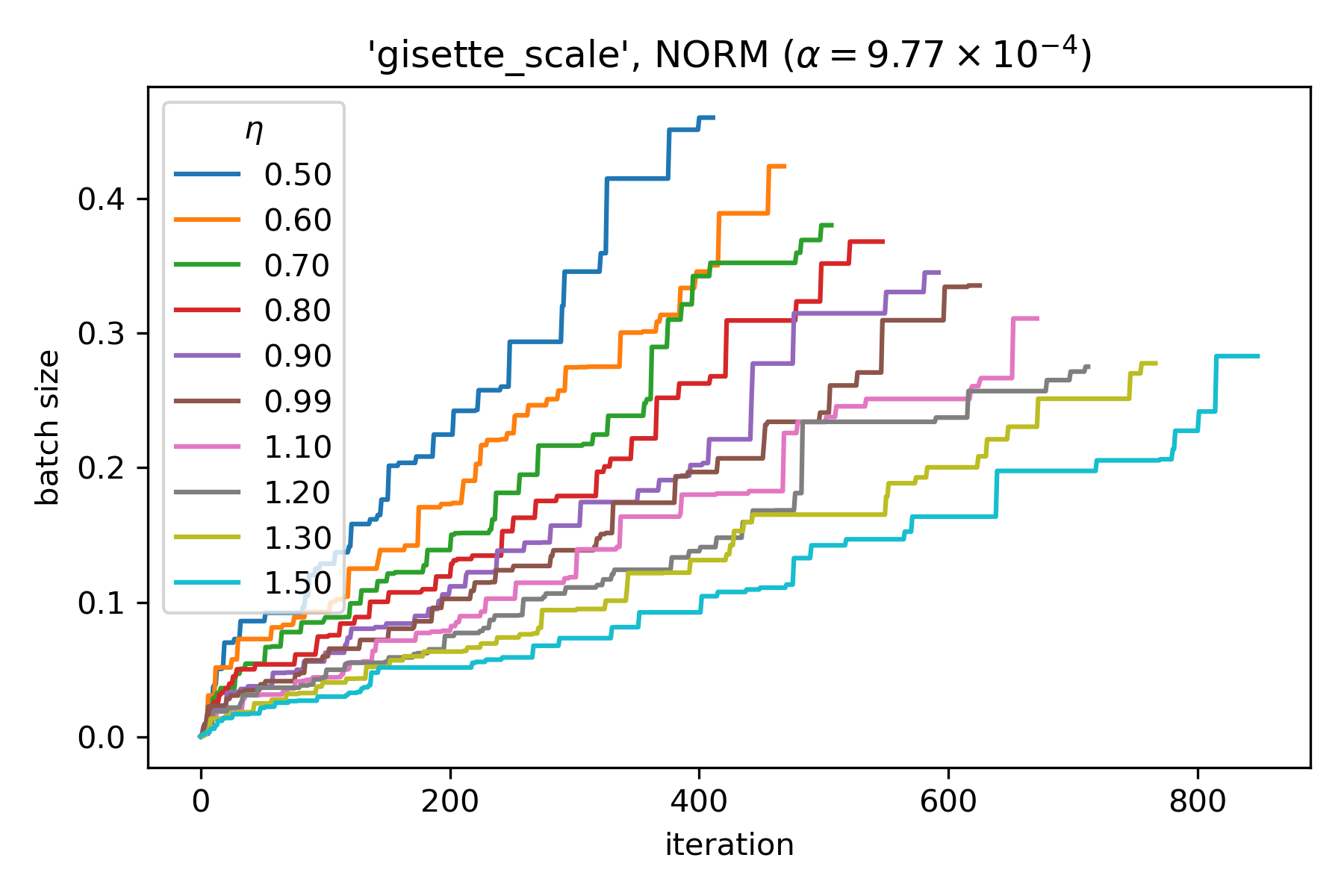} \\
	\includegraphics[width=0.45\linewidth]{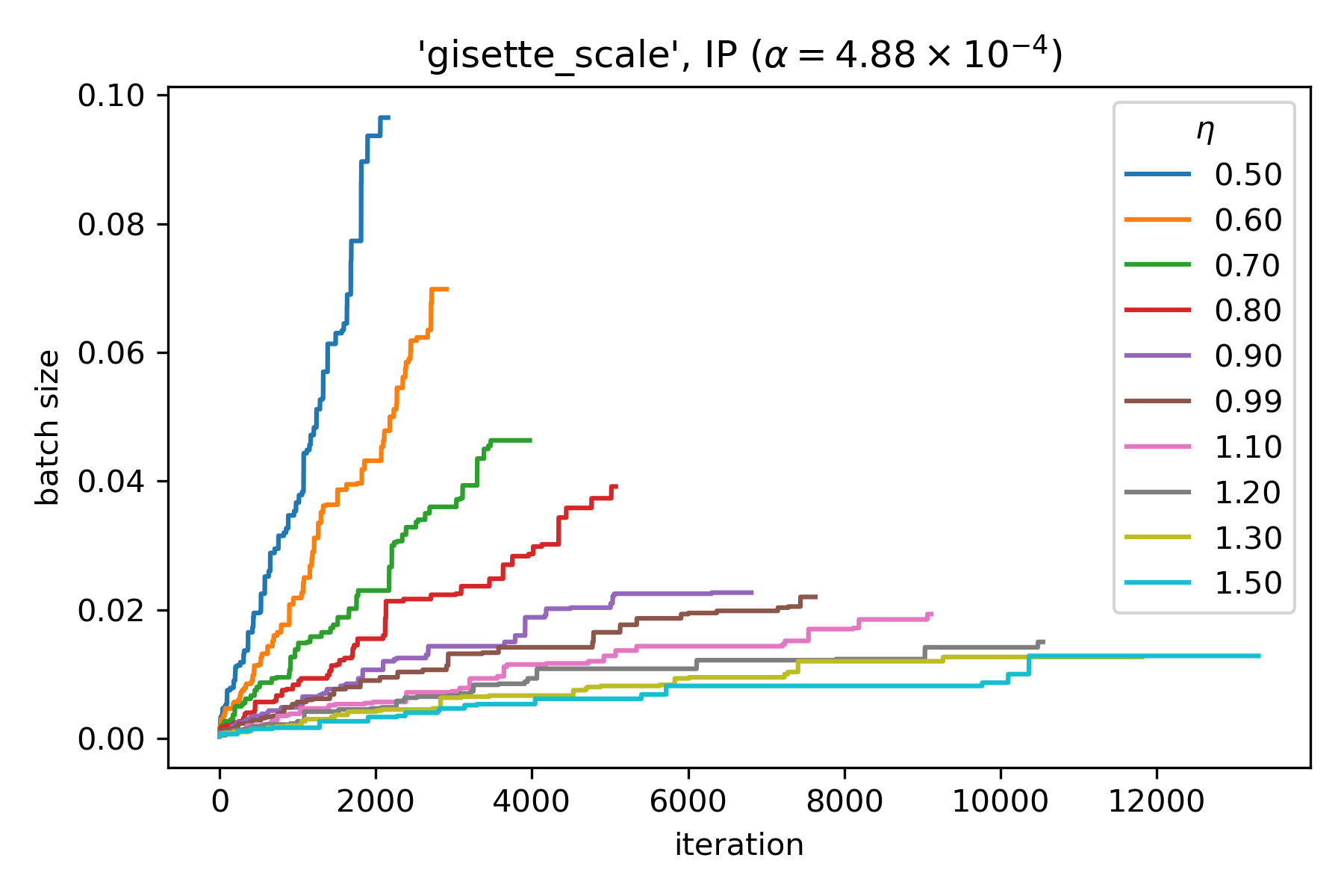}
    \includegraphics[width=0.45\linewidth]{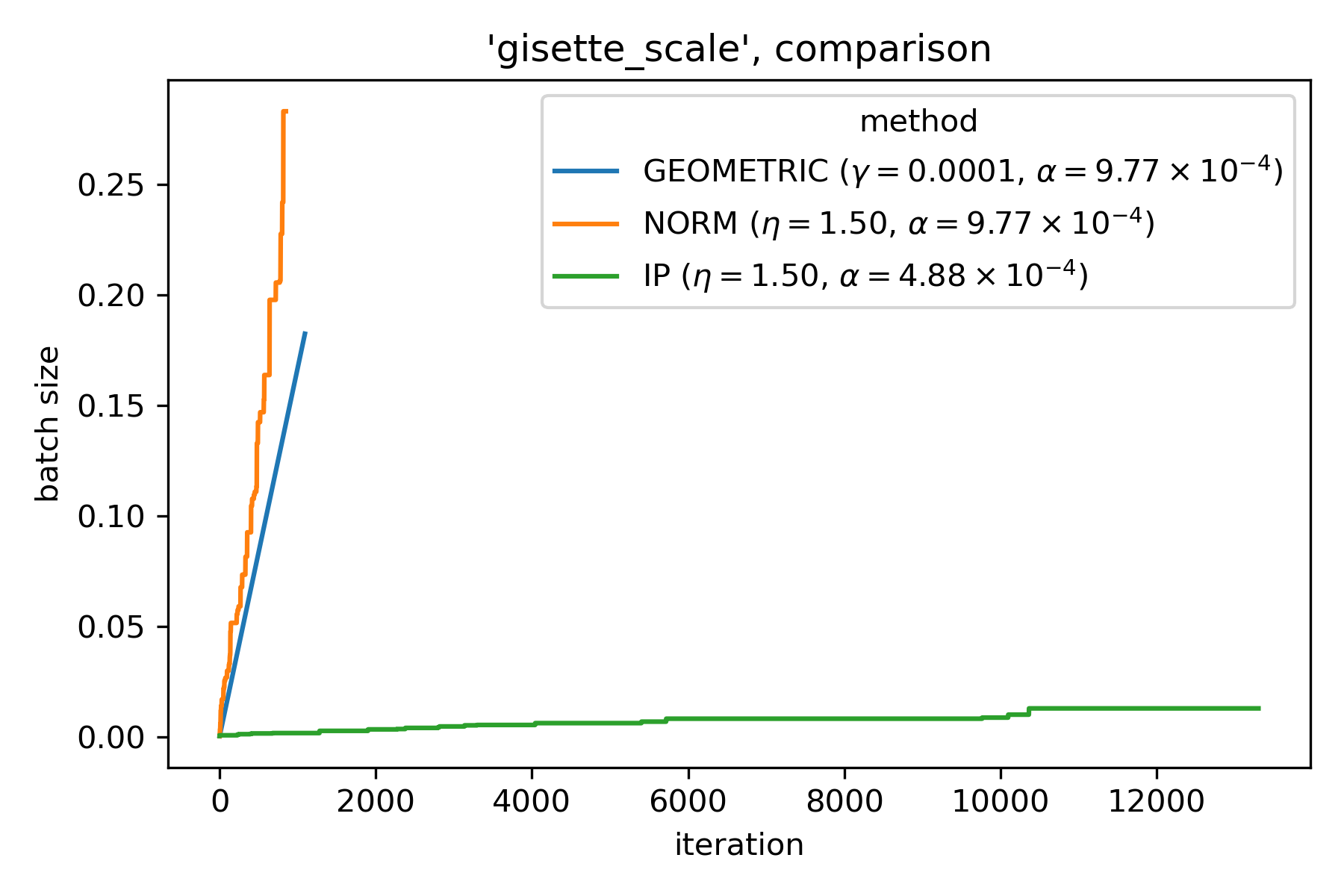} \\
    \caption{Batch size (as a fraction of total number of data points $N$) against iterations on dataset \texttt{gisette\_scale}, with different strategies to control batch size: geometric increase (top left), norm test (top right), inner-product test (bottom left), and comparison between the best run for each method (bottom right).
    }
    \label{fig:gisette_scale_batchsizes}
\end{figure}

\begin{figure}[htp]
    \centering
    \includegraphics[width=0.45\linewidth]{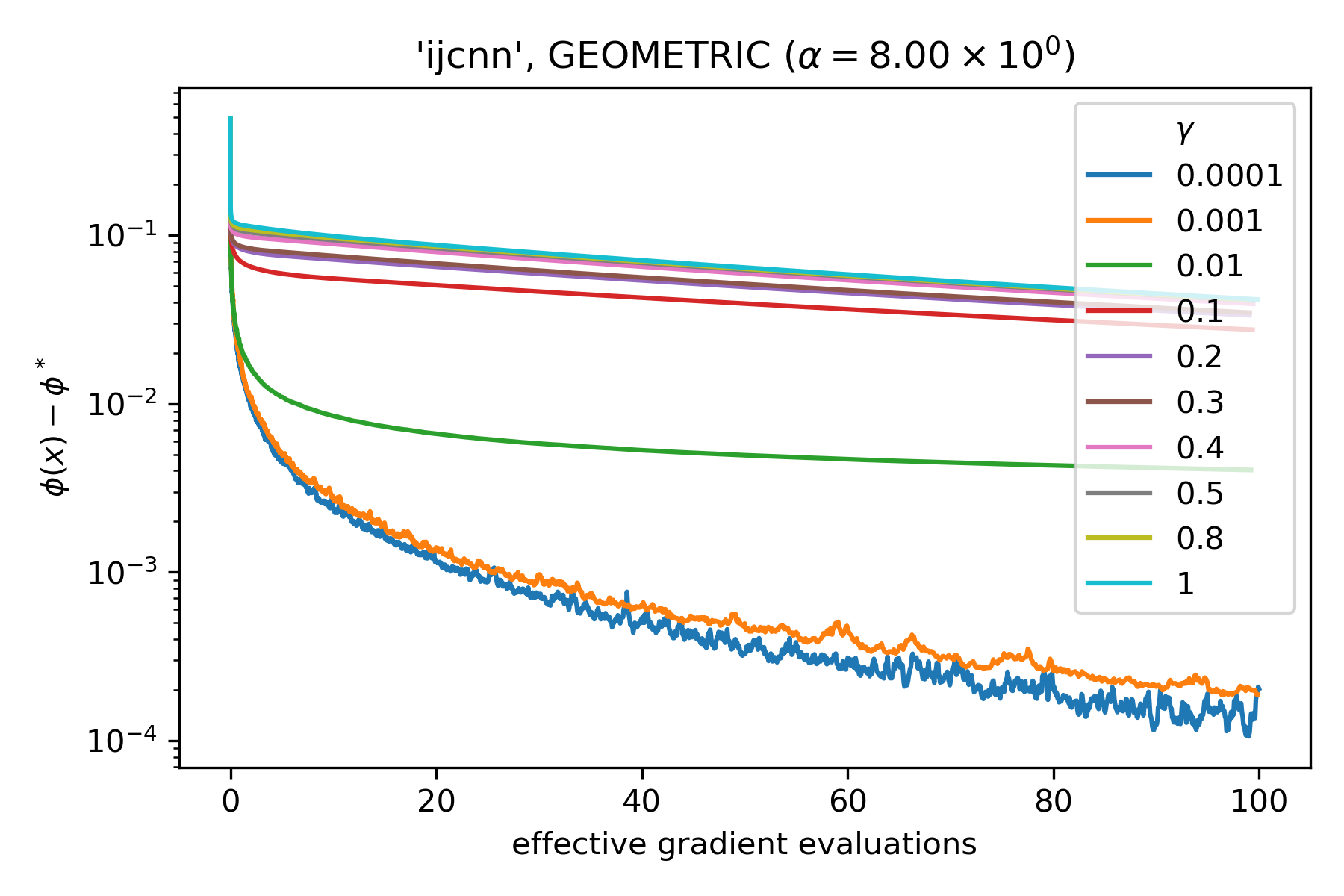}
	\includegraphics[width=0.45\linewidth]{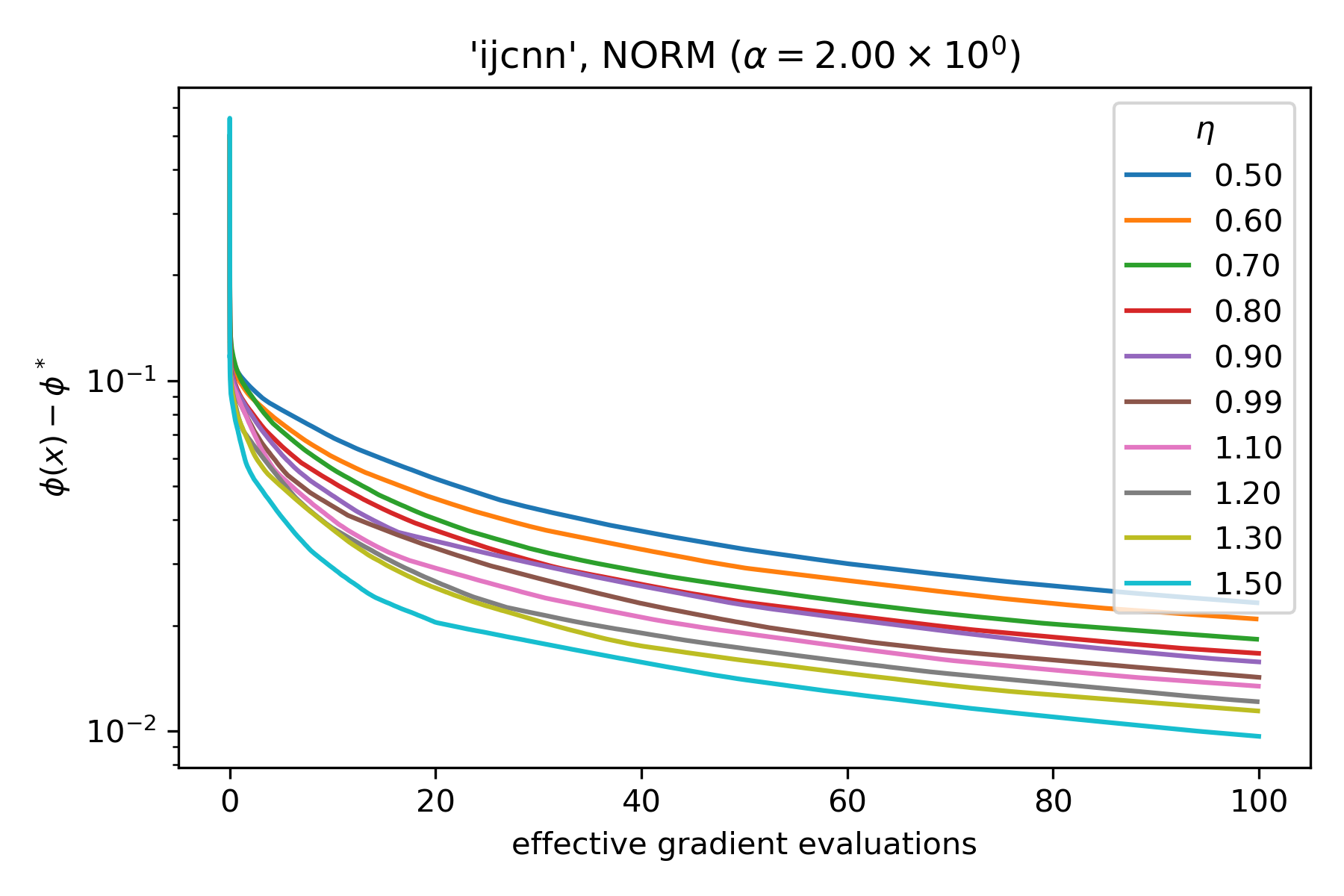} \\
	\includegraphics[width=0.45\linewidth]{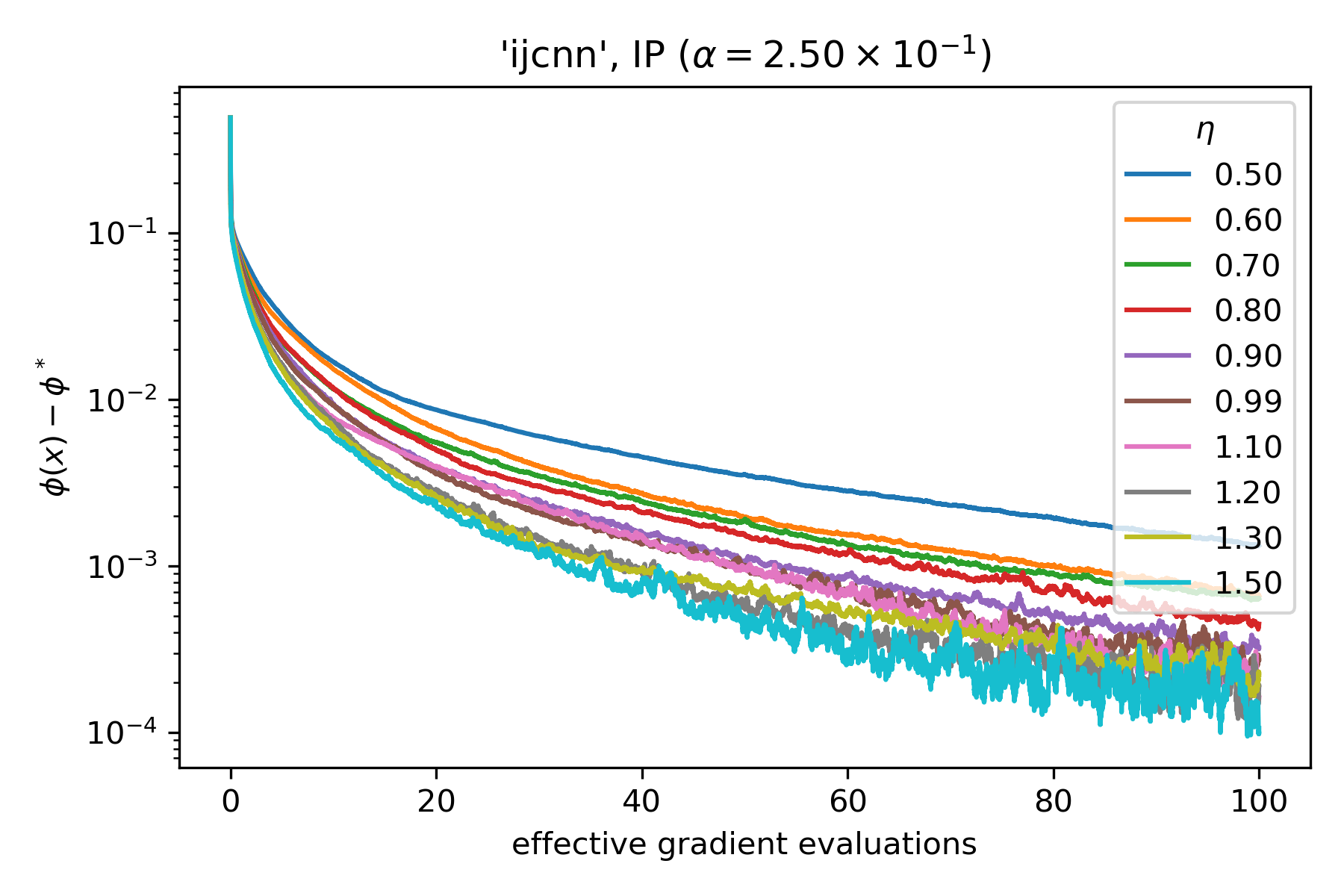}
    \includegraphics[width=0.45\linewidth]{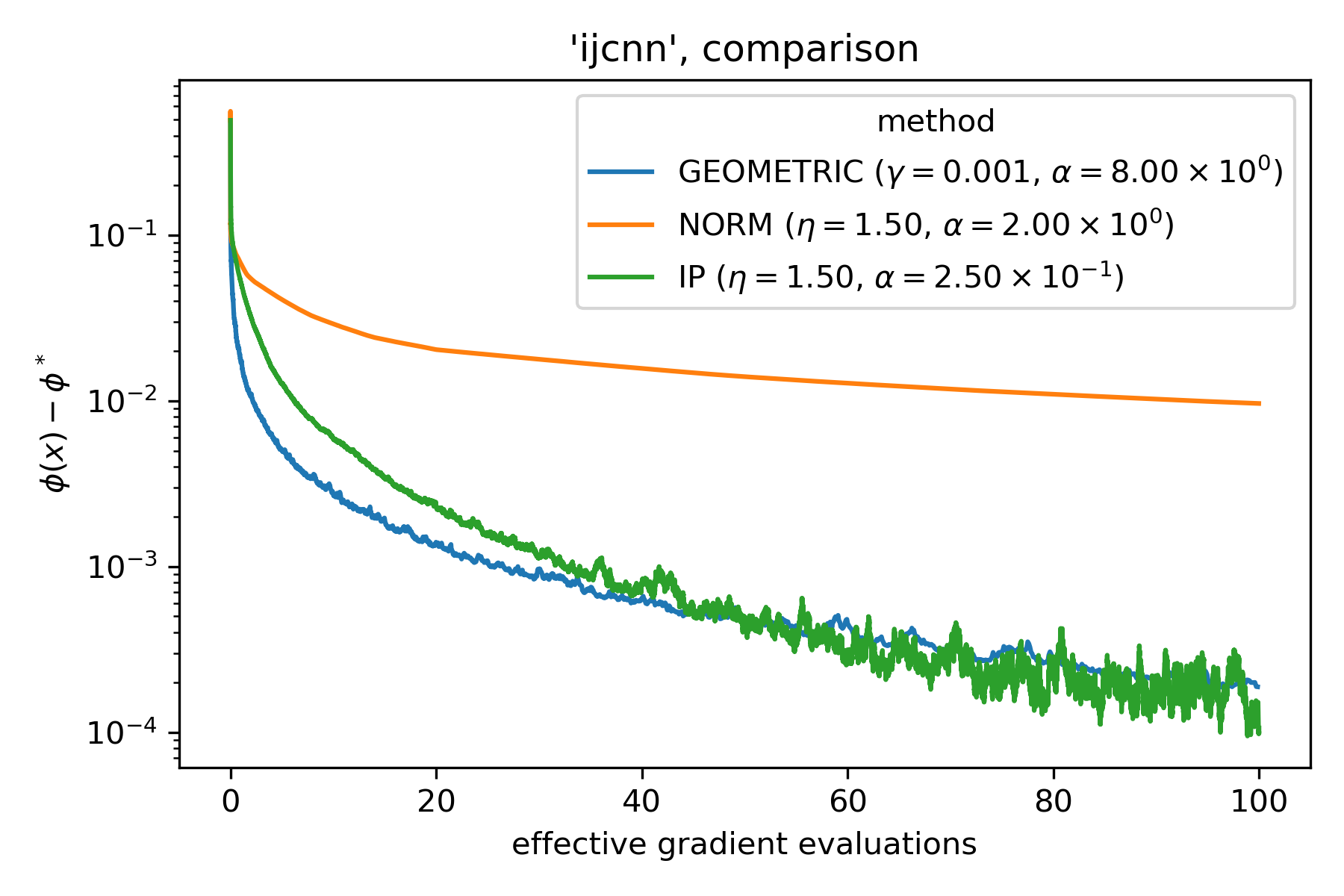} \\
    \caption{Optimality gap $\phi(x_k) - \phi^*$ against effective gradient evaluations on dataset \texttt{ijcnn}, with different strategies to control batch size: geometric increase (top left), norm test (top right), inner-product test (bottom left), and comparison between the best run for each method (bottom right).
    }
    \label{fig:ijcnn_func}
\end{figure}

\begin{figure}[htp]
    \centering
    \includegraphics[width=0.45\linewidth]{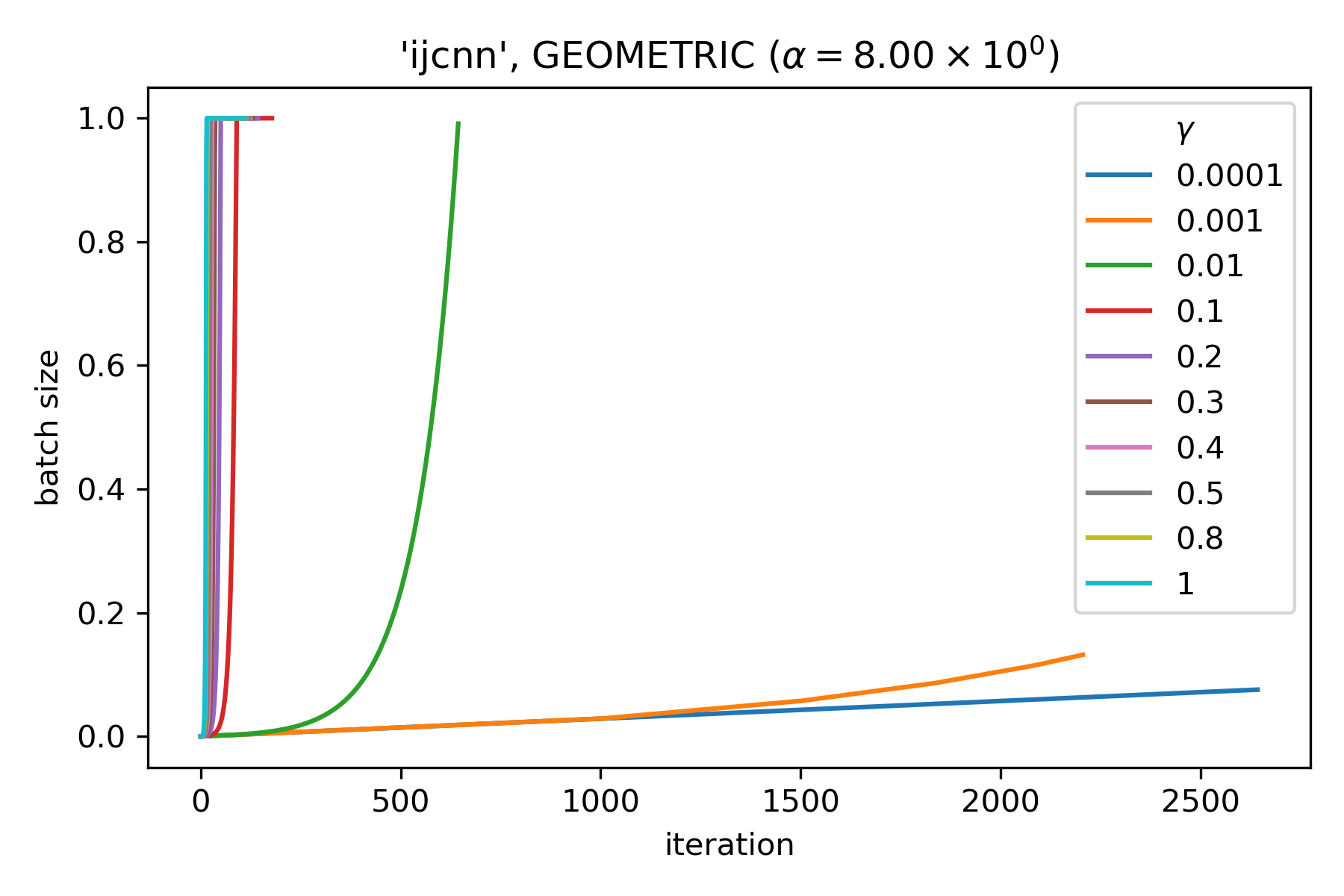}
	\includegraphics[width=0.45\linewidth]{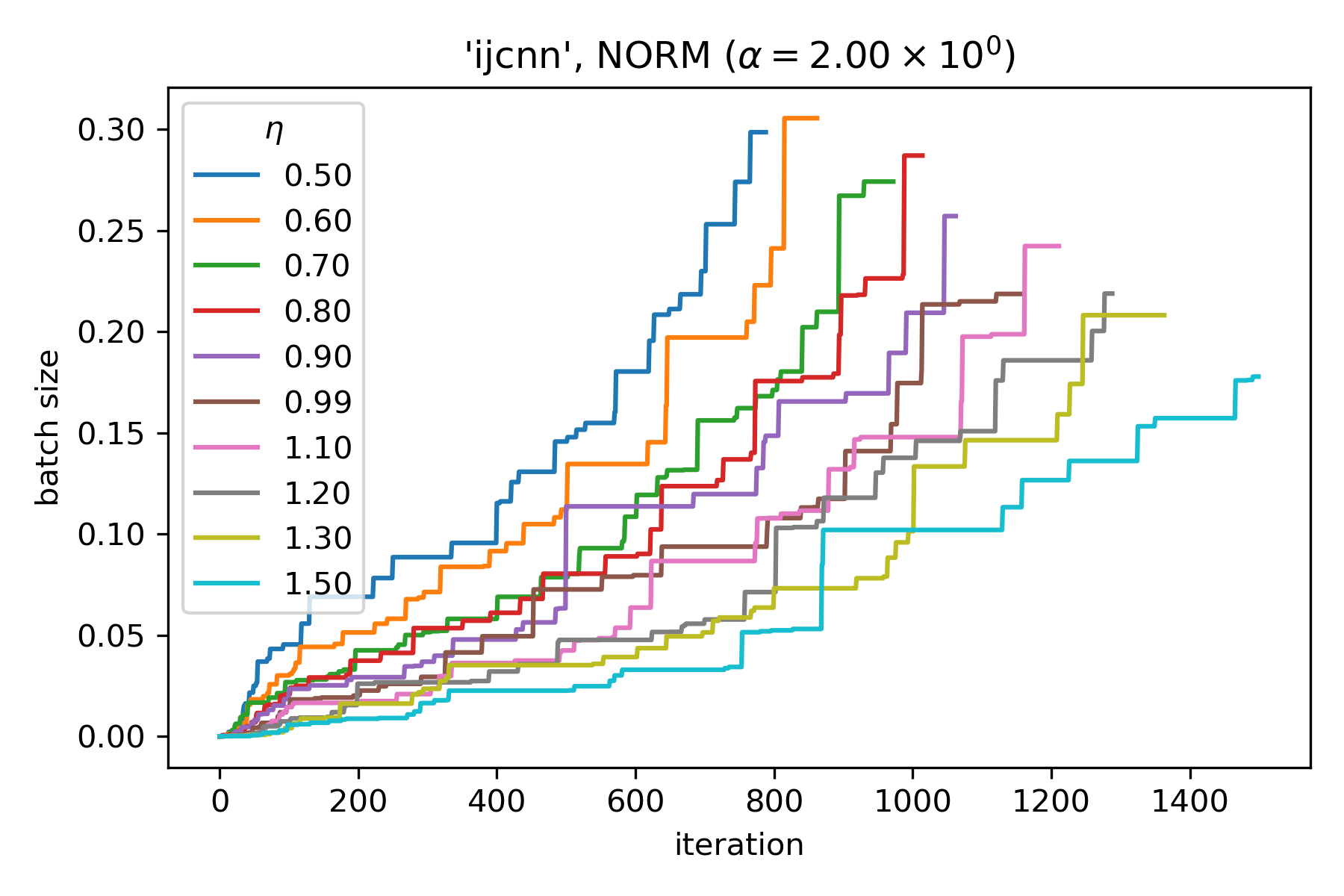} \\
	\includegraphics[width=0.45\linewidth]{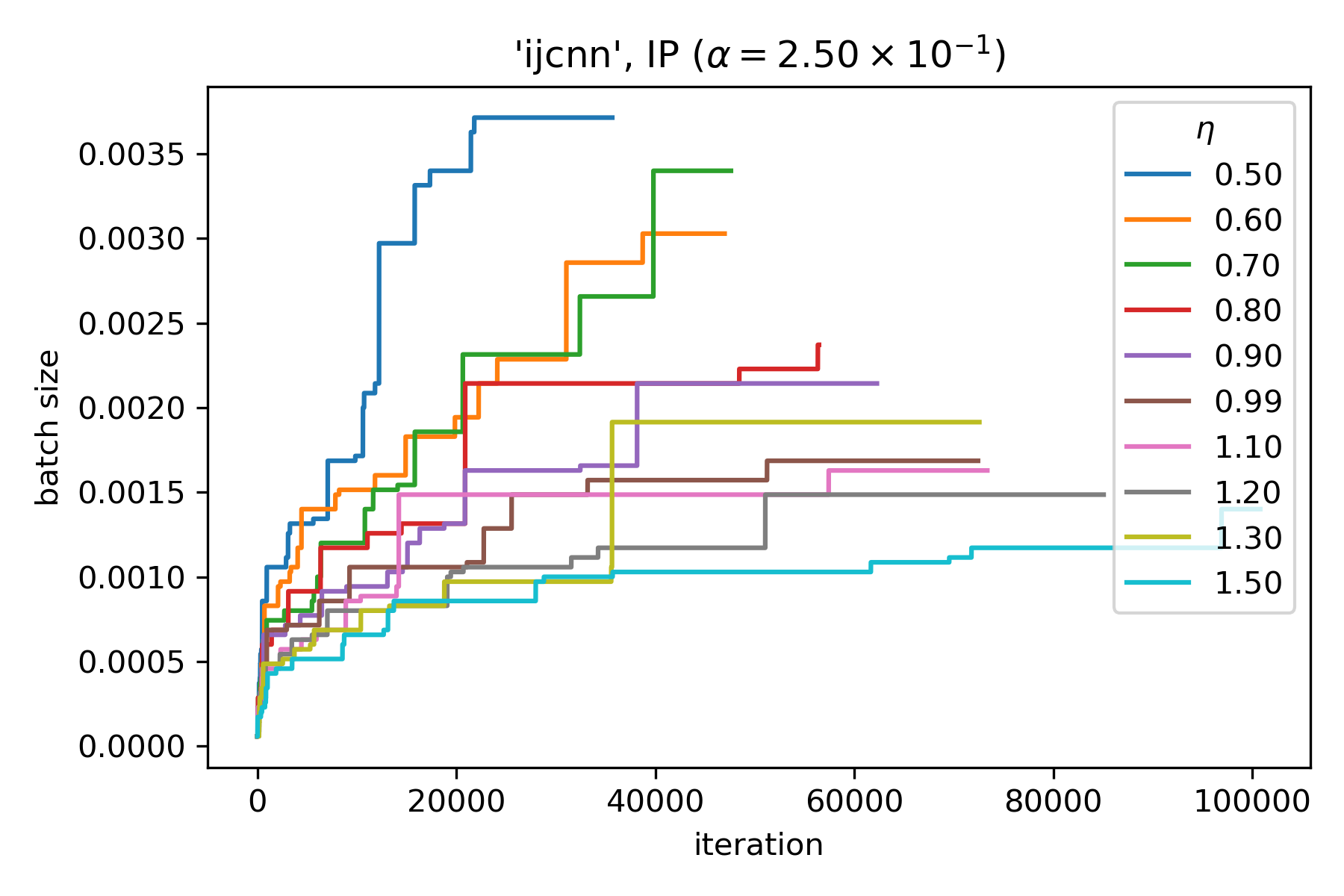}
    \includegraphics[width=0.45\linewidth]{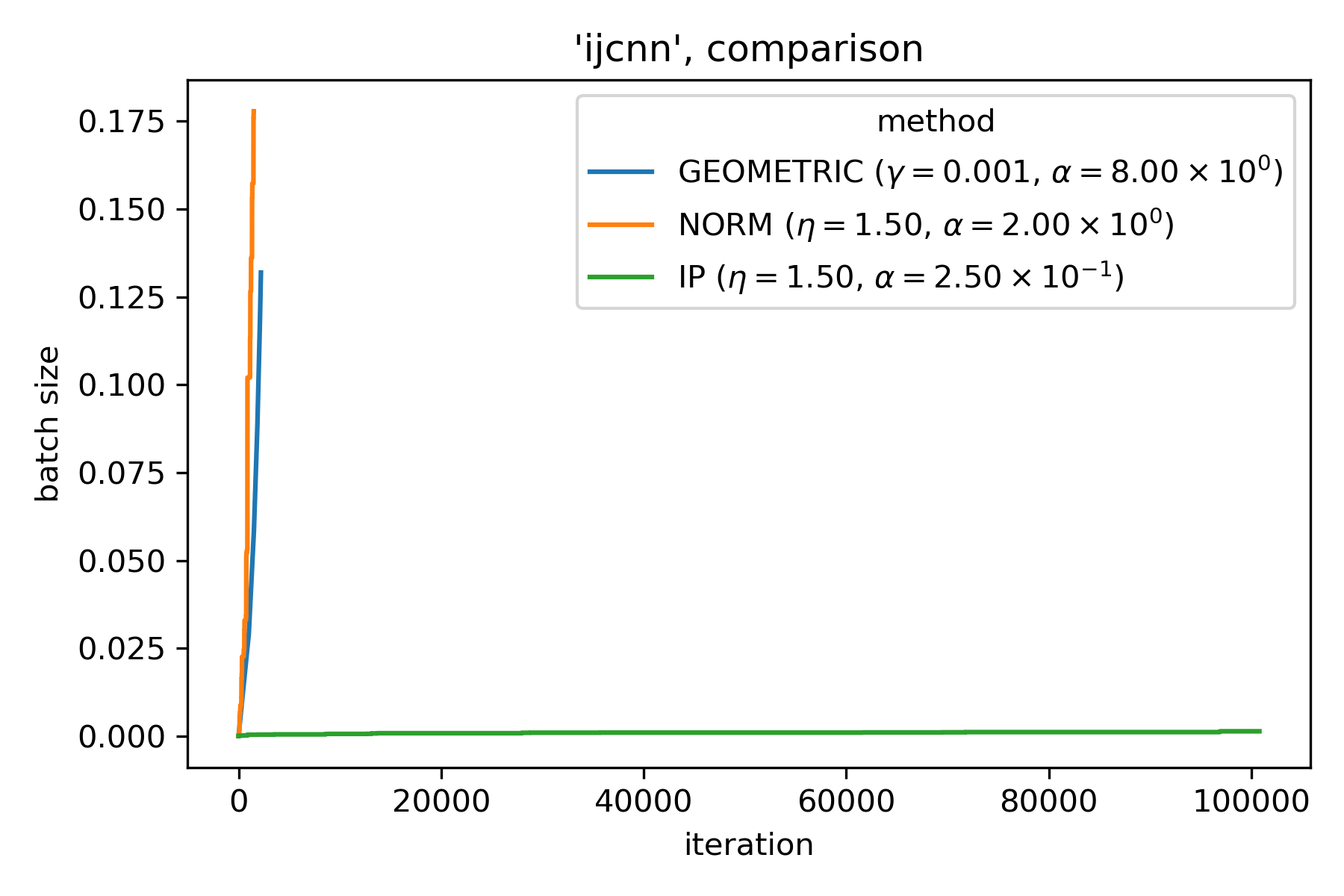} \\
    \caption{Batch size (as a fraction of total number of data points $N$) against iterations on dataset \texttt{ijcnn}, with different strategies to control batch size: geometric increase (top left), norm test (top right), inner-product test (bottom left), and comparison between the best run for each method (bottom right).
    }
    \label{fig:ijcnn_batchsizes}
\end{figure}

\begin{figure}[htp]
    \centering
    \includegraphics[width=0.45\linewidth]{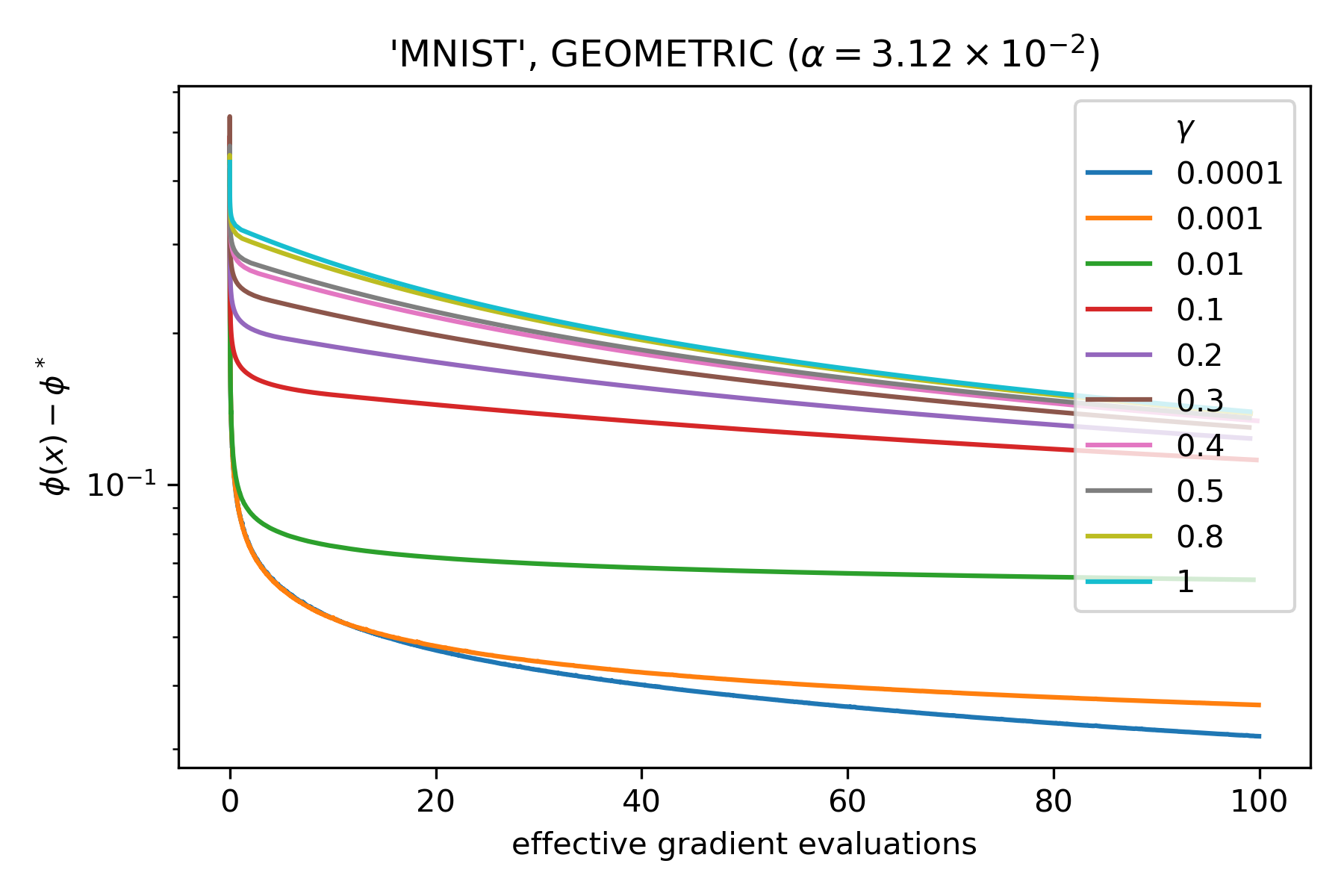}
	\includegraphics[width=0.45\linewidth]{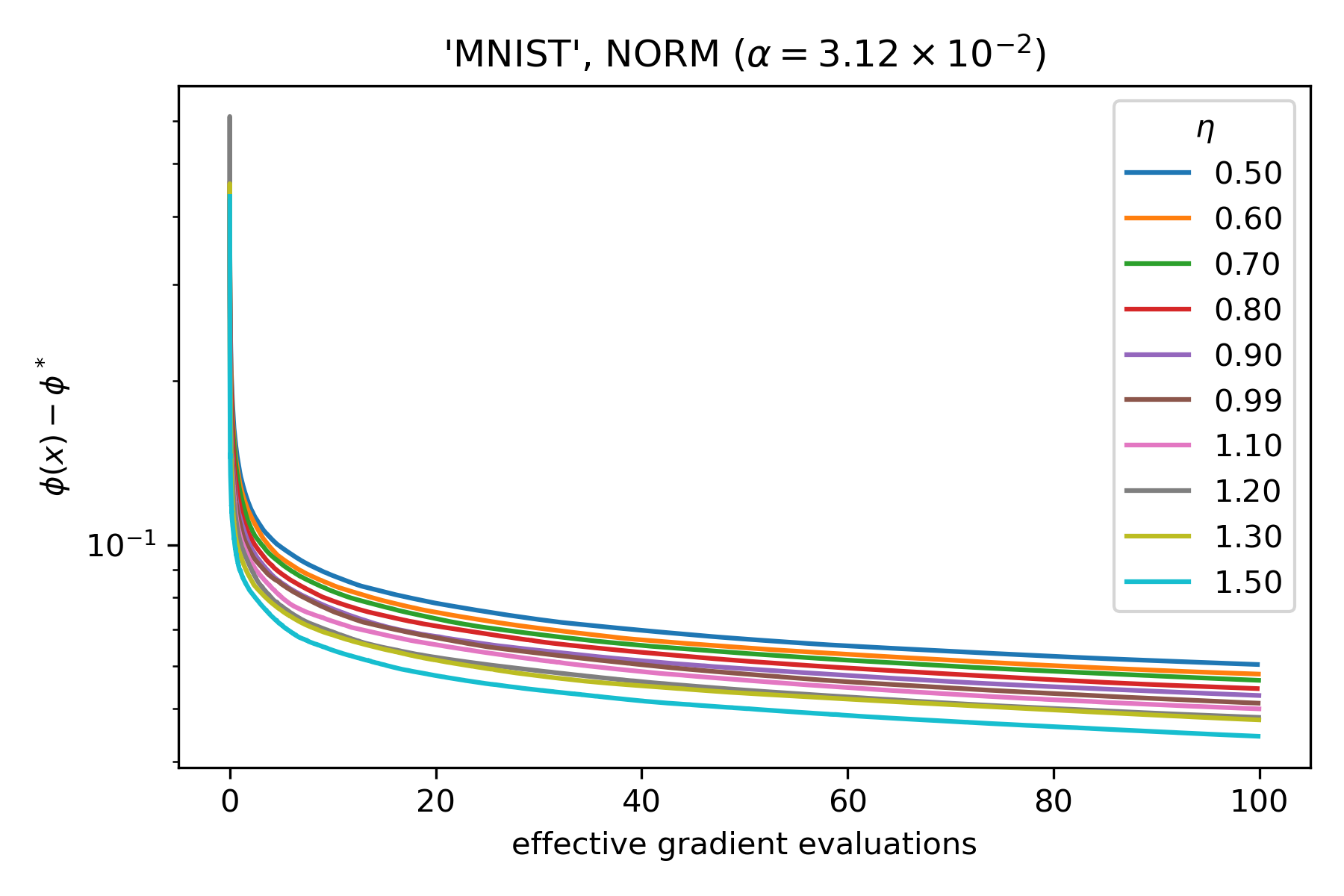} \\
	\includegraphics[width=0.45\linewidth]{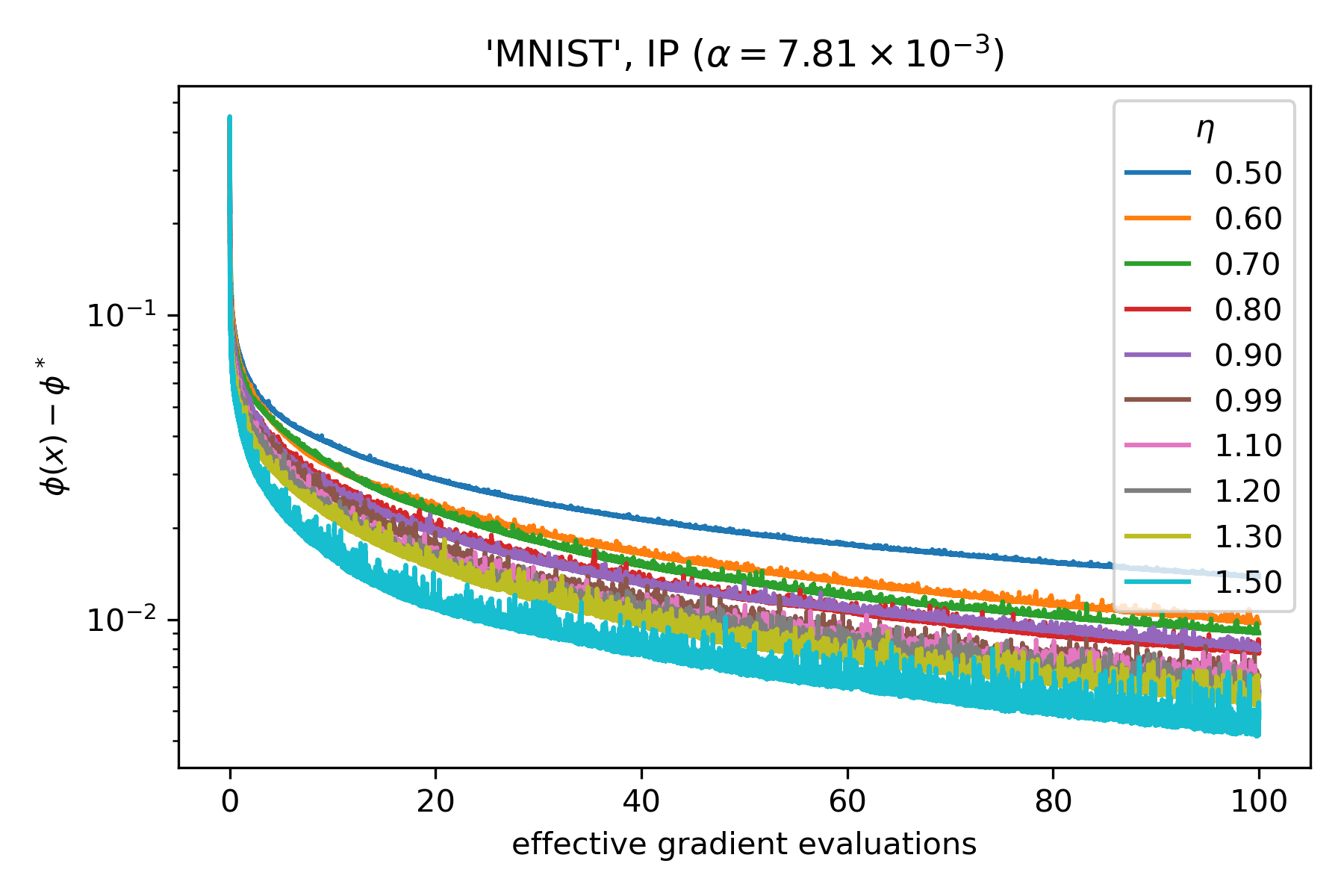}
    \includegraphics[width=0.45\linewidth]{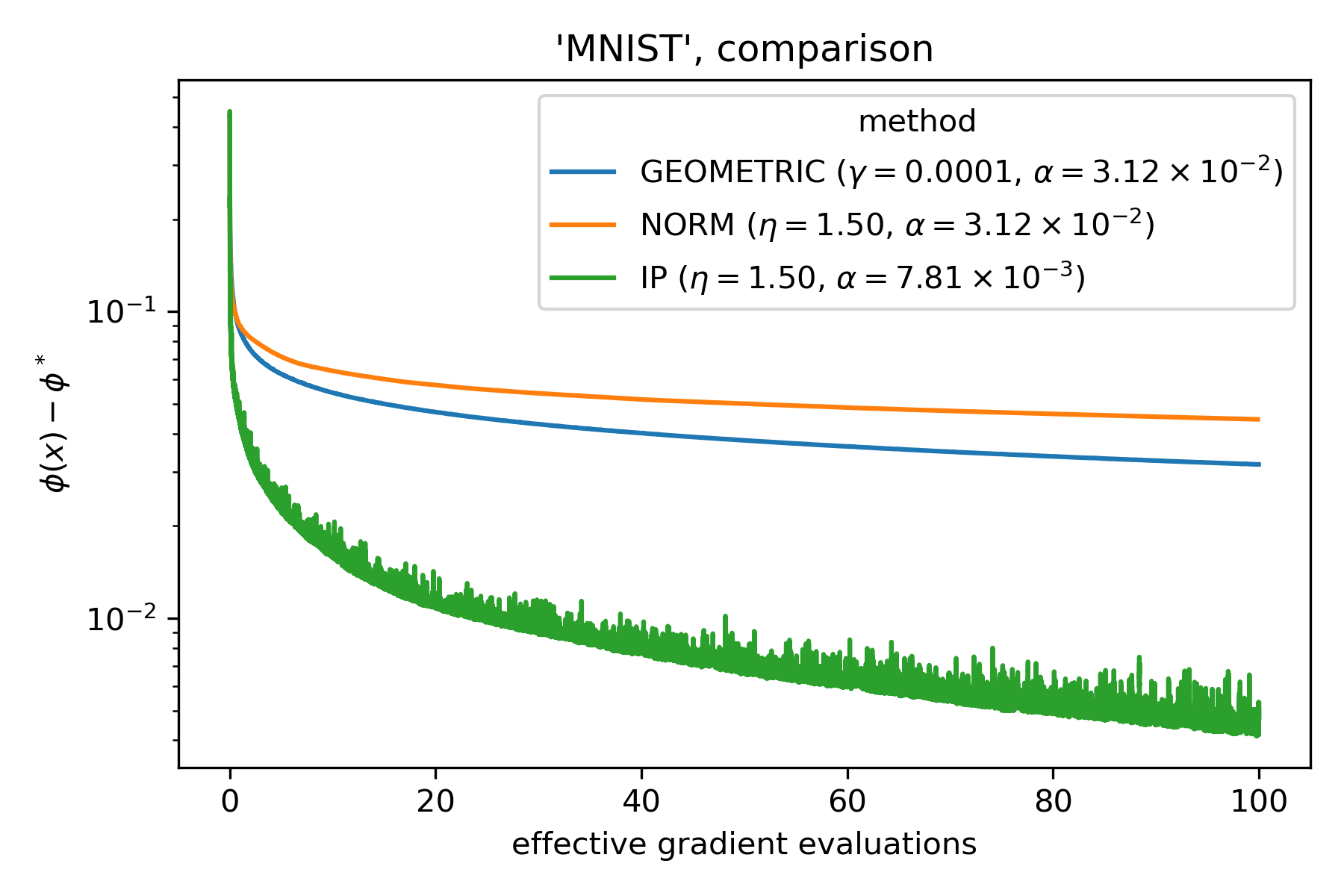} \\
    \caption{Optimality gap $\phi(x_k) - \phi^*$ against effective gradient evaluations on dataset \texttt{MNIST}, with different strategies to control batch size: geometric increase (top left), norm test (top right), inner-product test (bottom left), and comparison between the best run for each method (bottom right).
    }
    \label{fig:MNIST_func}
\end{figure}

\begin{figure}[htp]
    \centering
    \includegraphics[width=0.45\linewidth]{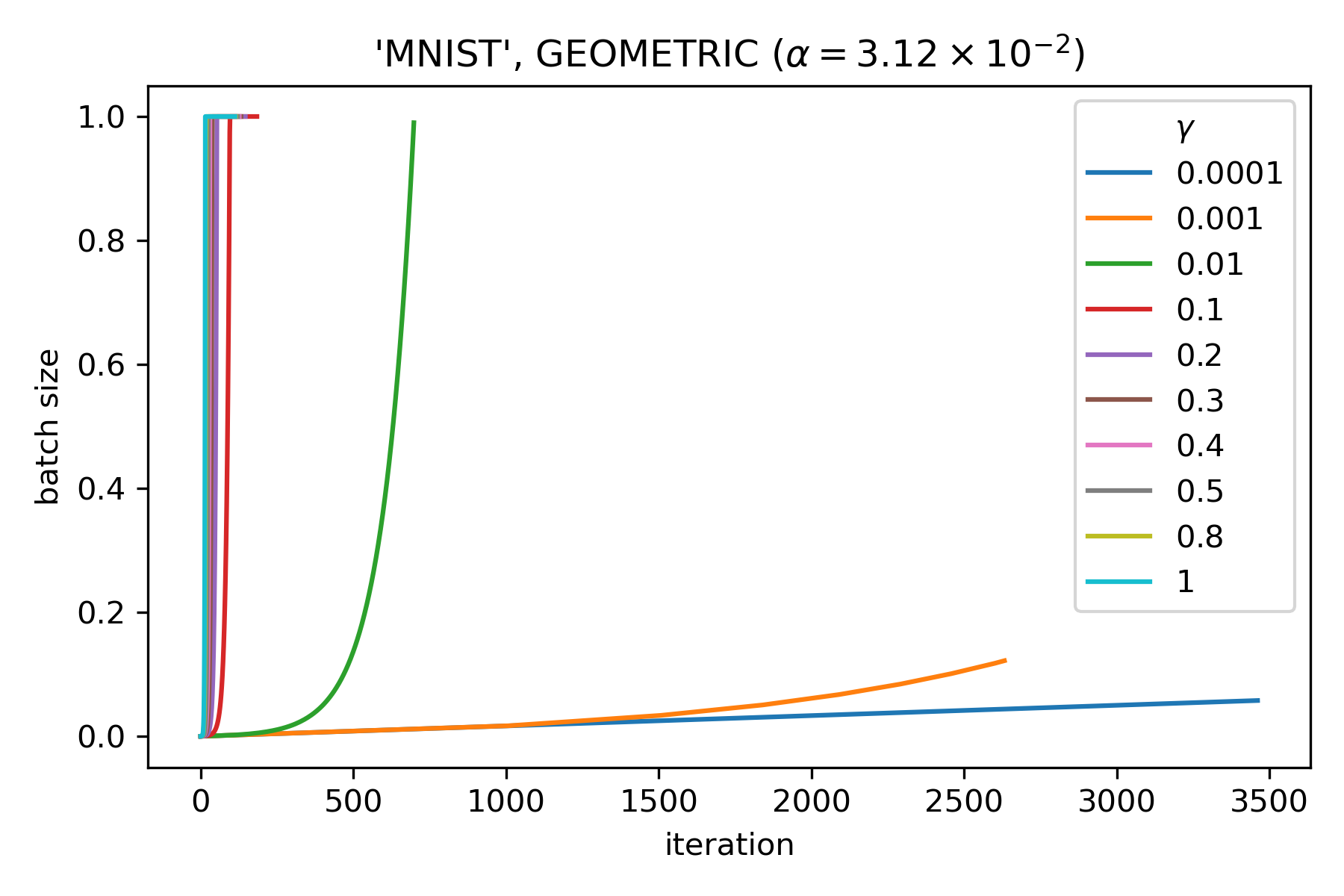}
	\includegraphics[width=0.45\linewidth]{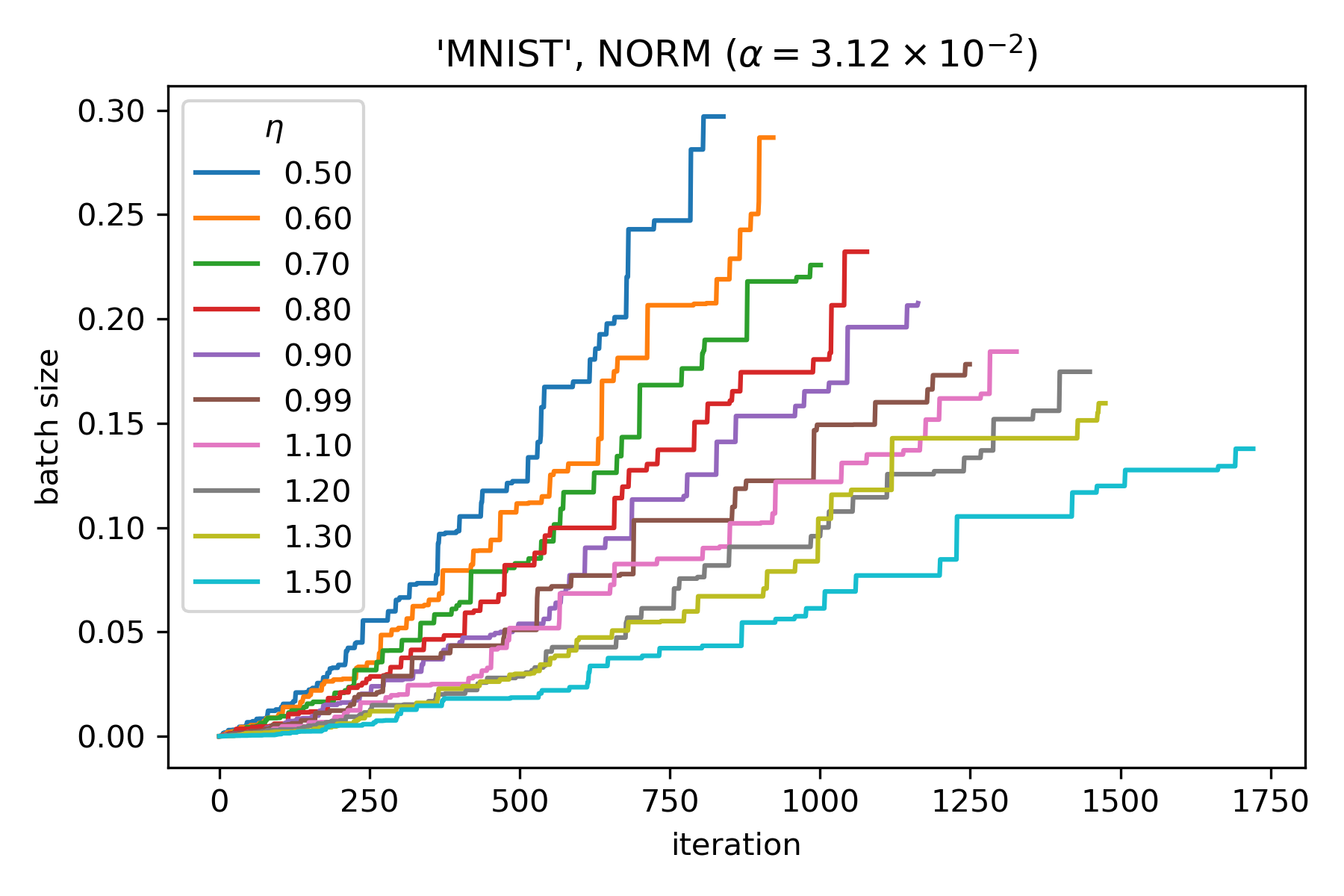} \\
	\includegraphics[width=0.45\linewidth]{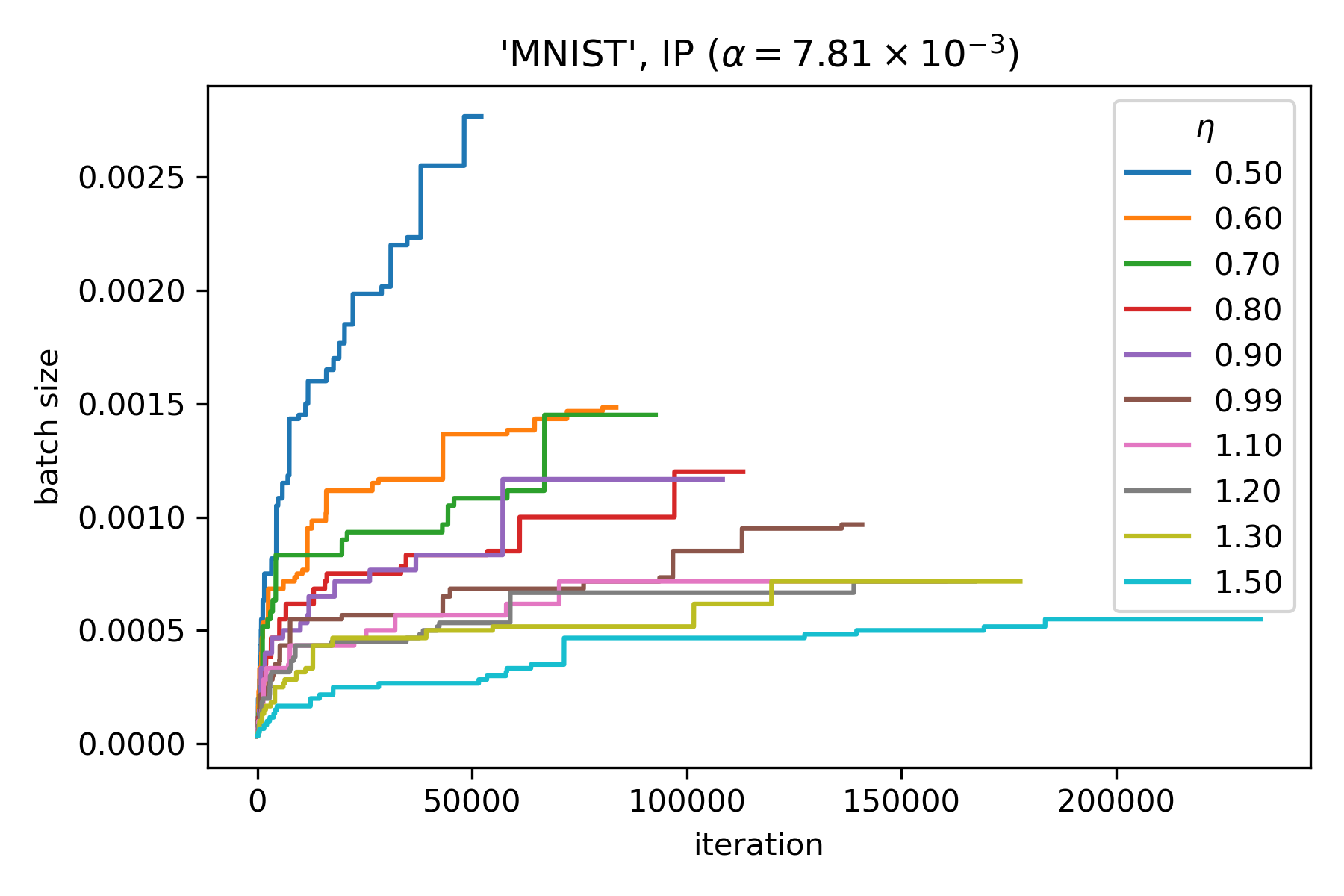}
    \includegraphics[width=0.45\linewidth]{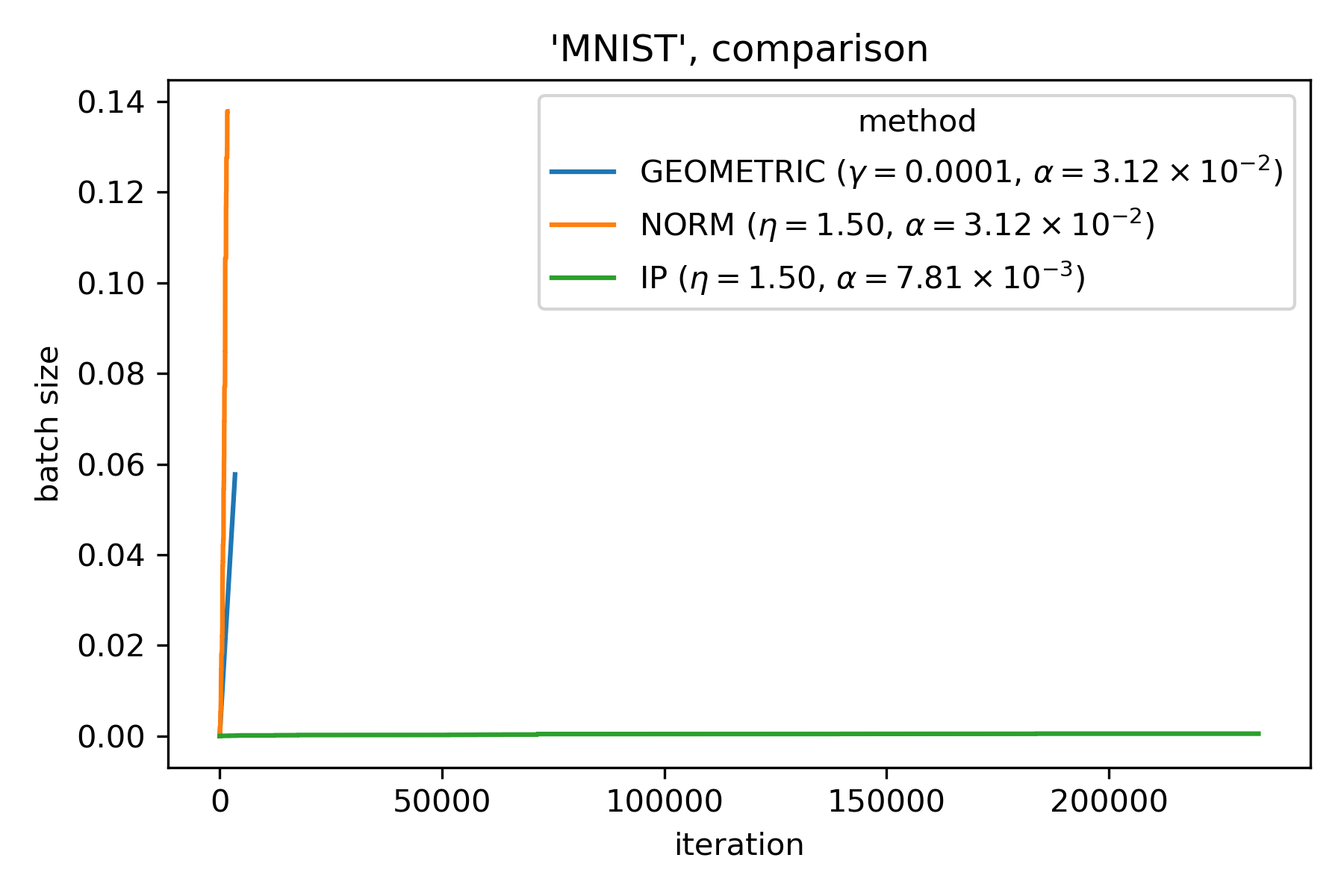} \\
    \caption{Batch size (as a fraction of total number of data points $N$) against iterations on dataset \texttt{MNIST}, with different strategies to control batch size: geometric increase (top left), norm test (top right), inner-product test (bottom left), and comparison between the best run for each method (bottom right).
    }
    \label{fig:MNIST_batchsizes}
\end{figure}

\begin{figure}[htp]
    \centering
    \includegraphics[width=0.45\linewidth]{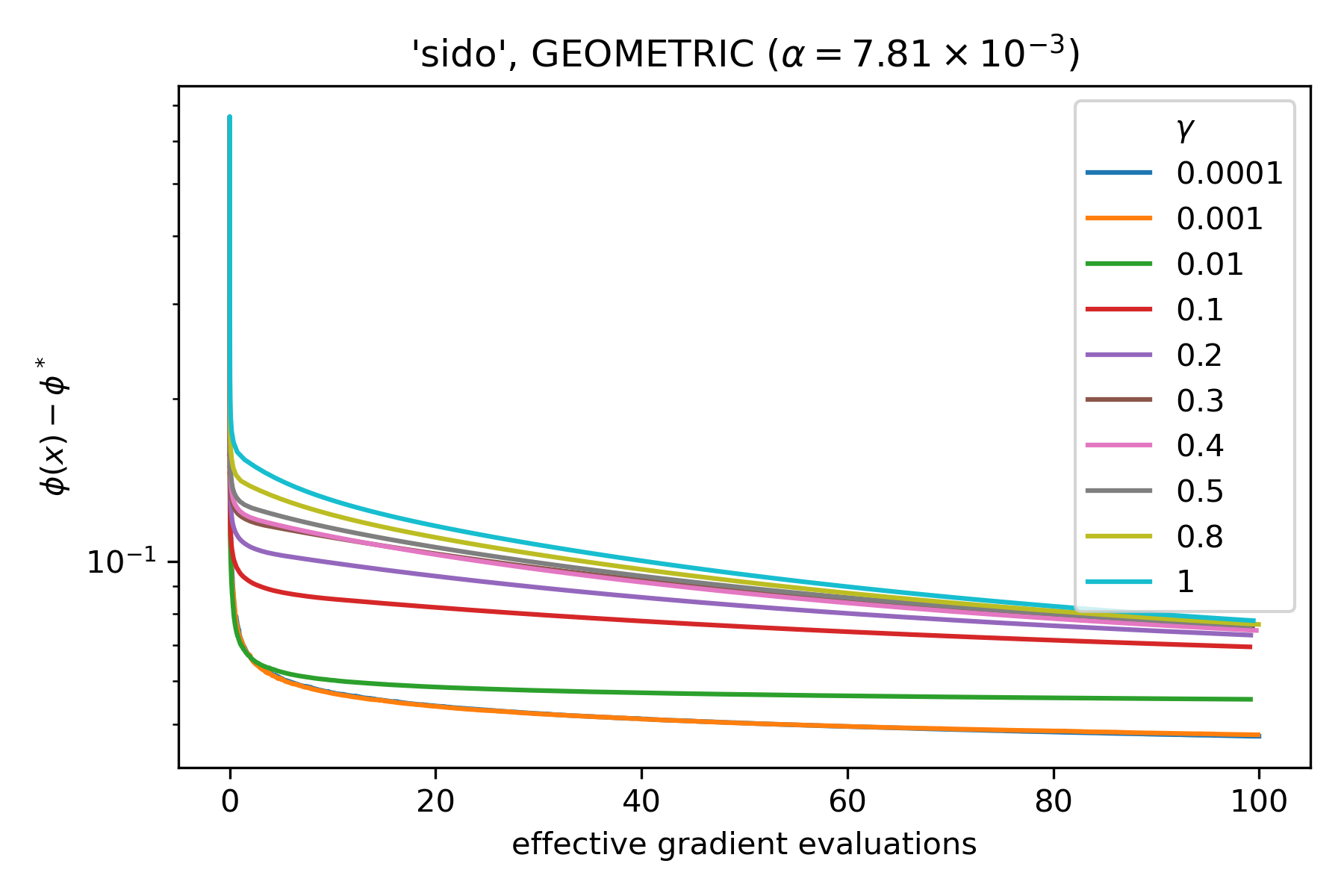}
	\includegraphics[width=0.45\linewidth]{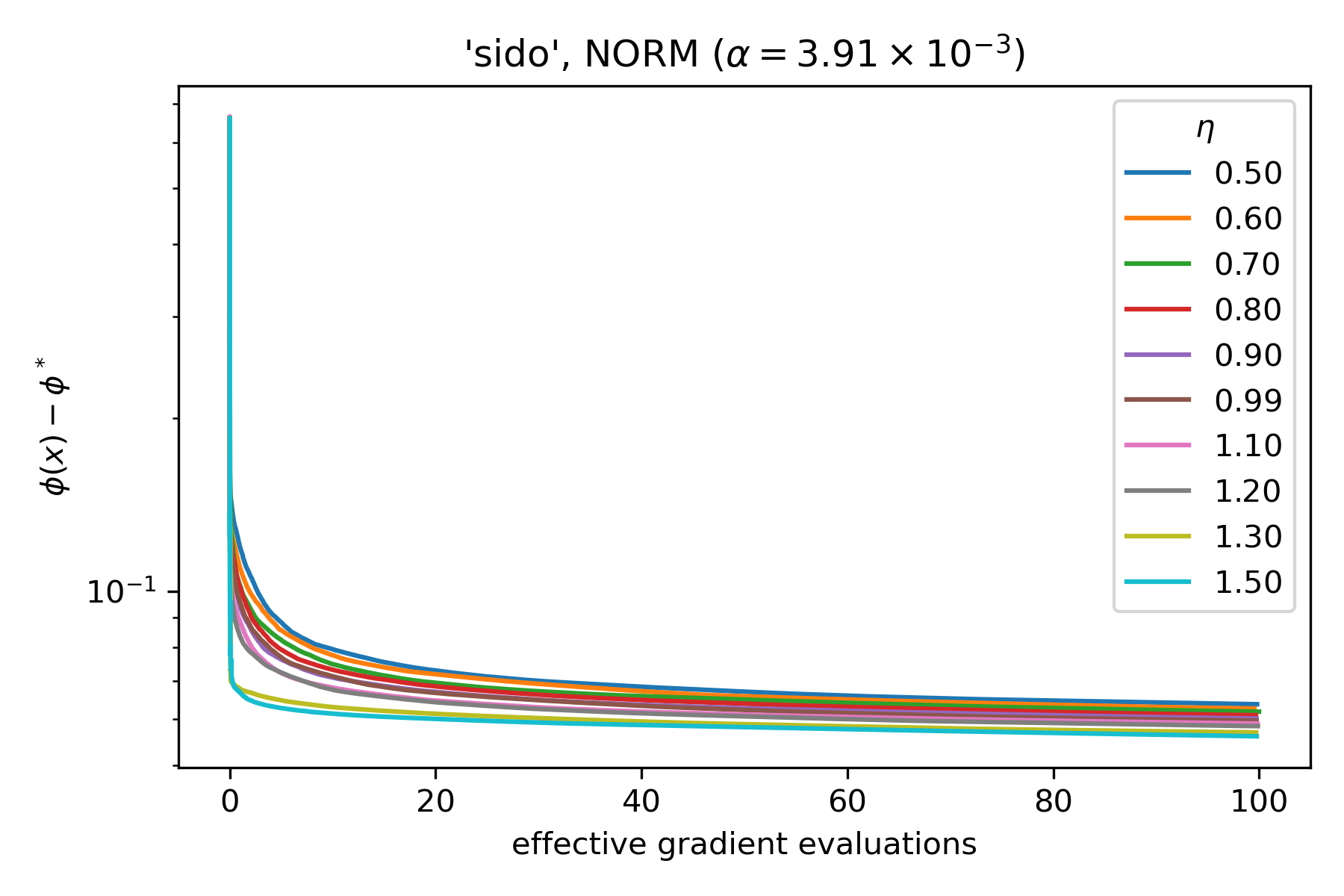} \\
	\includegraphics[width=0.45\linewidth]{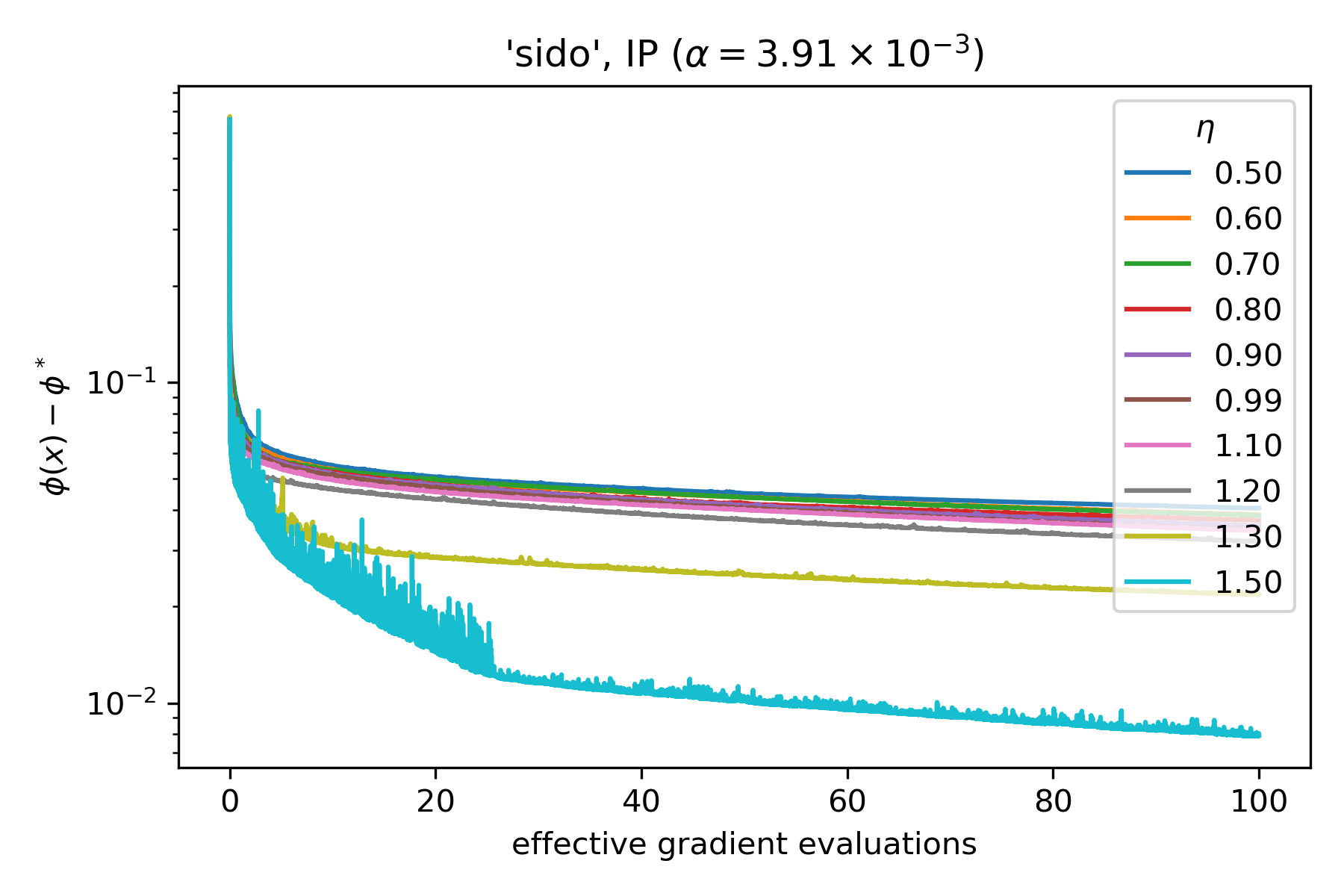}
    \includegraphics[width=0.45\linewidth]{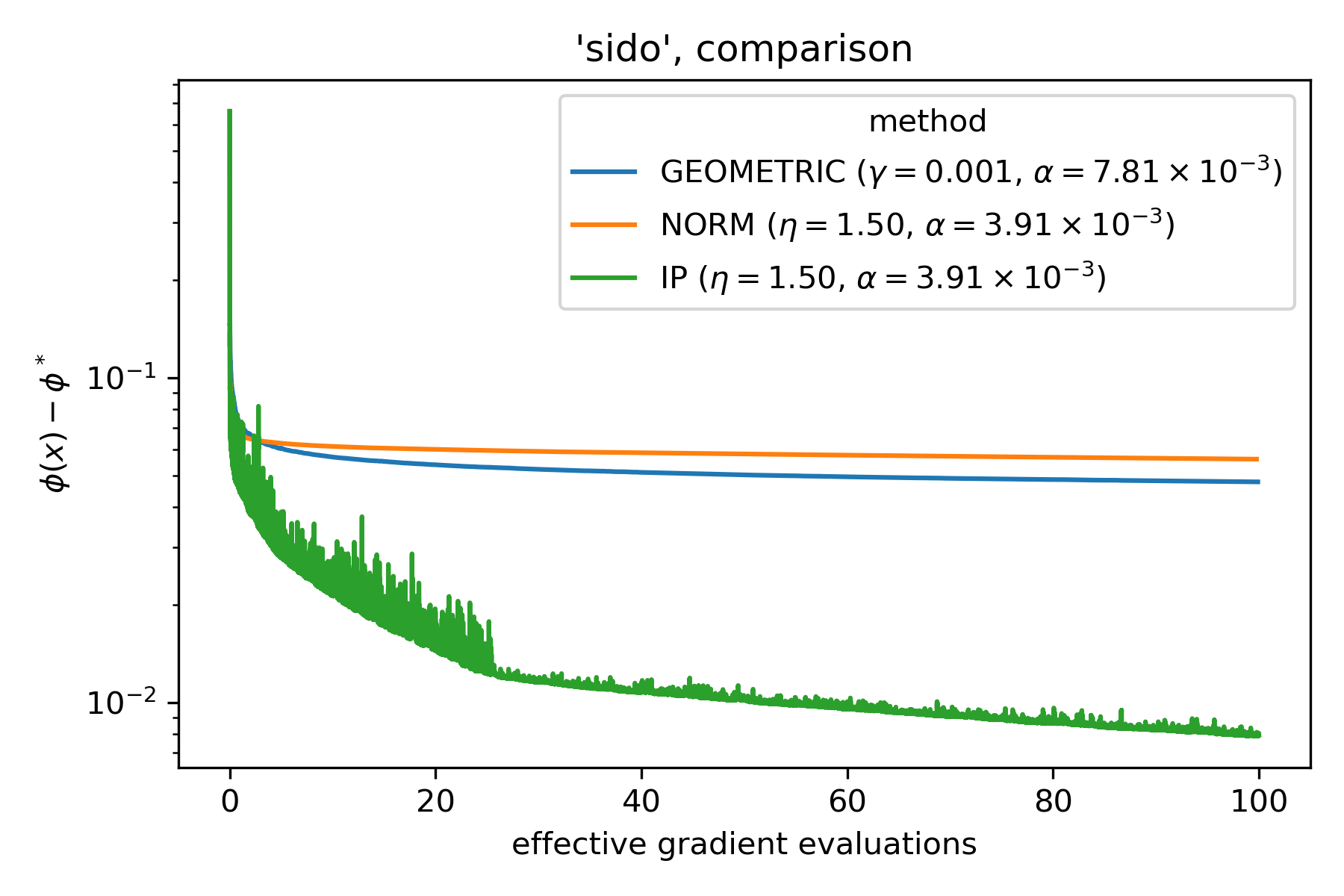} \\
    \caption{Optimality gap $\phi(x_k) - \phi^*$ against effective gradient evaluations on dataset \texttt{sido}, with different strategies to control batch size: geometric increase (top left), norm test (top right), inner-product test (bottom left), and comparison between the best run for each method (bottom right).
    }
    \label{fig:sido_func}
\end{figure}

\begin{figure}[htp]
    \centering
    \includegraphics[width=0.45\linewidth]{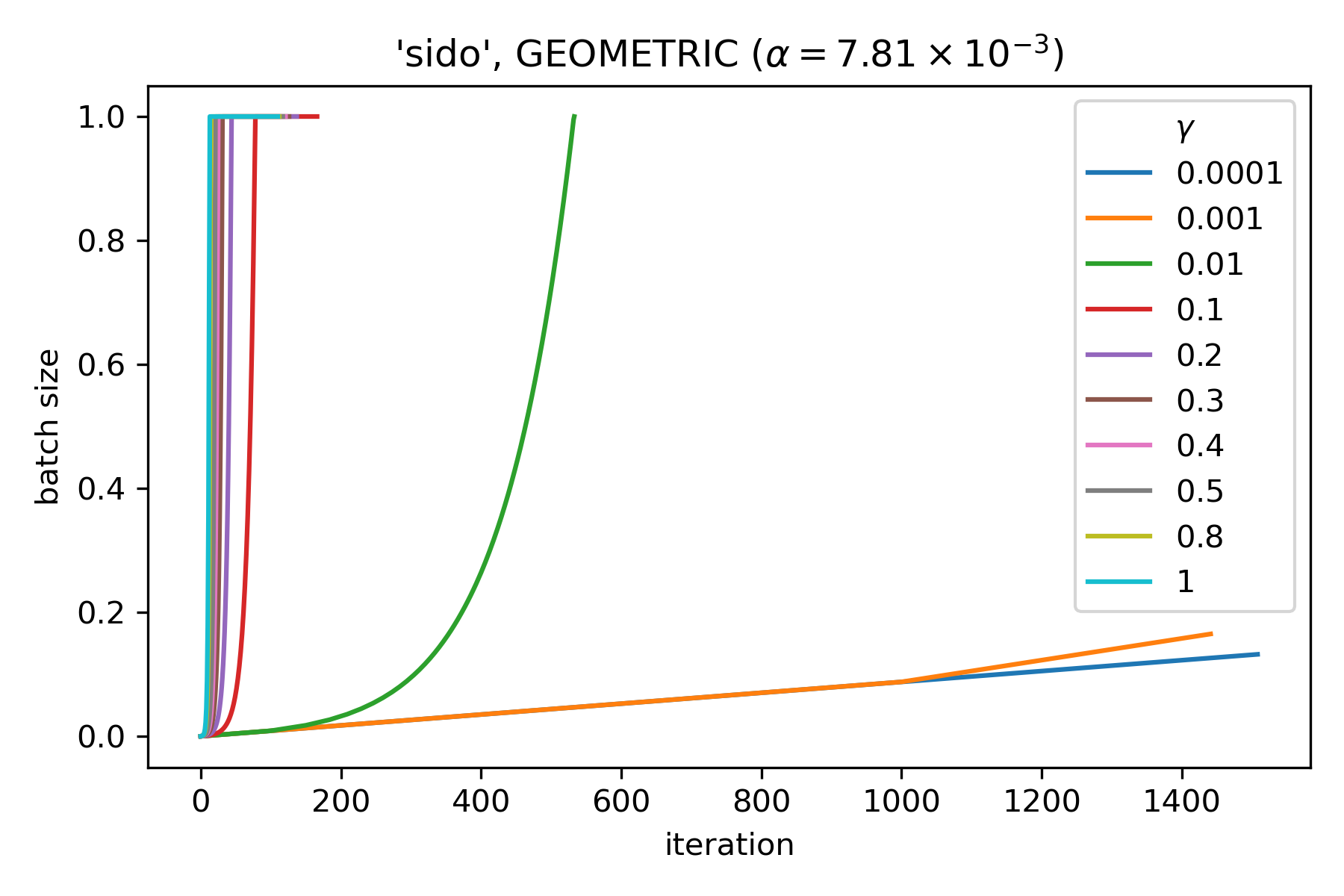}
	\includegraphics[width=0.45\linewidth]{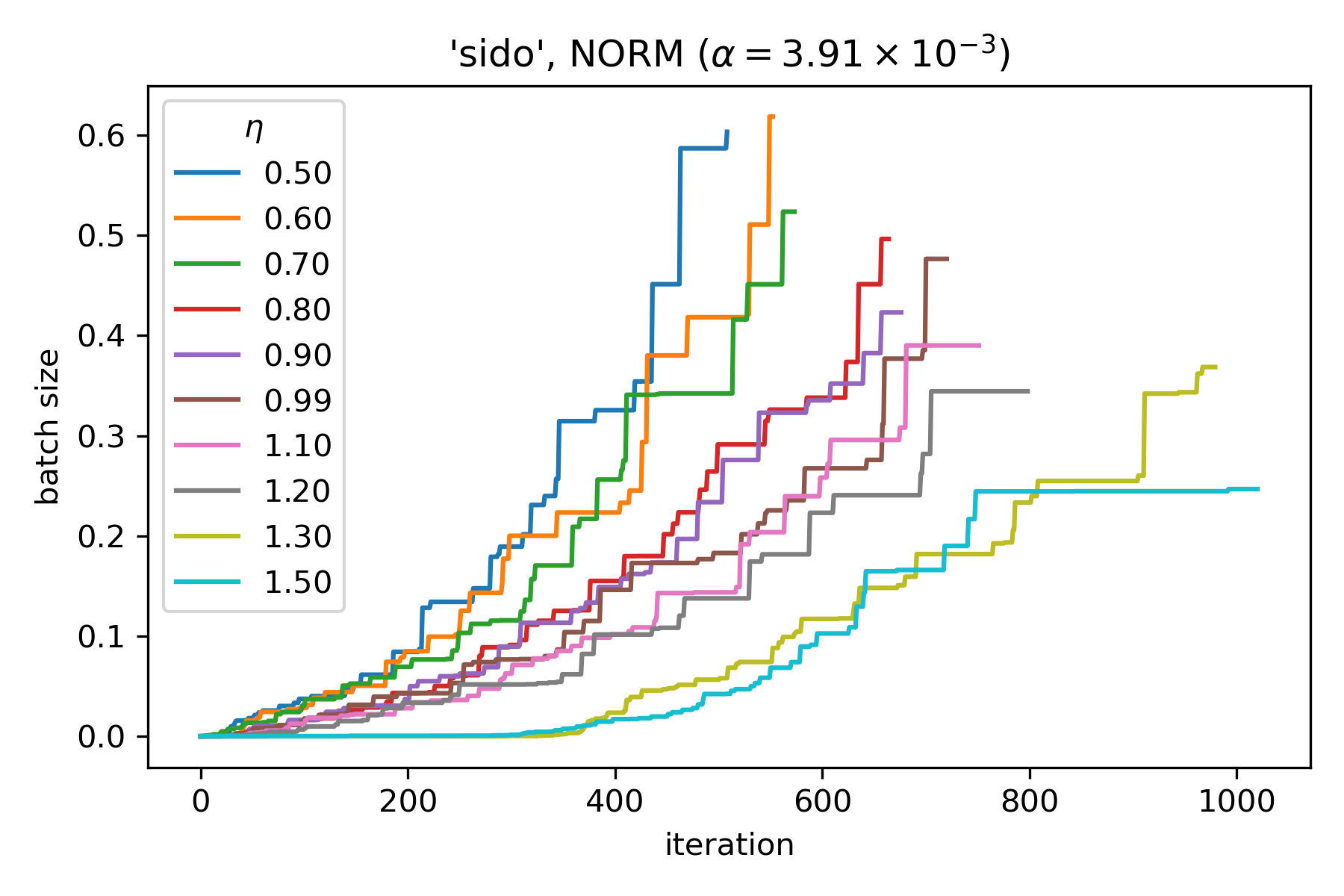} \\
	\includegraphics[width=0.45\linewidth]{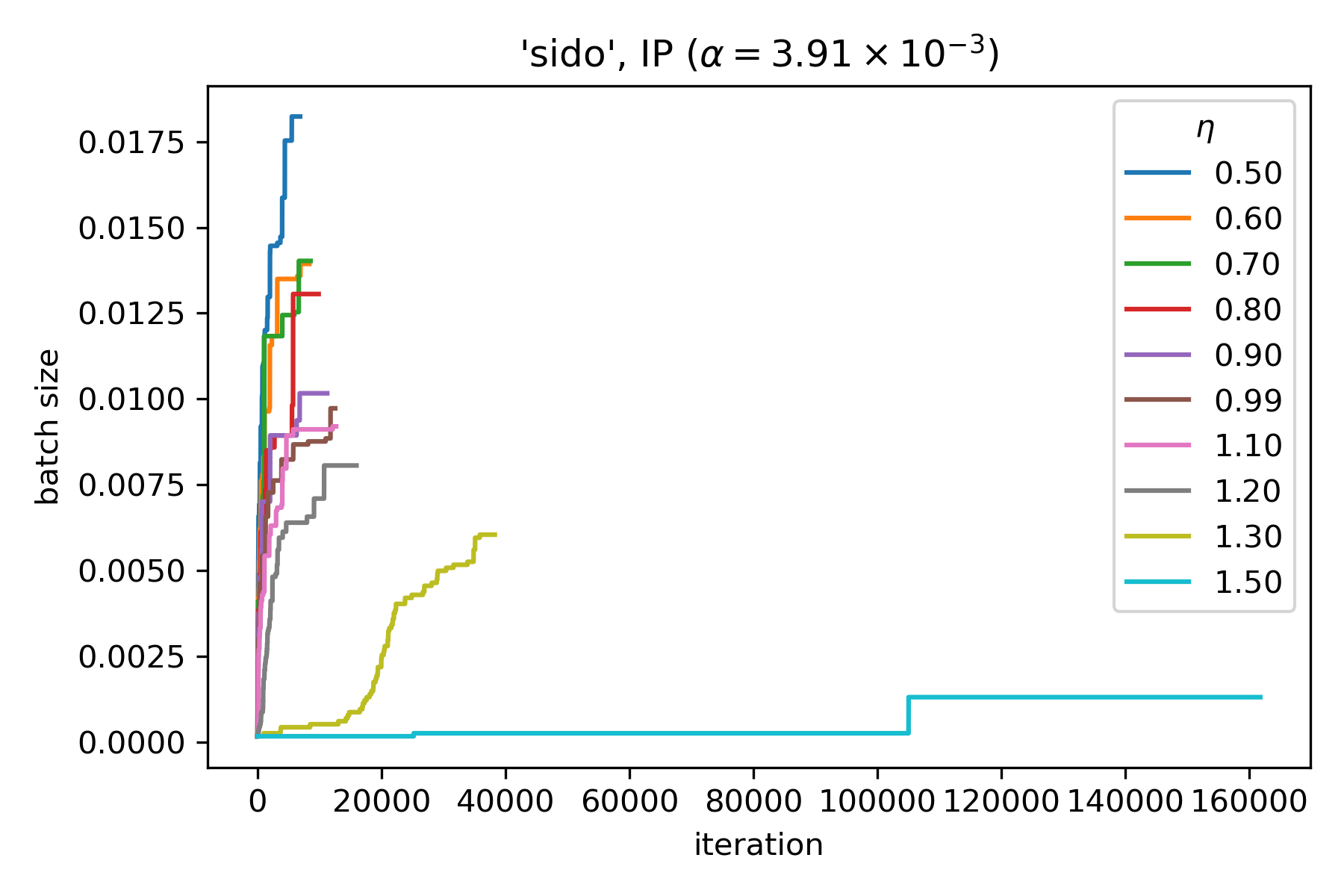}
    \includegraphics[width=0.45\linewidth]{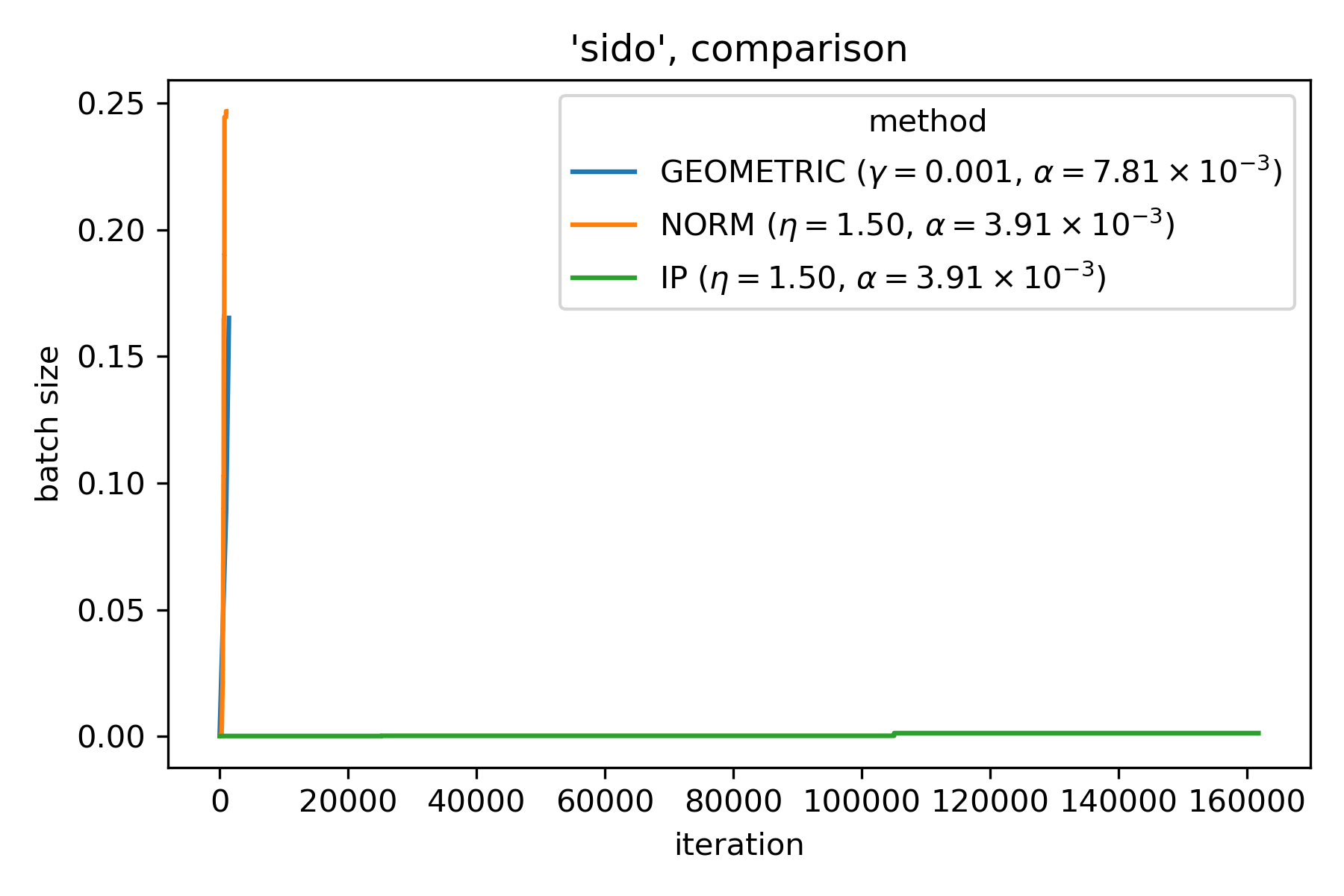} \\
    \caption{Batch size (as a fraction of total number of data points $N$) against iterations on dataset \texttt{sido}, with different strategies to control batch size: geometric increase (top left), norm test (top right), inner-product test (bottom left), and comparison between the best run for each method (bottom right).
    }
    \label{fig:sido_batchsizes}
\end{figure}

\end{document}